\documentclass[twoside,a4paper]{amsart}
\usepackage[english]{babel}
\usepackage[T1]{fontenc}
\usepackage{xcolor}
\usepackage{graphicx}
\usepackage{float}
\usepackage{amsfonts,amsmath,amsthm,amssymb,latexsym}
\usepackage{vmargin}
\usepackage{amscd}
\usepackage[all]{xy}
\usepackage[final]{showlabels} 
\usepackage{enumerate}
\usepackage{version}
\usepackage[normalem]{ulem}
\usepackage{enumitem}
\makeatletter
\def\namedlabel#1#2{\begingroup
    #2%
    \def\@currentlabel{#2}%
    \phantomsection\label{#1}\endgroup}
\makeatother

\newtheorem{thm}{Theorem}[section]
\newtheorem{cor}[thm]{Corollary}
\newtheorem{lem}[thm]{Lemma}
\newtheorem{prop}[thm]{Proposition}
\theoremstyle{definition}
\newtheorem{defn}[thm]{Definition}
\newtheorem{es}[thm]{Example}

\newtheorem{rmk}[thm]{Remark}

\newcommand{\id}{\mathrm{Id}}
\newcommand{\im}{\mathrm{Im}}

\newcommand{\Hom}{\mathrm{Hom}}

\newcommand{\cc}{\mathcal{C}}
\newcommand{\dd}{\mathcal{D}}
\newcommand{\e}{\mathcal{E}}
\newcommand{\f}{\mathcal{F}}

\newcommand{\p}{\mathcal{P}}

\newcommand{\bl}[1]{{\color{blue}{#1}}}
\newcommand{\rd}[1]{{\color{red}{#1}}}

\newcommand{\RR}{\mathbb{R}}

\newcommand{\Set}{{\sf Set}}
\newcommand{\Top}{{\sf Top}}

\newenvironment{invisible}{{\noindent\sc \colorbox{yellow}{Invisible:}\;}\color{gray}}{\medskip}
\excludeversion{invisible}

\begin{document}
\title{Semiseparable functors and conditions up to retracts}
\author{Alessandro Ardizzoni}
\email{alessandro.ardizzoni@unito.it}
\urladdr{\url{www.sites.google.com/site/aleardizzonihome}}

\author{Lucrezia Bottegoni}
\email{lucrezia.bottegoni@unito.it}

\address{%
\parbox[b]{\linewidth}{University of Turin, Department of Mathematics ``G. Peano'', via
Carlo Alberto 10, I-10123 Torino, Italy}}

\thanks{This paper was written while the authors were members of the
``National Group for Algebraic and Geometric Structures and their
Applications'' (GNSAGA-INdAM). They were partially supported by MUR
within the National Research Project PRIN 2017. The authors would like to express their gratitude to Fosco Loregian, Claudia Menini, Giuseppe Rosolini, Paolo Saracco, Enrico Vitale and Joost Vercruysse for meaningful comments on the topics treated.}

\begin{abstract}
In a previous paper we introduced the concept of semiseparable functor. Here we continue our study of these functors in connection with idempotent (Cauchy) completion. To this aim, we introduce and investigate the notions of (co)reflection and bireflection up to retracts. We show that the (co)comparison functor attached to an adjunction whose associated (co)monad is separable is a coreflection (reflection) up to retracts.
This fact allows us to prove that a right (left) adjoint functor is semiseparable if and only if the  associated (co)monad is separable and the (co)comparison functor is a bireflection up to retracts, extending a characterization pursued  by X.-W. Chen in the separable case. Finally, we provide a semi-analogue of a result obtained by P. Balmer in the framework of pre-triangulated categories.
\end{abstract}
\maketitle
\tableofcontents

\section*{Introduction}
The way a functor $F:\cc\to\dd$ acts on morphisms is encoded in the natural transformation $\f$ given on components by $\f_{X,Y} : \Hom_{\cc}(X,Y)\rightarrow \Hom_{\dd}(FX, FY),f\mapsto F(f)$, where $X$ and $Y$ are objects in $\cc$. In the literature, a functor is called \emph{separable} if there is  a natural transformation $\p$ such that $\p\circ\f=\id$ and \emph{naturally full} if one has $\f\circ\p=\id$ instead.
In \cite{AB22}, we introduced a weakening of both these notions, by naming \emph{semiseparable} a functor $F:\cc\to\dd$ such that $\f\circ\p\circ\f=\f$ for some $\p$.
Among other results, we obtained the following characterization:
Given a functor $G:\dd \to \cc$ with a left adjoint $F$, then $G$ is semiseparable if and only if the associated monad $GF$ is separable and the comparison functor $K_{GF}:\dd\to\cc_{CF}$ is naturally full. A closer inspection to the functor $K_{GF}$ in this setting reveals that it  satisfies the following extra properties
\begin{enumerate}
  \item[(\namedlabel{P1}{P1})] if it has a left adjoint, this is fully faithful; 
  \item[(\namedlabel{P2}{P2})] it has indeed a left adjoint if its source category is idempotent complete.
\end{enumerate}
The first problem we address in the present paper is to introduce a new type of functor, that we call \textbf{coreflection up to retracts}, that catches these two properties, and need to have neither an adjoint nor an idempotent complete source category a priori.  In order to give the rightful place to this notion, note that there are properties of a functor $F:\cc\to \dd$ that transfer to its (idempotent) completion $F^{\natural}:\cc^{\natural}\to \dd^{\natural}$ and vice versa (e.g. being either faithful, full, fully faithful, semiseparable, separable or naturally full, as we will see in Proposition \ref{prop:sscompletion} and Corollary \ref{cor:sep-nat-ff}). There are, however, other properties that do not share this behaviour. For instance, if $F$ is an equivalence of categories so is $F^{\natural}$ but the converse is not always true:
It is known that $F^{\natural}$ is an equivalence if and only if $F$ is fully faithful and surjective up to retracts, i.e. every $D\in\dd$ is a retract of $FC$ for some $C\in\cc$, and for this reason a functor $F$ such that $F^{\natural}$ is an equivalence is sometimes called an equivalence up to retracts in the literature. As we will see, something similar happens to a coreflection, i.e. a functor endowed with a fully faithful left adjoint: If $F$ is a coreflection so is $F^{\natural}$, but, again, the converse is not true in general. We are so prompted to define a coreflection up to retracts to be a functor $F$ whose completion $F^{\natural}$ is a coreflection.
It goes without saying that the functor $K_{GF}$ is expected  to be a coreflection  up to retracts in case $G$ is semiseparable. Since we noticed that $K_{GF}$ is also naturally full, and in \cite{AB22} we proved that a naturally full coreflection is the same as a bireflection, i.e. it has a left and right adjoint equal which is fully faithful and satisfies a suitable coherent condition,  we are also led to introduce the stronger notion of \textbf{bireflection up to retracts} which identifies a functor whose idempotent completion is a bireflection, and $K_{GF}$ is expected to be such a functor as well.
Luckily enough, in Proposition \ref{prop:corefl} and Proposition \ref{prop:idpcom-cutr} we are able to prove that each coreflection up to retracts (and a fortiori each bireflection up to retracts) verifies the properties \eqref{P1} and \eqref{P2} discussed above.

In order to go deeper into the properties of these functors, we have to deal with semifunctors, a notion studied by S. Hayashi in connection with $\lambda$-calculus, see \cite{Ha85}. A semifunctor is defined the same way as a functor, except that it needs not to preserve identities, and there is also a proper notion of semiadjunction for semifunctors. We show how to construct a semiadjunction out of a right (left) semiadjoint in the sense of \cite{MW13}.
These tools permit to pursue a characterization of (co)reflections up to retracts as part of suitable semiadjunctions, see Corollary \ref{cor:charactutr}, and to provide sufficient conditions guaranteeing that a functor is a (co)reflection up to retracts, see Proposition \ref{prop:trinat}.

Now, given a category $\cc$ and an idempotent natural transformation $e:\id_\cc\to\id_\cc$ one can consider the coidentifier category $\cc_e$ which is a suitable quotient category. In Theorem \ref{thm:H-corefl-utr} we prove that the quotient functor $H:\cc\to\cc_e$ is another instance of coreflection up to retracts, in fact a bireflection up to retracts, by means of the aforementioned characterization employing semifunctors (it is noteworthy that this functor is a bireflection if and only if $e$ splits, see Remark \ref{rmk:Hnotbiref}). Through the same characterization, exceeding the initial expectations, in Theorem \ref{thm:monutr} we find out that the (co)comparison functor attached to an adjunction whose associated (co)monad is (co)separable is always a coreflection (reflection) up to retracts.
This result allows us to obtain in Theorem \ref{thm:computr} the following semi-analogue of  \cite[Proposition 3.5]{Chen15} proved by X.-W. Chen: Given a functor $G:\dd \to \cc$ with a left adjoint $F$, then $G$ is semiseparable if and only if the associated monad $GF$ is separable and the comparison functor $K_{GF}:\dd\to\cc_{CF}$ is a bireflection up to retracts. It is well-known that $GF$ is separable if and only if the forgetful functor $U_{GF}:\cc_{GF}\to\cc$ is a separable functor so that we get the factorization $\xymatrix{\dd\ar[r]^{ K_{GF}}& {\cc_{GF}}\ar[r]^{ U_{GF}}& \cc}$ of $G$ as a bireflection up to retracts followed by a separable functor. In \cite{AB22} we proved that when $G$ is semiseparable we can associate to it an invariant that we called the associated idempotent natural transformation $e:\id_\dd\to\id_\dd$ and that $G$ admits a factorization of the form $\xymatrix{\dd\ar[r]^{{H}}& {\dd_e}\ar[r]^{{G_e}}& \cc}$ where $G_e$ is separable and $H$ is the quotient functor, that, by the foregoing, is a bireflection up to retracts. Summing up we have two factorizations of the same type and it is then natural to wonder how they are related. In Proposition \ref{prop:CoidEil}, we prove there is an equivalence up to retracts $(K_{GF})_e:\dd_e\to\cc_{GF}$ such that $\left( K_{GF}\right) _{e}\circ H=K_{GF}$ and $U_{GF}\circ \left(K_{GF}\right) _{e}=G_{e}$.  As a consequence of this result, in Proposition \ref{prop:KL-HL-equiv-utr}, we show that when $G$ is semiseparable the idempotent completions of the Kleisli category associated to the monad $GF$, of the coidentifier $\dd_e$ and of the Eilenberg-Moore category $\cc_{GF}$ are equivalent categories.

As an application of our results, we achieve for semiseparable functors in the context of pre-triangulated categories an analogue of P. Balmer's \cite[Theorem 4.1]{Balm11}. More explicitly, we introduce the notion of \textbf{stably semiseparable} functor by adapting the one of stably separable functor given in \cite[Definition 3.7]{Balm11}. Then Theorem \ref{thm:4.1} shows how, given a stably semiseparable right adjoint $G:\dd\to\cc$ with associated idempotent natural transformation $e$, under the relevant assumptions, we can transfer the pre-triangulation from $\cc$ to the coidentifier category  $\dd_e$. We  point out that the original result of Balmer requires $G$ to be stably separable and induces a pre-triangulation on $\dd$ rather than $\dd_e$. Finally, we provide conditions for the Eilenberg-Moore category $\cc_{GF}$ to inherit the pre-triangulation from the base category $\cc$, see Corollary \ref{coro:EMpretriang}.

  \medskip 
  \emph{Organization of the paper.} In Section \ref{sec:background} we recall the known results on semiseparable functors we will use. Section \ref{sec:idempcmp} deals with results involving the idempotent completion. We study how the notions of faithful, full, fully faithful, semiseparable, separable or naturally full functor behave with respect to idempotent completion. Then we introduce and investigate (co)reflections up to retracts and bireflections up to retracts.  We consider semifunctors and semiadjunctions as a tool to provide a characterization of (co)reflections up to retracts. We show that a (co)reflection up to retracts comes out to be always surjective up to retracts and we give sufficient conditions guaranteeing that a functor is a (co)reflection up to retracts.
  
  Section \ref{se:quotcom} collects the fall-outs of the results we achieved so far. First we prove that the quotient functor onto the coidentifier category is a coreflection up to retracts and that so is the comparison functor attached to an adjunction whose associated monad is separable. A similar result is obtained for the cocomparison functor in case the associated comonad is coseparable. These facts allow us to characterize a semiseparable right (left) adjoint in terms of (co)separability of the associated (co)monad and the requirement that the (co)comparison functor is a bireflection up to retracts.
We prove that two canonical factorizations attached to a semiseparable right adjoint functor, namely the one through the coidentifier category  and the one through the comparison functor, are the same up to an equivalence up to retracts. Then we relate the idempotent completions of the Kleisli category and Eilenberg-Moore category attached to a separable monad and, in case this monad is induced by an adjunction with semiseparable right adjoint, the idempotent completion of the coidentifier category is  added to the picture.
Finally, we show an analogue for semiseparable functors of a result obtained by P. Balmer in the framework of pre-triangulated categories.
\medskip 

\emph{Notations.}
Given an object $X$ in a category $\cc$, the identity morphism on $X$ will be denoted either by $\id_X$ or $X$ for short. For categories $\cc$ and $\dd$, a functor $F:\cc\to \dd$ just means a covariant functor. By $\id_{\cc}$ we denote the identity functor on $\cc$. For any functor $F:\cc\to \dd$, we denote $\id_{F}:F\to F$ (or just $F$, if there is no danger of confusion) the natural transformation defined by $(\id_{F})_X:=\id_{FX}$. By a ring we mean a unital associative ring.

\section{Background on semiseparability}\label{sec:background}
In this section we recall from \cite{AB22} some results on semiseparable functors we need. In particular, in Subsection \ref{sub:seminat} we provide a characterization of separable and naturally full functors in terms of semiseparable functors and we explain the behaviour of semiseparable functors with respect to composition.
Subsection \ref{sub:coidentifier} deals with the idempotent natural transformation associated to a semiseparable functor, that measures its distance from being separable. Then we discuss the existence of a canonical factorization of a semiseparable functor through the coidentifier category attached to this idempotent. 
Subsection \ref{sub:Eil-Kleisli} concerns a characterization of semiseparable functors having an adjoint in terms of properties of the attached (co)monad and (co)comparison functor. In Subsection \ref{sub:corefbiref} we explore  the connection with (co)reflections and bireflections.

\subsection{(Semi)separability and natural fullness}\label{sub:seminat}
Let $F: \cc \rightarrow \dd$ be a functor and consider the natural transformation
\begin{equation*}\label{nat_transf}
\f : \Hom_{\cc}(-,-)\rightarrow \Hom_{\dd}(F-, F-),
\end{equation*}
defined by setting $\f_{C,C'}(f)= F(f)$, for any $f:C\rightarrow C'$ in $\cc$.

If there is a natural transformation $\p : \Hom_{\dd}(F-, F-)\rightarrow \Hom_{\cc}(-,-)$ such that
\begin{itemize}
  \item $\p\circ\f = \id $, then $F$ is called \emph{separable} \cite{NVV89};
  \item $\f\circ\p = \id $, then $F$ is called \emph{naturally full} \cite{ACMM06};
  \item $\f\circ\p\circ\f = \f $, then $F$ is called \emph{semiseparable} \cite{AB22}.
\end{itemize}
We will write $\f^F$, $\p^F$ when needed to stress the dependence on the functor $F$ we are considering. The following results compares the notions of separable, naturally full and semiseparable functor.

\begin{prop}\label{prop:sep} \cite[Proposition 1.3]{AB22}
Let $F: \cc \rightarrow \dd$ be a functor. Then,
\begin{itemize}
\item[(i)]$F$ is separable if and only if $F$ is semiseparable and faithful;
\item[(ii)]$F$ is naturally full if and only if $F$ is semiseparable and full.
\end{itemize}
\end{prop}

It is well-known that if $F:\cc\to\dd$ and $G:\dd\to\e$ are separable functors so is their composition $G\circ F$ and, the other way around, if the composition $G\circ F$ is separable so is $F$, see \cite[Lemma 1.1]{NVV89}. A similar result with some difference, holds for naturally full functors, see \cite[Proposition 2.3]{ACMM06}. The following result concerns the behaviour of semiseparability with respect to composition. It is proved in \cite[Lemma 1.12 and Lemma 1.13]{AB22}.

\begin{lem}\label{lem:comp}
Let $F: \cc \rightarrow \dd$ and $G:\dd\rightarrow\e$ be functors and consider the composite $G\circ F:\cc\rightarrow \e$.
\begin{itemize}
\item[(i)] If $F$ is semiseparable and $G$ is separable, then $G\circ F$ is semiseparable.
\item[(ii)] If $F$ is naturally full and $G$ is semiseparable, then $G\circ F$ is semiseparable.
\item[(iii)] If $G\circ F$ is semiseparable and $G$ is faithful, then $F$ is semiseparable.
\end{itemize}
\end{lem}

\subsection{The associated idempotent and the coidentifier}\label{sub:coidentifier}
Recall that an endomorphism $e_X:X\to X$ in a category $\cc$ is \emph{idempotent} if $e_X^2=e_X$. A natural transformation $e:\id_{\cc%
	}\rightarrow \id_{\cc}$ is idempotent if the component $e_X:X\to X$ in $\cc$ is idempotent for all $X\in\cc$. The following result uniquely attaches an idempotent natural transformation to a given semiseparable functor.

\begin{prop}\label{prop:idempotent} \cite[Proposition 1.4]{AB22}
 Let $F:\cc\rightarrow \dd$ be a semiseparable functor. Then
there is a unique idempotent natural transformation $e:\id_{\cc%
}\rightarrow \id_{\cc}$ such that $Fe=\id_{F}$ with the following universal property: if $f,g:A\to B$ are morphisms, then $Ff=Fg$ if and only if $e_B\circ f=e_B\circ g$.
 \end{prop}

  The idempotent natural transformation $e:\id_\cc\to\id_\cc$ we have attached to a semiseparable functor  $F:\cc\to\dd$ in Proposition \ref{prop:idempotent} 
will be called the \textbf{associated idempotent natural transformation}. Explicitly, $e$ is defined on components by $e_{X}:=\mathcal{P}_{X,X}\left( \id_{FX}\right) $ where $\mathcal{P}$ is any natural transformation such that $\mathcal{F}\circ \mathcal{P}\circ \mathcal{F}=\mathcal{F}$.
It controls the separability of $F$ as follows.

\begin{cor}\label{cor:esep}\cite[Corollary 1.7]{AB22}
Let $F:\cc\to\dd$ be a semiseparable functor and let $e:\id_\cc\to\id_\cc$ be the associated idempotent natural transformation. Then $F$ is separable if and only if $e=\id$.
\end{cor}

\begin{rmk}\label{rmk:idemp}
	Let $F:\cc\to\dd$, $G:\dd\to\e$ be functors. By Lemma \ref{lem:comp} we know that $G\circ F$ is semiseparable in both cases (i) and (ii). Then, in (ii) the idempotent natural transformation associated to $GF$ is given by $$e^{GF}_X=\p^{GF}_{X,X}(\id_{GFX})=\p^F_{X,X}\p^G_{FX,FX}(\id_{GFX})=\p^F_{X,X}(e^G_{FX}),$$ where $e^G:\id_\dd\to\id_\dd$ is the one associated to the semiseparable functor $G$. In particular, if $G$ is further separable as in (i), by Corollary \ref{cor:esep} the idempotent natural transformation associated to $GF$ is given by $e^{GF}_X=\p^{F}_{X,X}(e^G_{FX})=\p^{F}_{X,X}(\id_{FX})=e^F_X$, where $e^F:\id_\cc\to\id_\cc$ is the associated idempotent to the semiseparable functor $F$.  
\end{rmk}

Given a category $\cc$ and an idempotent natural transformation $e:%
\id_{\cc}\rightarrow \id_{\cc}$, the coidentifier $\cc_{e}$, see \cite[Example 17]{FOPTST99}, is the quotient category $\cc/\!\sim$ of $\cc$ where $\sim$ is the congruence relation on the hom-sets defined, for all $f,g:A\to B$, by setting $f\sim g$ if and only if $e_{B}\circ f=e_{B}\circ g$. Thus $\mathrm{Ob}\left( \cc%
_{e}\right) =\mathrm{Ob}\left( \cc\right) $ and $\Hom_{%
\cc_{e}}\left( A,B\right) =\Hom_{\cc}\left(
A,B\right) /\!\sim $. We denote by $\overline{f}$ the class of $f\in \Hom_{\cc%
}\left( A,B\right) $ in $\Hom_{\cc_{e}}\left( A,B\right) $.
It is remarkable that  the quotient functor $H:\cc\rightarrow \cc_{e}$,
acting as the identity on objects and as the canonical projection on
morphisms, is naturally full with respect to $\mathcal{P}_{A,B}:\Hom_{%
\cc_{e}}\left( A,B\right)\to \Hom_{\cc}\left(
A,B\right)$ defined by  $\mathcal{P}_{A,B}(\overline{f})=e_B\circ f$. Moreover the idempotent natural transformation associated to $H$ is exactly $e$.\medskip

Next lemma is essentially the universal property of the coidentifier that can be deduced from the
dual version of \cite[Definition 14(1)]{FOPTST99}, see also \cite[Lemma 1.14(1)]{AB22}.

\begin{lem}\label{lem:coidentifier}
 Let $\cc$ be a category, let $e:\id_{\cc}\rightarrow \id_{\cc}$ be an idempotent natural transformation and let $H:\cc\rightarrow \cc_{e}$ be the quotient functor. A functor $F:\cc\to \dd$  satisfies $Fe=\id_{F}$ if and only if there is a functor $F_{e}:\cc_{e}\rightarrow \dd$ (necessarily unique) such that $F=F_{e}\circ H$. Given $F,F':\cc\to \dd$  such that $Fe=\id_{F}$ and $F'e=\id_{F'}$, and a natural transformation $\beta:F\to F'$, there is a unique natural transformation $\beta_e:F_e\to F'_e$ such that $\beta=\beta_eH$.
\end{lem}

The following result shows that each semiseparable functor factors, through the coidentifier category, as a naturally full functor followed by a separable one.

\begin{thm}\label{thm:coidentifier}\cite[Theorem 1.15]{AB22}
Let $F:\cc\rightarrow \dd$ be a semiseparable functor and let $e:\id_\cc\to\id_\cc$ be the associated idempotent natural transformation. Then, there is a unique functor $F_{e}:\cc%
_{e}\rightarrow \dd$ (necessarily separable) such that $F=F_{e}\circ H$ where $H:\cc\rightarrow \cc_{e}$ is the quotient functor.
Furthermore if $F$ also factors as $S\circ N$ where $S:\mathcal{E}\rightarrow \dd$ is a separable functor and $N:\cc\rightarrow \mathcal{E}$ is a naturally full functor, then there is a unique functor $N_e:\cc_e\to \mathcal{E}$ (necessarily fully faithful) such that $N_e\circ H=N$ and $S\circ N_e=F_e$, and $e$ is also the idempotent natural transformation associated to $N$ (by Remark \ref{rmk:idemp}).
\begin{gather*}
 \xymatrix{\cc\ar[rd]_F\ar[r]^H&\cc_e \ar@{.>}[d]^{F_e}\\&\dd  }\qquad\qquad
  \xymatrix{\cc\ar[r]^H\ar[d]_{N}&\cc_e \ar@{.>}[dl]|{N_e}\ar[d]^{F_e}\\\mathcal{E}\ar[r]^S&\dd  }
  \end{gather*}
The natural transformation making $F_e$ separable is uniquely determined by the equality $\p^{F_e}_{HX,HY}:=\f^H_{X,Y}\circ\p^F_{X,Y}$, where $\p^F_{X,Y}$ is the one making $F$ semiseparable, for all $X,Y$ in $\cc$.
\end{thm}

\subsection{Eilenberg-Moore category}\label{sub:Eil-Kleisli}
In order to present the behaviour of semiseparable adjoint functors in terms of separable (co)monads and associated (co)comparison functor, we remind some notions concerning Eilenberg-Moore categories \cite{EM65}.

Given a monad $(\top,m:\top\top\to \top,\eta:\id_\cc\to \top)$ on a category $\cc$ we denote by $\cc_{\top}$ the Eilenberg-Moore category of modules (or algebras) over it. The forgetful functor $U_\top:\cc_\top\to \cc$ has a left adjoint, namely the \emph{free functor}
$$V_\top:\cc\to \cc_\top,\qquad C\mapsto (\top C,m _C),\qquad f\mapsto \top (f).$$
The unit $\id_\cc\to U_\top V_\top=\top$ is exactly $\eta$ while the counit $\beta:V_\top U_\top\to \id_{\cc_\top}$ is completely determined by the equality by $U_\top\beta_{(X,\mu)}=\mu$ for every object $(X,\mu)$ in $\cc_\top$ (see \cite[Proposition 4.1.4]{BorII94}). Dually, given a comonad $(\bot,\Delta:\bot\to \bot\bot,\epsilon:\bot\to \id_\cc)$ on a category $\cc$ we denote by $\cc^{\bot}$ the Eilenberg-Moore category of comodules (or coalgebras) over it. The forgetful functor $U^\bot:\cc^\bot\to \cc$ has a right adjoint, namely the \emph{cofree functor} $$V^\bot:\cc\to \cc^\bot,\qquad C\mapsto (\bot C,\Delta _C),\qquad f\mapsto \bot (f).$$
The unit $\alpha:\id_{\cc^\bot}\to V^\bot U^\bot $ is completely determined by the equality $U^\bot\alpha_{(X,\rho)}=\rho$ for every object $(X,\rho)$ in $\cc^\bot$ while the counit $U^\bot V^\bot =\bot\to \id_{\cc}$ is exactly $\epsilon$.\medskip

Given an adjunction  $F\dashv G:\dd\to\cc$, with unit $\eta$ and counit $\epsilon$, we can consider the monad $(GF, G\epsilon F, \eta )$ and the comonad $(FG, F\eta G, \epsilon )$.
We have the \emph{comparison functor}
$$K_{GF}:\dd\to \cc_{GF},\qquad D\mapsto (GD,G\epsilon_{D}),\qquad f\mapsto G (f)$$
and the \emph{cocomparison functor}
$$K^{FG}:\cc\to \dd^{FG},\qquad C\mapsto (FC,F\eta_{C}),\qquad f\mapsto F (f).$$
Thus we have the following diagram
\begin{gather*}\label{diag:eil-moore}
\vcenter{%
\xymatrixcolsep{1.8cm}\xymatrixrowsep{1cm}\xymatrix{
\dd^{FG} \ar@{}[r]|-\perp \ar@<1ex>[r]^-{U^{FG}}&\dd\ar@/^1pc/[rd]^{K_{GF}} \ar@<1ex>[l]^-{V^{FG}}
\ar@<1ex>[d]^*-<0.1cm>{^{G}}\\
&\cc \ar@/^1pc/[lu]^{K^{FG}}  \ar@<1ex>[u]^*-<0.1cm>{^{F}}\ar@{}[u]|{\dashv}\ar@{}[r]|-\perp \ar@<1ex>[r]^-{V_{GF}}&\cc _{GF}\ar@<1ex>[l]^-{U_{GF}}
} }
\end{gather*} where $U_{GF}\circ K_{GF}=G$, $K_{GF}\circ F = V_{GF}$, $U^{FG}\circ K^{FG}=F$ and $K^{FG}\circ G = V^{FG}$.

We recall that a monad $(\top,m:\top\top\to \top,\eta:\id_\cc\to \top)$ on a category $\cc$ is said to be \emph{separable} \cite{BruV07} if there exists a natural transformation $\sigma :\top\to\top\top$ such that $m\circ\sigma =\id _{\top}$ and $\top m\circ\sigma \top = \sigma\circ m = m\top\circ \top\sigma$; in particular, a separable monad is a monad satisfying the equivalent conditions of \cite[Proposition 6.3]{BruV07}. Dually, a comonad $(\bot,\Delta:\bot\to \bot\bot,\epsilon:\bot\to \id_\cc)$ on a category $\cc$ is said to be \emph{coseparable} if there exists a natural transformation $\tau :\bot\bot\to\bot$ satisfying $\tau\circ\Delta = \id _{\bot}$ and $\bot\tau\circ\Delta\bot=\Delta \circ \tau =\tau\bot\circ\bot\Delta$.\medskip

%
%

The following results characterize the semiseparability of a right (left) adjoint functor in terms of the natural fullness of the (co)comparison functor and of the separability of the forgetful functor from the Eilenberg-Moore category of (co)modules over the associated (co)monad.

\begin{thm}\label{thm:ssepMonad}\cite[Theorem 2.9 and Theorem 2.14]{AB22}
Let $F\dashv G:\dd\to\cc$ be an adjunction.

\begin{enumerate}
  \item[(i)] $G$ is semiseparable if and only if the forgetful functor $U_{GF}:\cc_{GF}\to\cc$ is separable (equivalently, the monad $(GF,G\epsilon F,\eta)$ is separable) and the comparison functor $K_{GF}:\dd\to \cc_{GF}$ is naturally full.
  \item[(ii)] $F$ is semiseparable if and only if the forgetful functor $U^{FG}: \dd^{FG} \to \dd$ is separable (equivalently, the comonad $(FG, F\eta G, \epsilon )$ is coseparable) and the cocomparison functor $K^{FG}:\cc\to \dd^{FG}$ is naturally full.
\end{enumerate}
\end{thm}

As a consequence of Theorem \ref{thm:ssepMonad}, one recovers the following similar characterization for separable adjoint functors. The first item, should be compared with \cite[proof of Proposition 3.5]{Chen15} and \cite[Proposition 2.16]{AGM15}, while the second item is \cite[Corollary 2.15]{AB22}.

\begin{cor} \label{cor:sepmonad}
Let $F\dashv G:\dd\to\cc$ be an adjunction.
\begin{itemize}
  \item[(i)]  $G$ is separable if and only if the forgetful functor $U_{GF}:\cc_{GF}\to\cc$ is separable (equivalently, the monad $(GF,G\epsilon F,\eta)$ is separable) and the comparison functor $K_{GF}:\dd\to \cc_{GF}$ is fully faithful (i.e. $G$ is premonadic).
  \item[(ii)] $F$ is separable if and only if the forgetful functor $U^{FG}: \dd^{FG} \to \dd$ is separable (equivalently, the comonad $(FG, F\eta G, \epsilon )$ is coseparable) and the cocomparison functor $K^{FG}:\cc\to \dd^{FG}$ is fully faithful (i.e. $F$ is precomonadic).
\end{itemize}
\end{cor}

\subsection{(Co)reflections and bireflections}\label{sub:corefbiref}
Recall that
\begin{itemize}
  \item a \emph{reflection} is a functor admitting a fully faithful right adjoint;
  \item a \emph{coreflection} is a functor admitting a fully faithful left adjoint, see \cite{Ber07};
  \item a \emph{bireflection} is a functor $G:\dd\to\cc$ having a left and right adjoint equal, say $F:\cc\to \dd$, which is fully faithful and satisfies the coherent condition $\eta^r\circ \epsilon^l=\id$ where $\epsilon^l:FG\to\id$ is the counit of $F\dashv G$ while $\eta^r:\id\to  FG$ is the unit of $G\dashv F$, cf. \cite[Definition 8]{FOPTST99}.
\end{itemize} Being a coreflection (respectively a reflection) is equivalent to the fact that the unit (respectively counit) of the corresponding adjunction is an isomorphism, see \cite[Proposition 3.4.1]{Bor94}. The adjoint of the inclusion of a (co)reflective subcategory is a typical example of (co)reflection. Bireflective subcategories of a given category $\cc$ provide examples of bireflections. It is known that these subcategories correspond bijectively to split-idempotent natural transformations $e:\id_\cc\to\id_\cc$ with specified splitting, see \cite[Theorem 13]{FOPTST99} and \cite[Theorem 1.3]{Joh96}; this fact is connected to the quotient functor $H:\cc\to \cc_e$ which comes out to be a bireflection if and only if $e$ splits, see \cite[Proposition 2.27]{AB22}.\medskip

\begin{rmk}\label{rmk:coref-biref}
		(Co)reflections are closed under composition. In fact, if $G:\dd\to\cc$, $G':\e\to\dd$ are (co)reflections with fully faithful left (right) adjoints $F:\cc\to\dd$ and $F':\dd\to\e$ respectively, then $G\circ G'$ is a (co)reflection with fully faithful left (right) adjoint $F'\circ F$. Moreover, also bireflections are closed under composition. Indeed, if $G:\dd\to\cc$, $G':\e\to\dd$ are bireflections with fully faithful left and right adjoints $F$ and $F'$ respectively, satisfying the coherent conditions $\eta^r\circ \epsilon^l=\id$ and $\bar{\eta}^r\circ \bar{\epsilon}^l=\id$ where $\epsilon^l:FG\to\id$ is the counit of $F\dashv G$, $\bar{\epsilon}^l:F'G'\to\id$ is the counit of $F'\dashv G'$ while $\eta^r:\id\to  FG$ is the unit of $G\dashv F$ and $\bar{\eta}^r:\id\to  F'G'$ is the unit of $G'\dashv F'$, then $G\circ G'$ is a bireflection with fully faithful left and right adjoint $F'\circ F$, satisfying the coherent condition $F'\eta^rG'\circ\bar{\eta}^r\circ\bar{\epsilon}^l\circ F'\epsilon^l G'=\id$. 
\end{rmk}

Next result shows how the above notions interact in case the functor is semiseparable.

\begin{thm}
\label{thm:frobenius} \cite[Theorem 2.24]{AB22}
A functor is a semiseparable (co)reflection if and only if it is a naturally full (co)reflection if and only if it is a bireflection.
\end{thm}

\section{Conditions up to retracts}\label{sec:idempcmp}
In order to introduce (co)reflections up to retracts and bireflections up to retracts we have to deal with general facts about idempotent completions. First in Subsection \ref{sub:idempcom} we recall the notions of idempotent completion of categories, functors and natural transformations.
In Subsection \ref{sub:complsep} we prove that a functor $F$ is either faithful, full, fully faithful, semiseparable, separable or naturally full if and only if so is its completion $F^{\natural}$.
Then we introduce (co)reflections up to retracts and bireflections up to retracts. We collect some properties of these new notions and relate the latter one to the concepts of semiseparable and naturally full functor.
Then, in Subsection \ref{sub:featcofer}, we show that  (co)reflections (and bireflections) up to retracts verify properties of type \eqref{P1} and \eqref{P2} discussed in the Introduction.
In Subsection \ref{sub:semiadj} we consider semifunctors and semiadjunctions. Among other results, we show how to construct a semiadjunction out of a right (left) semiadjoint in the sense of \cite{MW13}.
These notions are applied in Subsection \ref{sub:condutr}  in order to provide a characterization of (co)reflections up to retracts. A first consequence is that a (co)reflection up to retracts comes out to be always surjective up to retracts. Then we give sufficient conditions guaranteeing that a functor is a (co)reflection up to retracts that will be applied to the (co)comparison functor in the next section.

\subsection{Idempotent completion}\label{sub:idempcom}
We recall from \cite{Chen15} what is the idempotent completion of a category $\cc$. 
An idempotent morphism $e:X\to X$ \emph{splits} if there exist two morphisms $p:X\to Y$ and $i:Y\to X$ such that $e=i\circ p$ and $\id _Y=p\circ i$; the category $\cc$ is said to be \emph{idempotent complete} or \emph{Cauchy complete} if all idempotents split. The \emph{idempotent completion} or \emph{Karoubi envelope} \cite{Kar78} $\cc^\natural$ of a category $\cc$ is the category whose objects are pairs $(X,e)$, where $X$ is an object in $\cc$ and $e:X\to X$ is an idempotent morphism in $\cc$; a morphism $f:(X,e)\to (X',e')$ in $\cc^\natural$ is a morphism $f:X\to X'$ in $\cc$ such that $f=e'\circ f\circ e$. Note that $\id_{(X,e)}=e:(X,e)\to (X,e).$

There is a canonical functor $$\iota_\cc : \cc\to\cc^\natural,\quad X\mapsto (X,\id_X),\quad [f:X\to Y]\mapsto [f:(X,\id_X)\to (Y,\id_Y)],$$  which is fully faithful. The functor $\iota_\cc$ is an equivalence if and only if $\cc$ is idempotent complete. A functor $F:\cc\to\dd$ can be extended to a functor $F^\natural : \cc^\natural\to\dd^\natural$, the \emph{completion} of $F$, which is defined by setting $F^\natural (X,e)= (F(X),F(e))$ and $F^\natural (f)=F(f)$, so that $\iota_\dd\circ F =F^\natural\circ\iota_\cc$, i.e. 
\begin{equation*}\xymatrixcolsep{1.6cm}\xymatrixrowsep{.6cm}
\xymatrix{\cc\ar[r]^{\iota_\cc}\ar[d]_{F}&\cc^\natural\ar[d]^{F^\natural}\\
\dd\ar[r]_{\iota_\dd}&\dd^\natural }
\end{equation*}
is a commutative diagram. A natural transformation $\alpha :F\rightarrow F^{\prime }$ induces the natural
transformation $\alpha ^{\natural }:F^{\natural }\rightarrow \left( F^{\prime
}\right) ^{\natural }$ with components $\alpha _{\left( X,e\right)
}^{\natural }:=\alpha _{X}\circ Fe=F^{\prime }e\circ \alpha _{X}$.
As a consequence, an adjunction $\left( F,G,\eta ,\epsilon \right) $
induces an adjunction $\left( F^{\natural },G^{\natural },\eta ^{\natural
},\epsilon ^{\natural }\right) $.

\subsection{The completion of semiseparable functors}\label{sub:complsep}
Next aim is to explore the behaviour of semiseparability with respect to idempotent completion. We also include the case of faithful and full functors although it is known in the literature at least in one direction.

\begin{prop}\label{prop:sscompletion}
Let $F:\cc\to\dd$ be a functor. Then,
\begin{itemize}
\item[(1)] $F$ is faithful if and only if so is $F^{\natural}$;
\item[(2)] $F$ is full if and only if so is $F^{\natural}$;
\item[(3)]$F$ is fully faithful if and only if so is $F^{\natural}$.
\end{itemize}
\end{prop}

\proof The ``only if'' part is well-known, see e.g. \cite[Lemma 58]{Ri16}.
\begin{itemize}
\item[(1)]If $F^\natural$ is faithful, then the composite $\iota_\dd\circ F=F^\natural\circ\iota_\cc$ is faithful, hence $F$ is faithful.
    \begin{invisible}
      Conversely, assume that $F$ is faithful and let $f,g:(C,e)\to (C',e')$ be morphisms in $\cc^\natural$ such that $F^\natural (f)=F^\natural (g)$. Since for all $C,C'$ in $\cc$ the map $\f_{C,C'}: \Hom_{\cc}(C,C')\rightarrow \Hom_{\dd}(F(C), F(C'))$ is injective, from $F(f)=F^\natural (f)=F^\natural (g)=F(g)$ it follows $f=g$. Therefore, also the map $$\f^{\natural} _{C,C'}: \Hom_{\cc^\natural}((C,e),(C',e'))\rightarrow \Hom_{\dd^\natural}((F(C),F(e)) ,(F(C'),F(e'))$$ is injective, so $F^\natural$ is faithful.
    \end{invisible}
\item[(2)] If $F^\natural$ is full, then $\iota_\dd\circ F=F^\natural\circ\iota_\cc$ is full. Since $\iota_\dd$ is faithful, we get that $F$ is full.
\begin{invisible}
  Indeed, given $g:FX\to FY$ in $\dd$ then $\iota_\dd g:\iota_\dd FX\to \iota_\dd FY$ must be of the form $\iota_\dd Ff$ for some $f:X\to Y$. Since $\iota_\dd$ is faithful, from $\iota_\dd g=\iota_\dd Ff$ we deduce $g=Ff$ whence $F$ is full. Conversely Assume that $F$ is full. Let $g:\f^\natural ((C,e))=(F(C),F(e))\to \f^\natural ((C',e'))=(F(C'),F(e'))$ be a morphism in $\dd^\natural$. Since $F$ is full, there exists a morphism $f:C\to C'$ in $\cc$ such that $Ff=g$. Set $f':=e'\circ f\circ e :(C,e)\to (C',e')$. Then, $F^\natural f'=Ff'=Fe'\circ Ff\circ Fe=Fe'\circ g\circ Fe = g$, hence $F^\natural$ is full.
\end{invisible}
\item[(3)] It follows from (1) and (2). \qedhere
\end{itemize}
\endproof

In the following result, the proof that the semiseparability of $F$ implies the one of $F^{\natural}$, was suggested to us by Paolo Saracco. The ``only if'' part of $(2)$ in the following result seems to be known, see e.g. \cite[Lemma 3.11]{Sun19}.
\begin{cor}\label{cor:sep-nat-ff}
Let $F:\cc\to\dd$ be a functor. Then,
\begin{itemize}
\item[(1)] $F$ is semiseparable if and only if so is $F^{\natural}$;
\item[(2)] $F$ is separable if and only if so is $F^{\natural}$;
\item[(3)] $F$ is naturally full if and only if so is $F^{\natural}$.
\end{itemize}
\end{cor}

\proof
(1) Assume that $F^\natural$ is semiseparable. Since $\iota_\cc$ is fully faithful, it is in particular naturally full, hence, by Lemma \ref{lem:comp} (ii), $F^\natural\circ\iota_\cc$ is semiseparable. From $\iota_\dd\circ F =F^\natural\circ\iota_\cc$ it follows that $\iota_\dd\circ F$ is semiseparable as well, so that, since $\iota_\dd$ is faithful, $F$ is semiseparable, by Lemma \ref{lem:comp}(iii). Conversely, if $F$ is semiseparable, then there exists a natural transformation $\p^F: \Hom_{\dd}(F-, F-)\rightarrow \Hom_{\cc}(-,-)$ such that $\f^F\p^F\f^F =\f^F$. Define $\p^{F^\natural} : \Hom_{\dd^\natural}(F^\natural -, F^\natural -)\rightarrow \Hom_{\cc^\natural}(-,-)$ by $\p^{F^\natural}_{C,C'} (g)=\p^F_{C,C'}(g)$, for every $g: (F(C),F(e))\to (F(C'),F(e'))$ in $\dd^\natural$. Since $g=Fe'\circ g\circ Fe$, by naturality of $\p^F$ it follows that $e'\circ \p^F_{C,C'}(g)\circ e =\p^F_{C,C'}(Fe'\circ g\circ Fe)= \p^F_{C,C'}(g)$, hence $\p^F_{C,C'}(g)$ is a morphism in $\cc^\natural$.  Moreover, $\p^{F^\natural}$ is a natural transformation and it holds $\f^{F^\natural}_{C,C'}\p^{F^\natural}_{C,C'}\f^{F^\natural}_{C,C'}(g)=\f^F_{C,C'}\p^F_{C,C'}\f^F_{C,C'}(g)=\f^F_{C,C'}(g)=\f^{F^\natural}_{C,C'}(g)$.

(2) and (3) follow from (1), Proposition \ref{prop:sep}  and Proposition \ref{prop:sscompletion}.
\begin{invisible}
\item[(2)] It follows from $(1)$, Proposition \ref{prop:sep} (i) and Proposition \ref{prop:sscompletion} (1).
\item[(3)] It follows from $(1)$, Proposition \ref{prop:sep} (ii) and Proposition \ref{prop:sscompletion} (2).
\end{invisible}
\endproof

\subsection{(Co)reflections and bireflections up to retracts}\label{sub:condutr}
We are now ready to introduce and investigate the announced notion of (co)reflection up to retracts. We also recall two notions that are already present in the literature, i.e. those of equivalence up to retracts and of surjective up to retracts. Recall that an object $A$ in a category $\cc$ is a \emph{retract} of an object $B$ in $\cc$ if there are morphisms $i: A\to B$ and $p:B\to A$ such that $p\circ i =\id_A$.
\begin{defn}\label{defn:(co)refl}
Consider a functor $F:\cc\to\dd$ and its completion $F^\natural:\cc^\natural\to\dd^\natural$. Then, $F$ is
\begin{itemize}
  \item an \emph{equivalence up to retracts} if $F^\natural$ is an equivalence, see \cite[page 47]{Chen15};
  \item \emph{surjective up to retracts}\footnote{These functors are also called \emph{dense up to retracts}, see \cite[Notation and conventions]{Sun19}.} if every object $D$ in $\dd$ is a retract of $FC$ for some object $C$ in $\cc$, see \cite[Definition 2.5]{BD72};
  \item a \textbf{reflection up to retracts} if $F^\natural$ is a reflection;
   \item a \textbf{coreflection up to retracts} if $F^\natural$ is  a coreflection;
  \item a \textbf{bireflection up to retracts} if $F^\natural$ is a bireflection.
  \end{itemize}
\end{defn}

In the following lemma we collect some basic facts related to the above notions.
\begin{lem}\label{lem:trivfactsutr} The following assertions hold true.
\begin{enumerate}
  \item Any equivalence is an equivalence up to retracts.
  \item Any (co)reflection is a (co)reflection up to retracts.
  \item A functor is a bireflection up to retracts if and only if it is a semiseparable (co)reflection up to retracts if and only if it is a naturally full (co)reflection up to retracts.
  \item Any bireflection is a bireflection up to retracts.
  \item An equivalence is the same thing as a fully faithful bireflection.
  \item A functor is an equivalence up to retracts if and only if it is fully faithful and
surjective up to retracts if and only if it is a fully faithful bireflection up to retracts.
  \item An equivalence up to retracts is both a reflection up to retracts and a coreflection  up to retracts.
\end{enumerate}
\end{lem}

\begin{proof}
(1) If $F$ is an equivalence with quasi-inverse $G$, then $(F^\natural,G^\natural)$ is an equivalence and hence $F$ is an equivalence up to retracts.

(2) If $G$ is a coreflection, it has a fully faithful left adjoint $F$. Thus $F^\natural\dashv G^\natural$ and $F^\natural$ is fully faithful by Proposition \ref{prop:sscompletion}. Thus $G^\natural$  is a coreflection, i.e. $G$ is a coreflection up to retracts. The proof for reflections is similar.

(3) Assume $F$ is a semiseparable (resp. naturally full) (co)reflection up to retracts. By Corollary \ref{cor:sep-nat-ff}, $F^\natural$ is a semiseparable (resp. naturally full) (co)reflection. Thus, by Theorem \ref{thm:frobenius}, $F^\natural$ is a bireflection, i.e. $F$ is a bireflection up to retracts. Conversely, by means of Theorem \ref{thm:frobenius} and Corollary \ref{cor:sep-nat-ff}, in a similar way one gets that a bireflection up to retracts is a semiseparable (resp. naturally full) (co)reflection up to retracts.

(4) A bireflection $F$ is in particular a semiseparable (co)reflection by Theorem \ref{thm:frobenius}. As a consequence of $(2)$ and $(3)$, we get that $F$ is a bireflection up to retracts.

(5) An equivalence is clearly a fully faithful bireflection, and conversely a fully faithful bireflection is an equivalence as the unit and counit of the corresponding adjunction are both invertible (see \cite[Proposition 3.4.3]{Bor94}).

(6) It is well-known that $F$ is an equivalence up to retracts if and only if it is fully faithful and
surjective up to retracts, see e.g. \cite[Lemma 3.4(2)]{Chen15}. It is also equivalent to $F$ being a fully faithful bireflection up to retracts in view of Proposition \ref{prop:sscompletion} and Theorem \ref{thm:frobenius}. \begin{invisible} In fact, if $F$ is a fully faithful bireflection up to retracts, then  $F^\natural$ is a fully faithful bireflection and hence $F^\natural$ and its adjoint are both fully faithful i.e. $F^\natural$ is an equivalence. Conversely, if $F$ is an equivalence up to retracts then $F^\natural$ is an equivalence so that it and its adjoint are both fully faithful so that $F^\natural$ is a fully faithful coreflection. By Theorem \ref{thm:frobenius}, $F^\natural$ is a fully faithful bireflection so that, by Proposition \ref{prop:sscompletion}, $F$ is a fully faithful bireflection up to retracts.    \end{invisible}

(7) If $F$ is an equivalence up to retracts, its completion $F^\natural$ is an equivalence and hence $F^\natural$ is a (co)reflection. This means that $F$ is a (co)reflection up to retracts.
 \end{proof}

\begin{rmk}
	From Remark \ref{rmk:coref-biref}, it follows that also (co)reflections up to retracts and bireflections up to retracts are closed under composition.
\end{rmk}

\begin{es}
The canonical functor $\iota _{\cc}:\cc\rightarrow
\cc^{\natural }$ is an equivalence up to retracts, see e.g. \cite[Theorem A.6]{Kah20}.

\begin{invisible}
In Example \ref{es:iotaequtr} we have seen that the  unit $\eta_\cc
:\id _{\cc^{\natural }} \rightarrow \iota _{\cc}\upsilon _{\cc} $ of the semiadjunction $(\upsilon_\cc,\iota _{\cc})$ is a split monomorphism.  As a consequence any object $C$ in $\cc^{\natural }$ is a retract of $\iota_\cc\nu_\cc C$ which is an object in $%
\mathrm{Im}\left( \iota _{\cc}\right) $ and hence $\iota_\cc$ is surjective up to retracts. Since $\iota _{\cc}$
is fully faithful it follows, by Lemma \ref{lem:trivfactsutr}, that $\iota
_{\cc}$ is an equivalence up to retracts. By \cite[Theorem 7]{Ho90}, it follows that $\upsilon^\natural_\cc\dashv\iota^\natural_\cc$ and $\iota^\natural_\cc\dashv\upsilon^\natural_\cc$. Thus $\upsilon^\natural_\cc$ is a quasi-inverse of the equivalence $\iota^\natural_\cc$.
\end{invisible}
\end{es}

Recall that a functor $F:\cc\to\dd$ is called a \emph{Maschke functor} if it reflects split-monomorphisms i.e. if, for every morphism $i$ in $\cc$ such that $Fi$ is a split-monomorphism, then $i$ is a split-monomorphism\footnote{This is equivalent to \cite[Remark 6]{CMZ02}, where $F$ is called a Maschke functor if every object in $\cc$ is relative injective. Recall that an object $M$ is called relative injective if,  for every morphism $i:C\to C'$
such that $Fi$ is a split-monomorphism, then the map $\Hom_\cc(i,M):\Hom_\cc(C',M)\to \Hom_\cc(C,M),f\mapsto f\circ i$, is surjective.}. Similarly, $F$ is a \emph{dual Maschke functor} if it reflects split-epimorphisms. A functor is called \emph{conservative} if it reflects isomorphisms.

\begin{rmk}\label{rmk:Maschke}
By \cite[Proposition 1.2]{NVV89} a separable functor is both Maschke and dual Maschke.
Moreover a functor which is both Maschke and dual Maschke is conservative.
\end{rmk}

\begin{es}  Let $(F,G)$ be an adjunction.  Then, by \cite[Corollary 5]{Zhu03}, the functor $F$ is a Maschke functor if and only if $G$ is surjective up to retracts. Dually, the functor $G$ is dual Maschke if and only if $F$ is surjective up to retracts.
\begin{invisible}
We already observed that the definition we have given of $F$ to be a Maschke functor is equivalent to require that any object $C\in\cc$ is relative injective. On the other hand, one has that $C\in\cc$ is relative injective if and only if the unit $\eta_C:C\to GFC$ is a split-monomorphism, see \cite[Proposition 60]{CMZ02}. As a consequence $F:\cc\to\dd$ is a Maschke functor if and only if, for every object $C$ in $\cc$, one has that the unit $\eta_C:C\to GFC$ is a split-monomorphism. If $F:\cc\to\dd$ is a Maschke functor then the unit $\eta_C:C\to GFC$ is a split-monomorphism and hence $C$ is a retract of $GFC$. Conversely, if $C$ is endowed with morphisms $i:C\to GD$ and $r:GD\to C$ such that $r\circ i=\id_C$, then $r\circ G\epsilon_D\circ GFi\circ \eta_C=r\circ G\epsilon_D\circ \eta_{GD}\circ i=r\circ i=\id_C$ so that $\eta_C$ is a split-monomorphism and hence $F:\cc\to\dd$ is a Maschke functor.\par
Dually, $G:\dd\to\cc$ is a dual Maschke functor if and only if any object $D\in\dd$ is relative projective, hence if and only if, for every object $D$ in $\dd$, the counit $\epsilon_D:FGD\to D$ is a split-epimorphism. Assume $G:\dd\to\cc$ is a dual Maschke functor. Then the counit $\epsilon_D:FGD\to D$ is a split-epimorphism and hence $D$ is a retract of $FGD$. Conversely, if $D$ is endowed with morphisms $i:D\to FC$ and $r:FC\to D$ such that $r\circ i=\id_D$, then $\epsilon_D\circ FGr\circ F\eta_C\circ i=r\circ \epsilon_{FC}\circ F\eta_{C}\circ i=r\circ i=\id_D$, hence $\epsilon_D$ is a split-epimorphism and so $G:\dd\to\cc$ is a dual Maschke functor.
\end{invisible}
\end{es}

%
%

\subsection{Two peculiar features}\label{sub:featcofer}
The following result includes among others the announced property \eqref{P1}, discussed in the Introduction, for a coreflection up to retracts, namely that, if it has a left adjoint, then it is a coreflection.

\begin{prop}\label{prop:corefl}
The following assertions hold true.
\begin{enumerate}
  \item If a coreflection up to retracts has a left adjoint, then it is a coreflection.
  \item If a coreflection up to retracts has a right adjoint, then it is a reflection.
  \item If a reflection up to retracts has a right adjoint, then it is a reflection.
  \item If a reflection up to retracts has a left adjoint, then it is a coreflection.
  \item If a bireflection up to retracts has an adjoint, then it is a bireflection.
  \item If an equivalence up to retracts has an adjoint, then it is an equivalence.
\end{enumerate}
\end{prop}
\proof
(1)  If $G$ has a left adjoint $F$, then $F^\natural\dashv G^\natural$. If $G$ is a coreflection up to retracts, then  $G^\natural$ is a coreflection. Thus $F^\natural$ is fully faithful and hence so is $F$ by  Proposition \ref{prop:sscompletion}, i.e. $G$ is a coreflection.


(2) If $F$ has a right adjoint $G$, then $F^\natural\dashv G^\natural$. If $F$ is a coreflection up to retracts, then  $F^\natural$ is a coreflection. Thus it has a fully faithful left adjoint. Then also the right adjoint $G^\natural$ is fully faithful by \cite[Proposition 3.4.2]{Bor94}. By Proposition \ref{prop:sscompletion}, $G$ is fully faithful i.e. $F$ is a reflection.

(3) is dual to (1) and (4) is dual to (2).

(5) If $F$ is a bireflection up to retracts, then by Lemma \ref{lem:trivfactsutr} $F$ is a naturally full (co)reflection up to retracts.
If $F$ has a left adjoint, by (1), it is a naturally full coreflection while if $F$ has a right adjoint, by (3), it is a naturally full reflection. In both cases, by Theorem \ref{thm:frobenius}, $F$ is a bireflection.

(6) By Lemma \ref{lem:trivfactsutr} an equivalence up to retracts is a fully faithful bireflection up to retracts. If it has an adjoint, by (5), it is a fully faithful bireflection, i.e. an equivalence by Lemma \ref{lem:trivfactsutr}.
\endproof
\begin{rmk}
		By Proposition \ref{prop:corefl} and Lemma \ref{lem:trivfactsutr}, it follows that 
\begin{itemize}
    \item any coreflection up to retracts with a right adjoint is a reflection up to retracts,
    \item   any reflection up to retracts with a left adjoint is a coreflection up to retracts.
\end{itemize}
\end{rmk}

%

We are now going to prove the property \eqref{P2}, announced in the Introduction, namely that  a coreflection up to retracts whose source category is idempotent complete has a left adjoint (it is indeed a coreflection). First we need the following lemma.

\begin{lem}
\label{lem:idpcom-adj}Let $\dd$ be an idempotent complete category. A functor $G:\dd\rightarrow \cc$ has a
left (resp. right) adjoint if and only if so does $G^{\natural }$.
\end{lem}

\begin{proof}
 If $F\dashv G$, we know that $F^{\natural }\dashv G^{\natural }$. Conversely, assume that $L\dashv G^{\natural }:\dd^{\natural }\rightarrow \cc%
^{\natural }$. Since $\dd$ is idempotent complete, the functor $\iota _{\dd}:%
\dd\rightarrow \dd^{\natural }$ is an equivalence of
categories and hence it has a left adjoint $V_{\dd}:\dd%
^{\natural }\rightarrow \dd$. From $V_{\dd}\dashv \iota _{\dd}$ and $L\dashv G^{\natural }$, we get
$V_{\dd}L\dashv G^{\natural }\iota _{\dd}$ and hence $V_{\dd}L\dashv \iota _{\cc} G$. Since $\iota _{\cc}$ is fully faithful, this implies $V_{\dd}L\iota _{\cc}\dashv G$.
\begin{invisible}
Let $F\dashv UG$ with $U:\cc\to\cc$ fully faithful. Given objects $C$ and $D$ in the relevant categories, we have a chain of natural isomorphisms
$\Hom\left(FUC,D\right) \cong
\Hom\left( UC,UGD\right) \cong
\Hom\left( C,GD\right),$
the latter one being fully faithfulness of $U$. Thus $FU\dashv G$. Here $U=\iota_\cc:\cc\to\cc^\natural$. Given an object $C$ in $\cc$
	and an object $D$ in $\dd$, we have a chain of natural isomorphisms	$
	\Hom_{\dd}\left( V_{\dd}L\iota _{\cc%
	}C,D\right) \cong \Hom_{\dd^{\natural }}\left( L\iota _{%
		\cc}C,\iota _{\dd}D\right) \cong \Hom_{\cc%
		^{\natural }}\left( \iota _{\cc}C,G^{\natural }\iota _{\dd%
	}D\right) =\Hom_{\cc^{\natural }}\left( \iota _{\cc%
	}C,\iota _{\cc}GD\right) \cong \Hom_{\cc}\left(
	C,GD\right)
	$ where the latter isomorphism results from the fully faithfulness of $\iota _{%
		\cc}$.
\end{invisible}
The case in which $G$ has a right adjoint follows similarly.
\begin{invisible}
If $G\dashv H$, then $G^{\natural }\dashv H^{\natural }$. Conversely, assume that $G^{\natural }\dashv R:\cc^{\natural }\rightarrow \dd^{\natural }$. Since $%
\dd$ is idempotent complete, the functor $\iota _{\dd}:%
\dd\rightarrow \dd^{\natural }$ is an equivalence of
categories and hence it has a right adjoint $V_{\dd}:\dd%
^{\natural }\rightarrow \dd$. From $\iota _{\dd}\dashv V_{\dd}$ and $G^{\natural }\dashv R$, we get
 $G^{\natural } \iota _{\dd}\dashv V_{\dd} R$ i.e. $\iota _{\cc}G\dashv V_{\dd} R$. Since $\iota _{\cc}$  is fully faithful we deduce that $G\dashv V_{\dd} R\iota _{\cc}$. In fact, if $UF\dashv G$ with $U$ fully faithful, we get $\Hom\left(C,GUD\right) \cong
\Hom\left( UFC,UD\right) \cong
\Hom\left( FC,D\right)$.
\end{invisible}
\end{proof}

\begin{prop}
\label{prop:idpcom-cutr}Let $\dd$ be an idempotent complete category.
A functor $G:\dd\rightarrow \cc$  is a coreflection (resp. reflection, bireflection, equivalence) up to retracts if and only if it is a coreflection (resp. reflection, bireflection, equivalence).
\end{prop}

\begin{proof}If $G$ is a coreflection (resp. reflection) up to retracts, then $G^{\natural }$ has a left (resp. right)
adjoint so that, by Lemma \ref{lem:idpcom-adj}, so does $G.$ By Proposition \ref{prop:corefl} $G$ is a coreflection (resp. reflection). The other implication is
always true by Lemma \ref{lem:trivfactsutr}.
\begin{invisible}
If $G$ is a reflection up to retracts, then $G^{\natural }$ has a right
adjoint so that, by the foregoing, so does $G.$ By Proposition \ref{prop:corefl} $G$ is a reflection. The other implication is
always true.
\end{invisible}
Similarly, one deals with the case of bireflection and equivalence.
\end{proof}
For a deeper understanding of (co)reflections up to retracts, we are now going to investigate the notion of semiadjunction.

\subsection{Semiadjunctions}\label{sub:semiadj}
Recall from \cite{Ha85} that a \emph{semifunctor} is defined the same way as a functor, except that
a semifunctor needs not to preserve identities. The notion of semifunctor originally appeared in \cite[Definition 4.1]{EZ76} under the name of \emph{weak functor}.
For semifunctors $F,F^{\prime
}:\cc\rightarrow \dd$, a \emph{natural transformation} $\alpha
:F\rightarrow F^{\prime }$ is a family $\left( \alpha _{C}:FC\rightarrow
F^{\prime }C\right) _{C\in \cc}$ of morphisms in $\dd$ such
that $\alpha _{D}\circ Ff=F^{\prime }f\circ \alpha _{C}$ for every morphism $%
f:C\rightarrow D.$ If moreover $\alpha _{C}\circ F\left( \id_{C}\right)
=\alpha _{C}$, then $\alpha$ is called a  \emph{seminatural transformation}. By a \emph{semiadjunction} we mean a datum $\left( F,G,\eta
,\epsilon \right) $ where $F:\cc\rightarrow \dd$ and $G:%
\dd\rightarrow \cc$ are semifunctors endowed with natural transformations $\eta :\id_{\cc}\rightarrow GF$
(unit) and $\epsilon :FG\rightarrow \id_{\dd}$ (counit) such
that $G\epsilon \circ \eta G=G\id$ and $\epsilon
F\circ F\eta =F\id$, see \cite[Definition 22]{Ho90}. Although the terminology suggests that it is a weaker notion, a seminatural transformation $\alpha :F\rightarrow F^{\prime }$ is in particular a natural transformation but the converse is not true in general. It is true in case either $F$ or $F'$ is a functor, see \cite[Theorem 16]{Ho90}. For this reason, $\eta$ and $\epsilon$ as above are also seminatural transformations.

Any semifunctor $F:\cc\rightarrow \dd$ induces a functor $%
F^{\natural }:\cc^{\natural }\rightarrow \dd^{\natural }$
such that $F^{\natural }\left( C,c\right) =\left( FC,Fc\right) $ and $%
F^{\natural }f=Ff.$ In fact $F^{\natural }\id_{\left( C,c\right)
}=Fc=\id_{\left( FC,Fc\right) }=\id_{F^{\natural }\left(
C,c\right) },$ as observed in \cite[Definition 1.3]{Ha85}.
However note that $\iota_\dd\circ F \neq F^\natural\circ\iota_\cc$ unless $F$ is a functor.
Moreover any semifunctor is determined by its completion, cf. \cite[Proposition 1.4]{Ha85}.

Any seminatural transformation $\alpha :F\rightarrow F^{\prime }$ induces the natural
transformation $\alpha ^{\natural }:F^{\natural }\rightarrow \left( F^{\prime
}\right) ^{\natural }$ with components $\alpha _{\left( C,c\right)
}^{\natural }:=\alpha _{C}\circ Fc=F^{\prime }c\circ \alpha _{C},$ cf. \cite[%
Theorem 20]{Ho90}.

As a consequence any semiadjunction $\left( F,G,\eta ,\epsilon \right) $
induces an adjunction $\left( F^{\natural },G^{\natural },\eta ^{\natural
},\epsilon ^{\natural }\right) $ where $\eta _{\left( C,c\right) }^{\natural
}=\eta _{C}\circ c:\left( C,c\right) \rightarrow \left( GFC,GFc\right) $ and
$\epsilon _{\left( D,d\right) }^{\natural }=d\circ \epsilon _{D}:\left(
FGD,FGd\right) \rightarrow \left( D,d\right) $.

\begin{es}\label{es:iotaequtr}
Consider the canonical functor $\iota _{\cc}:\cc\rightarrow
\cc^{\natural }$. There is also a semifunctor $\upsilon_\cc : \cc^\natural \to\cc$ which maps an object 
$\left( C,c\right) $ in $\cc^\natural$ to the underlying object $C$ and a morphism $f:(C,c)\to (C',c') $ to the underlying morphism $%
\upsilon_\cc f:C\rightarrow C'$ such that $c'\circ\upsilon_\cc f\circ c=\upsilon_\cc f$. It is a semifunctor as $\upsilon_\cc (\id_{(C,c)})=c\neq \id_C$ in general. By \cite[Example 6]{Ho90} we have that $(\upsilon_\cc,\iota _{\cc})$ and $(\iota _{\cc},\upsilon_\cc)$ are semiadjunctions. Let us exhibit explicitly their units and counits. Note that $\iota _{\cc%
}\upsilon _{\cc}\left( C,c\right) =\left( C,\id_{C}\right)$. \begin{itemize}
  \item The unit of $(\upsilon_\cc,\iota _{\cc})$  is defined by $(\eta_\cc)
_{\left( C,c\right) }=c:\left( C,c\right) \rightarrow \left( C,\id_{C}\right) $.
  \item The counit of $(\upsilon_\cc,\iota _{\cc})$  is $\epsilon _{\cc} :=\id_{%
\id_{\cc}}:\upsilon _{\cc}\iota _{\cc}=%
\id_{\cc}\rightarrow \id_{\cc}$.
\item The unit of $(\iota _{\cc},\upsilon_\cc)$ is $\epsilon _{\cc} :=\id_{\id_{\cc}}:\id_\cc\to\id_\cc=\upsilon _{\cc}\iota _{\cc}$.
\item The counit of $(\iota _{\cc},\upsilon_\cc)$ is defined by $(\nu_\cc) _{\left( C,c\right) }=c:\left( C,\id_{C}\right) \rightarrow \left(
C,c\right) $.
\end{itemize}
One has that $\eta _\cc\circ \nu _\cc=\iota _\cc\upsilon _\cc\id$  and $\nu_\cc\circ\eta_\cc=\id $.
\end{es}

\begin{invisible}

For every $f:\left( C,c\right) \rightarrow \left( C^{\prime },c^{\prime
}\right) $ we have
\begin{eqnarray*}
\upsilon _{\cc}\left( \eta _{\left( C^{\prime },c^{\prime }\right)
}\circ f\right)  &=&\upsilon _{\cc}\eta _{\left( C^{\prime
},c^{\prime }\right) }\circ \upsilon _{\cc}f=c^{\prime }\circ
\upsilon _{\cc}f=\upsilon _{\cc}f\circ c=\upsilon _{\mathcal{%
C}}f\circ \upsilon _{\cc}\eta _{\left( C,c\right) }=\upsilon _{%
\cc}\left( \iota _{\cc}\upsilon _{\cc}f\circ \eta
_{\left( C,c\right) }\right) , \\
\upsilon _{\cc}\left( \eta _{\left( C,c\right) }\circ \id%
_{\left( C,c\right) }\right)  &=&c\circ c=c=\upsilon _{\cc}\left(
\eta _{\left( C,c\right) }\right)
\end{eqnarray*}%
so that $\eta _{\left( C^{\prime },c^{\prime }\right) }\circ f=\iota _{%
\cc}\upsilon _{\cc}f\circ \eta _{\left( C,c\right) },$ and $%
\eta _{\left( C,c\right) }\circ \id_{\left( C,c\right) }=\eta
_{\left( C,c\right) }$ and we can define the seminatural transformation $%
\eta _{\cc}:=\eta :\id_{\cc^{\natural }}\rightarrow
\iota _{\cc}\upsilon _{\cc}$ by setting $\eta :=\left( \eta
_{\left( C,c\right) }\right) _{\left( C,c\right) \in \cc^{\natural }}
$. Similarly $\epsilon _{C}\circ \upsilon _{\cc}\iota _{\cc%
}\left( \id_{C}\right) =\epsilon _{C}\circ \id_{C}=\epsilon
_{C}$ and hence $\epsilon :=\epsilon _{\cc}:\upsilon _{\cc%
}\iota _{\cc}=\id_{\cc}\rightarrow \id_{%
\cc}$ is seminatural. We compute
\begin{eqnarray*}
\left( \epsilon \upsilon _{\cc}\circ \upsilon _{\cc}\eta
\right) _{\left( C,c\right) } &=&\epsilon _{\upsilon _{\cc}\left(
C,c\right) }\circ \upsilon _{\cc}\eta _{\left( C,c\right) }=\epsilon
_{C}\circ c=\id_{C}\circ c=c=\upsilon _{\cc}\left( \mathrm{Id%
}_{\left( C,c\right) }\right)  \\
\upsilon _{\cc}\left( \iota _{\cc}\epsilon \circ \eta \iota
_{\cc}\right) _{C} &=&\upsilon _{\cc}\iota _{\cc%
}\epsilon _{C}\circ \upsilon _{\cc}\eta _{\left( C,\id%
_{C}\right) }=\epsilon _{C}\circ \id_{C}=\id_{C}\circ
\id_{C}=\id_{C}=\upsilon _{\cc}\iota _{\cc%
}\left( \id_{C}\right)
\end{eqnarray*}%
so that $\epsilon \upsilon _{\cc}\circ \upsilon _{\cc}\eta
=\upsilon _{\cc}\id$ and $\iota _{\cc}\epsilon \circ
\eta \iota _{\cc}=\iota _{\cc}\id$. Hence $\left(
\upsilon _{\cc},\iota _{\cc},\eta ,\epsilon \right) $ is a
semiadjunction.

Define $\nu _{\left( C,c\right) }:\iota _{\cc}\upsilon _{\cc%
}\left( C,c\right) =\left( C,\id_{C}\right) \rightarrow \left(
C,c\right) $ such that $\upsilon _{\cc}\nu _{\left( C,c\right)
}:=c:C\rightarrow C$. Given $f:\left( C,c\right) \rightarrow \left(
C^{\prime },c^{\prime }\right) ,$ we have%
\begin{eqnarray*}
\upsilon _{\cc}\left( \nu _{\left( C^{\prime },c^{\prime }\right)
}\circ \iota _{\cc}\upsilon _{\cc}f\right)  &=&\upsilon _{%
\cc}\nu _{\left( C^{\prime },c^{\prime }\right) }\circ \upsilon _{%
\cc}f=c^{\prime }\circ \upsilon _{\cc}f=\upsilon _{\mathcal{C%
}}f\circ c=\upsilon _{\cc}f\circ \upsilon _{\cc}\nu _{\left(
C,c\right) }=\upsilon _{\cc}\left( f\circ \nu _{\left( C,c\right)
}\right)  \\
\upsilon _{\cc}\left( \nu _{\left( C,c\right) }\circ \iota _{%
\cc}\upsilon _{\cc}\id_{\left( C,c\right) }\right)
&=&\upsilon _{\cc}\left( \id_{\left( C,c\right) }\circ \nu
_{\left( C,c\right) }\right) =\upsilon _{\cc}\left( \nu _{\left(
C,c\right) }\right)
\end{eqnarray*}%
so that $\nu _{\left( C^{\prime },c^{\prime }\right) }\circ \iota _{\mathcal{%
C}}\upsilon _{\cc}f=f\circ \nu _{\left( C,c\right) }$ and $\nu
_{\left( C,c\right) }\circ \iota _{\cc}\upsilon _{\cc}%
\id_{\left( C,c\right) }=\nu _{\left( C,c\right) }$ and hence $\nu
:=\left( \nu _{\left( C,c\right) }\right) _{\left( C,c\right) \in \cc%
^{\natural }}:\iota _{\cc}\upsilon _{\cc}\rightarrow \mathrm{%
Id}_{\cc^{\natural }}$ is seminatural. We have%
\begin{eqnarray*}
\upsilon _{\cc}\nu _{\left( C,c\right) }\circ \left( \epsilon _{%
\cc}^{-1}\right) _{\upsilon _{\cc}\left( C,c\right) }
&=&c\circ \left( \epsilon _{\cc}^{-1}\right) _{C}=c\circ \id%
_{C}=c=\upsilon _{\cc}\left( \id_{\left( C,c\right) }\right)
, \\
\upsilon _{\cc}\left( \nu _{\iota _{\cc}C}\circ \iota _{%
\cc}\left( \epsilon _{\cc}^{-1}\right) _{C}\right)
&=&\upsilon _{\cc}\nu _{\left( C,\id_{C}\right) }\circ
\upsilon _{\cc}\iota _{\cc}\id_{C}=\id%
_{C}\circ \id_{C}=\id_{C}=\upsilon _{\cc}\iota _{%
\cc}\left( \id_{C}\right)
\end{eqnarray*}%
and hence $\upsilon _{\cc}\nu \circ \epsilon _{\cc%
}^{-1}\upsilon _{\cc}=\upsilon _{\cc}\id$ and $\nu
\iota _{\cc}\circ \iota _{\cc}\epsilon _{\cc%
}^{-1}=\iota _{\cc}\id$ so that $\left( \iota _{\cc%
},\upsilon _{\cc}\right)$ is a semiadjunction. Note that
\begin{equation*}
\upsilon _{\cc}\left( \nu _{\left( C,c\right) }\circ \eta _{\left(
C,c\right) }\right) =\upsilon _{\cc}\nu _{\left( C,c\right) }\circ
\upsilon _{\cc}\eta _{\left( C,c\right) }=c\circ c=c=\upsilon _{%
\cc}\left( \id_{\left( C,c\right) }\right)
\end{equation*}%
and hence $\nu \circ \eta =\id$.
Note that $\upsilon _{\cc}\left( \left( \eta _{\cc}\right)
_{\left( C,c\right) }\circ \nu _{\left( C,c\right) }\right) =\upsilon _{%
\cc}\left( \eta _{\cc}\right) _{\left( C,c\right) }\circ
\upsilon _{\cc}\nu _{\left( C,c\right) }=c\circ c=c=\upsilon _{%
\cc}\iota _{\cc}\upsilon _{\cc}\id_{\left(
C,c\right) }$ so that $\left( \eta _{\cc}\right) _{\left( C,c\right)
}\circ \left( \nu _{\cc}\right) _{\left( C,c\right) }=\iota _{%
\cc}\upsilon _{\cc}\id_{\left( C,c\right) }$ and
hence $\eta _{\cc}\circ \nu _{\cc}=\iota _{\cc%
}\upsilon _{\cc}\id.$
\end{invisible}

\begin{invisible}
Given a seminatural transformation $\alpha :F\rightarrow F^{\prime }$, note
that $F^{\prime }c\circ \left( \alpha _{C}\circ Fc\right) \circ Fc=F^{\prime
}c\circ \alpha _{C}\circ Fc\circ Fc=\alpha _{C}\circ Fc\circ Fc\circ
Fc=\alpha _{C}\circ Fc$ and hence $\alpha _{\left( C,c\right) }^{\natural
}:=\alpha _{C}\circ Fc$ is really a morphism $F^{\natural }\left( C,c\right)
=\left( FC,Fc\right) \rightarrow \left( F^{\prime }C,F^{\prime }c\right)
=\left( F^{\prime }\right) ^{\natural }\left( C,c\right) $. Given $f:\left(
C,c\right) \rightarrow \left( C^{\prime },c^{\prime }\right) $, we compute
\begin{equation*}
\left( F^{\prime }\right) ^{\natural }f\circ \alpha _{\left( C,c\right)
}^{\natural }=F^{\prime }f\circ \alpha _{C}\circ Fc=\alpha _{C}\circ Ff\circ
Fc=\alpha _{C}\circ F\left( f\circ c\right) =\alpha _{C}\circ Ff=\alpha
_{C}\circ F^{\natural }f
\end{equation*}%
and hence we get a seminatural transformation $\alpha ^{\natural
}:F^{\natural }\rightarrow \left( F^{\prime }\right) ^{\natural }$ as
desired. Note that, if $\beta :F^{\prime }\rightarrow F^{\prime \prime }$ is
another seminatural transformation, then $\beta _{\left( C,c\right)
}^{\natural }\circ \alpha _{\left( C,c\right) }^{\natural }=\beta _{C}\circ
F^{\prime }c\circ \alpha _{C}\circ Fc=\beta _{C}\circ \alpha _{C}\circ
Fc\circ Fc=\left( \beta \circ \alpha \right) _{C}\circ Fc=\left( \beta \circ
\alpha \right) _{\left( C,c\right) }^{\natural }$ so that $\beta ^{\natural
}\circ \alpha ^{\natural }=\left( \beta \circ \alpha \right) ^{\natural }.$
Moreover%
\begin{eqnarray*}
\left( \alpha G\right) _{\left( C,c\right) }^{\natural } &=&\left( \alpha
G\right) _{C}\circ FGc=\alpha GC\circ FGc =\alpha _{\left( GC,Gc\right)
}^{\natural }=\left( \alpha ^{\natural }G^{\natural }\right) _{\left(
C,c\right) } \\
\left( H\alpha \right) _{\left( C,c\right) }^{\natural } &=&\left( H\alpha
\right) _{C}\circ HFc=H\alpha _{C}\circ HFc=H\left( \alpha _{C}\circ
Fc\right) =H\alpha _{\left( C,c\right) }^{\natural }=H^{\natural }\alpha
_{\left( C,c\right) }^{\natural }
\end{eqnarray*}%
so that $\left( \alpha G\right) ^{\natural }=\alpha ^{\natural }G^{\natural }
$ and $\left( H\alpha \right) ^{\natural }=H^{\natural }\alpha ^{\natural }$%
.

In particular we can consider $\eta ^{\natural }:\left( \id_{%
\cc}\right) ^{\natural }\rightarrow \left( GF\right) ^{\natural }$
i.e. $\eta ^{\natural }:\id_{\cc^{\natural }}\rightarrow
G^{\natural }F^{\natural }$ and similarly $\epsilon ^{\natural }:F^{\natural
}G^{\natural }\rightarrow \id_{\dd^{\natural }}$. We have%
\begin{eqnarray*}
G^{\natural }\epsilon ^{\natural }\circ \eta ^{\natural }G^{\natural }
&=&\left( G\epsilon \right) ^{\natural }\circ \left( \eta G\right)
^{\natural }=\left( G\epsilon \circ \eta G\right) ^{\natural }=\left(
G\id\right) ^{\natural }=\id_{G^{\natural }}, \\
\epsilon ^{\natural }F^{\natural }\circ F^{\natural }\eta ^{\natural }
&=&\left( \epsilon F\right) ^{\natural }\circ \left( F\eta \right)
^{\natural }=\left( \epsilon F\circ F\eta \right) ^{\natural }=\left(F\id\right) ^{\natural }=\id_{F^{\natural }}\text{.}
\end{eqnarray*}
\end{invisible}

\begin{invisible}
\begin{lem}\label{lem:Fiota}
Let $F,G:\cc^\natural\to \dd$ be functors such that $F\circ\iota_\cc=G\circ\iota_\cc$. Then $F\cong G$.
\end{lem}

\begin{proof}It follows from the fact that $\mathrm{Funct}(\cc^\natural,\dd)\to\mathrm{Funct}(\cc,\dd),F\mapsto F\circ\iota_\cc$, is an equivalence of categories, see e.g. \cite[Expos\'{e} IV, Exercise 7.5]{SGA4}.
LA PARTE RESTANTE DELLA DIM. ERA INVISIBILE:\\
  From Example \ref{es:iotaequtr}, consider the seminatural transformations $\eta _\cc:\id_{\cc^{\natural
}}\rightarrow \iota _\cc\upsilon _\cc$ and $\nu _\cc:\iota _\cc\upsilon _\cc\rightarrow \id_{\cc^{\natural
}}$ such that
$\eta _\cc\circ \nu _\cc=\iota _\cc\upsilon _\cc\id$  and $\nu_\cc\circ\eta_\cc=\id $. Thus $F\iota _{\cc}\upsilon _{%
\cc}\overset{F\nu _{\cc}}{\rightarrow }F\overset{F\eta _{%
\cc}}{\rightarrow }F\iota _{\cc}\upsilon _{\cc}$ and
$G\iota _{\cc}\upsilon _{\cc}\overset{%
G\nu _{\cc}}{\rightarrow }G\overset{%
G\eta _{\cc}}{\rightarrow }G\iota _{\mathcal{%
C}}\upsilon _{\cc}$ are two ways to split the same idempotent as $%
F\eta _{\cc}\circ F\nu _{\cc}=F\left( \eta _{\cc%
}\circ \nu _{\cc}\right) =F\iota _{\cc}\upsilon _{\cc%
}\id=G\iota _{\cc}\upsilon _{\cc}\mathrm{%
Id}=G\left( \eta _{\cc}\circ \nu _{\cc}\right)
=G\eta _{\cc}\circ G\nu _{\cc}$. As
a consequence the compositions $ F\overset{F\eta _{\cc}}{%
\rightarrow }F\iota _{\cc}\upsilon _{\cc}=G\iota
_{\cc}\upsilon _{\cc}\overset{G\nu _{\cc}%
}{\rightarrow }G$ and $G\overset{G\eta _{\cc}}{\rightarrow }%
G\iota _{\cc}\upsilon _{\cc}=F\iota _{\cc%
}\upsilon _{\cc}\overset{F\nu _{\cc}}{\rightarrow }F$ are mutual inverses.

\end{proof}
\end{invisible}

We include here the following well-known lemma that will be useful afterwards.

\begin{lem}\label{lem:anticompl} (Cf. \cite[proof of Theorem 1]{HoM95}) Let $\cc$ and $\dd$ be categories.
\begin{enumerate}
  \item[1)] For every functor $G:\cc^{\natural }\to \dd^{\natural }$, then $F:=\upsilon_\dd\circ G\circ\iota_\cc:\cc\rightarrow\dd$ is a semifunctor such that $F^\natural \cong G$.
  \item[2)] Given semifunctors $F,G:\cc\rightarrow \dd$ and a natural
transformation $\alpha :F^{\natural }\rightarrow G^{\natural }$, then $\beta :=\upsilon _{\dd}\alpha \iota_{\cc}:F\rightarrow G$ is a seminatural transformation such that $\beta ^{\natural }=\alpha$.
\end{enumerate}
\end{lem}

\begin{invisible}
1) Consider the semifunctor $\upsilon_\cc : \cc^\natural \to\cc$ of Example \ref{es:iotaequtr}. Now set $F:=\upsilon_\dd\circ G\circ\iota_\cc :\cc\to\dd$, which is a semifunctor, as $\upsilon_\dd: \dd^\natural\to\dd$ is so. Note that for every $(X,e)$ in $\cc^\natural$ and $f:(X,e)\to (X',e')$ in $\cc^\natural$
$$\upsilon_\dd F^\natural (X,e)=\upsilon_\dd(FX,Fe)=FX=F\upsilon_\cc (X,e),$$
$$\upsilon_\dd F^\natural (f)=\upsilon_\dd (F^\natural f:(FX,Fe)\to (FX',Fe'))=F\upsilon_\cc (f),$$
so that $\upsilon_\dd\circ F^\natural = F\circ\upsilon_\cc$.

For every object $C$ in $\cc,$ we have $FC=\upsilon _{\dd%
}G\iota _{\cc}C$ so that $G\iota _{\cc}C=\left(
FC,e_{C}\right) $ for some morphism $e_{C}:FC\rightarrow FC$. Then $
F\id_{C}=\upsilon _{\dd}G\iota _{\cc}\id%
_{C}=\upsilon _{\dd}\id_{G\iota _{\cc}\left(
C\right) }=\upsilon _{\dd}\id_{\left( FC,e_{C}\right) }=e_{C}
$ so that $G\iota _{\cc}C=\left( FC,e_{C}\right) =\left( FC,F\mathrm{Id%
}_{C}\right) =F^{\natural }\left( C,\id_{C}\right) =F^{\natural
}\iota _{\cc}C.$ Given a morphism $%
f:C\rightarrow C^{\prime }$, we have $\upsilon _{\dd}F^{\natural
}\iota _{\cc}f=F\upsilon _{\cc}\iota _{\cc%
}f=Ff=\upsilon _{\dd}G\iota _{\cc}f$ and hence $F^{\natural
}\iota _{\cc}f=G\iota _{\cc}f$. We have so proved that $
F^{\natural }\iota _{\cc}=G\iota _{\cc}$. By Lemma \ref{lem:Fiota}, we get $F^{\natural }\cong G$.

2) Since $F=\upsilon_\dd\circ F^\natural\circ\iota_\cc$ and $G=\upsilon_\dd\circ G^\natural\circ\iota_\cc$, we can consider the seminatural transformation $\beta=\upsilon_\dd\alpha\iota_\cc :F\to G$. Then $\beta$ is really seminatural as for every $C\in\cc$, $\beta_C\circ F(\id_C)=\upsilon_\dd\alpha\iota_\cc C\circ \upsilon_\dd F^\natural\iota_\cc (\id _C)=\upsilon_\dd(\alpha\iota_\cc C\circ \id_{F^\natural\iota_\cc C})=\upsilon_\dd \alpha\iota_\cc C =\beta_C$, and moreover, for every $f:C\to D$ in $\cc$, since $\alpha_{\iota_\cc D}\circ F^\natural\iota_\cc f=G^\natural\iota_\cc f\circ\alpha_{\iota_\cc C}$, we have $\beta_D\circ Ff = (\upsilon_\dd\alpha\iota_\cc D)\circ (\upsilon_\dd F^\natural\iota_\cc )(f)=\upsilon_\dd (\alpha\iota_\cc D\circ F^\natural \iota_\cc f)=\upsilon_\dd (G^\natural \iota_\cc f\circ\alpha_{\iota_\cc C})=\upsilon_\dd G^\natural\iota_\cc (f)\circ\upsilon_\dd\alpha\iota_\cc C=Gf\circ \beta_C$.

By naturality of $\alpha :F^{\natural }\rightarrow G^{\natural }$ we
have $G^{\natural }\eta _{\cc}\circ \alpha =\alpha \iota _{\mathcal{C%
}}\upsilon _{\cc}\circ F^{\natural }\eta _{\cc}$ where $\eta
_{\cc}:\id_{\cc^{\natural }}\rightarrow \iota _{%
\cc}\upsilon _{\cc}$ is the unit of the semiadjunction $%
\left( \upsilon _{\cc},\iota _{\cc}\right) .$ Since $%
\upsilon _{\cc}\eta _{\cc}=\upsilon _{\cc}\id%
,$ we get $\upsilon _\dd G^{\natural }\id=G\upsilon _\cc\id=G\upsilon _\cc\eta_\cc=\upsilon _\dd G^{\natural }\eta_\cc$ and hence
\begin{eqnarray*}
\upsilon _{\dd}\alpha &=&\upsilon _{\dd}\left( G^{\natural }\id\circ \alpha \right) =\upsilon _{\dd%
}G^{\natural }\id\circ
\upsilon _{\dd}\alpha =\upsilon _{\dd}G^{\natural }\eta _{\cc}\circ \upsilon _{%
\dd}\alpha =\upsilon _{\dd}\left( G^{\natural }\eta _{%
\cc}\circ \alpha \right) \\
&=& \upsilon _{\dd}\left( \alpha \iota
_{\cc}\upsilon _{\cc}\circ F^{\natural }\eta _{\cc%
}\right) =\upsilon _{\dd}\alpha \iota _{\cc}\upsilon _{%
\cc}\circ \upsilon _{\dd}F^{\natural }\eta _{\cc} =\beta \upsilon _{\cc}\circ F\upsilon _{\cc}\eta _{%
\cc}=\beta \upsilon _{\cc}\circ F\upsilon _{\cc}%
\id.
\end{eqnarray*}
We have so obtained $\upsilon _{\dd}\alpha =\beta \upsilon _{%
\cc}\circ F\upsilon _{\cc}\id$ which, evaluated on an object $(C,c)$, means that $\upsilon _{\dd}\alpha _{\left(
C,c\right) }=\beta _{C}\circ F\upsilon _{\cc}\id_{\left(
C,c\right) }=\beta _{C}\circ Fc=\upsilon _{\dd}\beta ^{\natural
}\left( C,c\right) $ so that $\alpha =\beta ^{\natural }.$
\end{invisible}

\begin{lem}\label{lem:corefsemiadj} The following assertions hold true.
  \begin{enumerate}
    \item Any functor $G$ whose completion has a left adjoint is part of a semiadjunction $(F,G)$.
    \item Any functor $F$ whose completion has a right adjoint is part of a semiadjunction $(F,G)$.
  \end{enumerate}
\end{lem}

\begin{proof}
(1) Let $G:\dd\to\cc$ be a functor whose completion $G^\natural:\dd^\natural\to\cc^\natural$ has a left adjoint $L:\cc^\natural\to\dd^\natural$. From Lemma \ref{lem:anticompl}, there exists a semifunctor $F:\cc\rightarrow\dd$ such that $F^\natural \cong L$, hence $F^\natural\dashv G^\natural$. Thus, by \cite[Theorem 3.5]{Ho93} it follows that $(F,G)$ is a semiadjunction.

(2) It is proved similarly.
\end{proof}

In \cite[Definition 1.3]{MW13}, the authors introduced the concept of ``right semiadjoint'' (resp. ``left semiadjoint'') which is a priori unrelated to the one of semiadjunction in the sense we are using here: it consists of functors $F:\cc\rightarrow \dd$ and $G:\dd\rightarrow \cc$ endowed with
natural transformations $\eta :\id_{\cc}\rightarrow GF$ and $%
\epsilon :FG\rightarrow \id_{\dd}$ such that $G\epsilon \circ \eta G=\id_{G}$ (resp. $\epsilon F \circ F\eta =\id_{F}$). The following result essentially shows how to construct a semiadjunction out of a right (left) semiadjoint\footnote{In order to avoid confusion we have not used the expression ``right (left) semiadjoint'' in the statement.}. 

\begin{lem}
\label{lem:rightsemi}Let $F:\cc\rightarrow \dd$\textbf{\ }and%
\textbf{\ }$G:\dd\rightarrow \cc$ be functors endowed with
natural transformations $\eta :\id_{\cc}\rightarrow GF$ and $%
\epsilon :FG\rightarrow \id_{\dd}$.
\begin{enumerate}
  \item If $G\epsilon \circ \eta G=\id_{G}$, then there is a
semifunctor $F^{\prime }:\cc\rightarrow \dd$, that acts as $%
F$ on objects, such that $\left( F^{\prime },G\right) $ is a semiadjunction.
  \item If $\epsilon F \circ F\eta =\id_{F}$, then there is a
semifunctor $G^{\prime }:\dd\rightarrow \cc$, that acts as $%
G$ on objects, such that $\left( F,G^{\prime }\right) $ is a semiadjunction.
\end{enumerate}
\end{lem}

\begin{proof}
We just prove (1). Set $e:=\epsilon F\circ F\eta :F\rightarrow F$. It is well-known that $e$
is idempotent, see e.g. \cite[Lemma 1.4(2)]{MW13}.
\begin{invisible}
$e\circ e=\epsilon F\circ F\eta \circ \epsilon F\circ F\eta =\epsilon F\circ
\epsilon FGF\circ FGF\eta \circ F\eta =\epsilon F\circ FG\epsilon F\circ
F\eta GF\circ F\eta =\epsilon F\circ F\eta =e.$
\end{invisible}
Let us check that there is a semifunctor $F^{\prime }:\cc%
\rightarrow \dd$ that acts as $F$ on objects and sends a morphism $%
f:X\rightarrow Y$ to $Ff\circ e_{X}$. Given $f:X\rightarrow Y$ and $g:Y\rightarrow Z$ in $\cc$ we have%
\begin{equation*}
F^{\prime }g\circ F^{\prime }f=Fg\circ e_{Y}\circ Ff\circ e_{X}=Fg\circ
Ff\circ e_{X}\circ e_{X}=F\left( g\circ f\right) \circ e_{X}=F^{\prime
}\left( g\circ f\right)
\end{equation*}%
so that $F^{\prime }$ is a semifunctor. Let us check that $\left( F^{\prime
},G,\eta ^{\prime },\epsilon ^{\prime }\right) $ is a semiadjunction where $%
\eta _{C}^{\prime }:=\eta _{C}$ and $\epsilon _{D}^{\prime }=\epsilon _{D}.$
To this aim, we first note that%
\begin{align*}
\epsilon _{X}\circ e_{GX}& =\epsilon _{X}\circ \epsilon _{FGX}\circ F\eta
_{GX}=\epsilon _{X}\circ FG\epsilon _{X}\circ F\eta _{GX}=\epsilon _{X}\circ
F\left( G\epsilon _{X}\circ \eta _{GX}\right) =\epsilon _{X}\circ F\left(
\id_{GX}\right) =\epsilon _{X}, \\
Ge_{X}\circ \eta _{X}& =G\epsilon _{FX}\circ GF\eta _{X}\circ \eta
_{X}=G\epsilon _{FX}\circ \eta _{GFX}\circ \eta _{X}=\left( G\epsilon \circ
\eta G\right) _{FX}\circ \eta _{X}=\id_{GFX}\circ \eta _{X}=\eta
_{X}.
\end{align*}
so that we get the equalities
\begin{equation}  \label{form:e1}
\epsilon \circ eG =\epsilon\qquad\text{and}\qquad Ge\circ \eta =\eta.
\end{equation}

For every object $D$ in $\dd$, we have $\epsilon _{D}^{\prime }\circ
F^{\prime }G\id_{D}=\epsilon _{D}\circ FG\id_{D}=\epsilon
_{D}=\epsilon _{D}^{\prime }$ and for every morphism $f:X\rightarrow Y$ in $%
\dd$, we have
\begin{equation*}
\epsilon _{Y}^{\prime }\circ F^{\prime }Gf=\epsilon _{Y}\circ FGf\circ
e_{GX}=f\circ \epsilon _{X}\circ e_{GX}\overset{(\ref{form:e1})}{=}f\circ
\epsilon _{X}=f\circ \epsilon _{X}^{\prime }
\end{equation*}%
so that we can define the seminatural transformation $\epsilon ^{\prime
}:=\left( \epsilon _{D}\right) _{D\in \dd}:F^{\prime }G\rightarrow
\id_{\dd}$.

For every object $C$ in $\cc$, we have $\eta _{C}^{\prime }\circ
\id_{\cc^{\prime }}\left( \id_{C}\right) =\eta
_{C}^{\prime }\circ \id_{C}=\eta _{C}^{\prime }$ and for every
morphism $f:X\rightarrow Y$ in $\cc$, we have
\begin{equation*}
GF^{\prime }f\circ \eta _{X}^{\prime }=G\left( Ff\circ e_{X}\right) \circ
\eta _{X}=G\left( e_{Y}\circ Ff\right) \circ \eta _{X}=Ge_{Y}\circ GFf\circ
\eta _{X}=Ge_{Y}\circ \eta _{Y}\circ f\overset{(\ref{form:e1})}{=}\eta
_{Y}\circ f=\eta _{Y}^{\prime }\circ f
\end{equation*}%
so that we can define the seminatural transformation $\eta ^{\prime
}:=\left( \eta _{D}\right) _{D\in \dd}:\id_{\cc%
}\rightarrow GF^{\prime }$. We compute%
\begin{equation*}
G\epsilon _{D}^{\prime }\circ \eta _{GD}^{\prime }=G\epsilon _{D}\circ \eta
_{GD}=\id_{G}
\end{equation*}%
and
\begin{equation*}
\epsilon _{F^{\prime }C}^{\prime }\circ F^{\prime }\eta _{C}^{\prime
}=\epsilon _{FC}\circ F^{\prime }\eta _{C}=\epsilon _{FC}\circ F\eta
_{C}\circ e_{C}=e_{C}\circ
e_{C}=e_{C}=F^{\prime }\id_{C}.
\end{equation*}%
Therefore $\left( F^{\prime },G,\eta ^{\prime },\epsilon ^{\prime }\right) $
is a semiadjunction.
\end{proof}

\subsection{Characterization of (co)reflections up to retracts}
Now, we provide a characterization of (co)reflections up to retracts which are semiadjoint functors. It will be applied to the quotient functor $H:\cc\to\cc_e$ in Theorem \ref{thm:H-corefl-utr}.

\begin{prop}
\label{prop:corefupretract}Let $\left( F,G,\eta ,\epsilon \right) $ be a
semiadjunction. Then,\begin{itemize}
\item[(1)] $G$ is a coreflection up to retracts if and only if there is $\nu :GF\rightarrow \id_{\cc}$ such that $\eta \circ \nu =GF\id$ and $\nu \circ \eta =\id_{\id_\cc}$.
\item[(2)] $F$ is a reflection up to retracts if and only if there is $\gamma :\id _\dd\rightarrow FG$
such that $\gamma \circ \epsilon =FG\id$ and $\epsilon \circ \gamma =\id_{\id_\dd}$.
\end{itemize}
\end{prop}

\begin{proof}
\begin{itemize}
\item[(1)] Assume there is $\nu :GF\rightarrow \id_{\cc}$ such that $\eta \circ \nu =GF\id$ and $\nu \circ \eta =\id_{\id_\cc}$. Let us prove that $\eta ^{\natural }$ is an isomorphism with inverse $\nu
^{\natural }$ defined by $\nu _{\left( C,c\right) }^{\natural }:=c\circ \nu
_{C}$ so that $F^{\natural }$ is fully faithful, i.e. $G$ is a coreflection
up to retracts. Note that $c\circ \left( c\circ \nu _{C}\right) \circ
GFc=c\circ c\circ \nu _{C}\circ GFc=c\circ c\circ c\circ \nu _{C}=c\circ \nu
_{C}$ and hence we get the morphism $\nu _{\left( C,c\right) }^{\natural
}:\left( GFC,GFc\right) \rightarrow \left( C,c\right) $. We compute%
\begin{align*}
\eta _{\left( C,c\right) }^{\natural }\circ \nu _{\left( C,c\right)
}^{\natural } &=\eta _{C}\circ c\circ c\circ \nu _{C}=\eta _{C}\circ c\circ
\nu _{C}=GFc\circ \eta _{C}\circ \nu _{C}=GFc\circ GF\id=GFc=\mathrm{%
Id}_{\left( GFC,GFc\right) }, \\
\nu _{\left( C,c\right) }^{\natural }\circ \eta _{\left( C,c\right)
}^{\natural } &= c\circ \nu _{C}\circ \eta _{C}\circ c=c\circ \id%
_{C}\circ c=c\circ c=c=\id_{\left( C,c\right) }
\end{align*}%
so that $\eta _{\left( C,c\right) }^{\natural }$ is an isomorphism in $%
\cc^{\natural }$. Conversely, assume that $G$ is a coreflection up to retracts. Then $G^\natural$ has a left adjoint $F^{\natural }$ which is fully faithful, so the unit $\eta ^{\natural }:\id _{\cc^\natural}\to G^\natural F^\natural$ of the adjunction $(F^\natural ,G^\natural ,\eta^\natural ,\epsilon^\natural )$ is an isomorphism. 
By Lemma \ref{lem:anticompl}, there exists a seminatural transformation $\nu:GF\to\id _\cc$ such that $\nu^\natural=(\eta ^{\natural })^{-1}$. Thus we have $(\eta\circ\nu )^\natural=\eta^\natural\circ\nu^\natural =\id _{G^\natural F^\natural} =(GF\id  )^{\natural}$ and $(\nu\circ\eta )^\natural=\nu^\natural\circ\eta^\natural = \id _{\id _{\cc^\natural}}=(\id _{\id _\cc})^\natural$, hence by \cite[Lemma 23]{Ho90} it follows that $\eta\circ\nu =GF\id $ and $\nu\circ\eta =\id _{\id _\cc}$, respectively. 
\item[(2)] The proof follows by the same arguments. \qedhere
\begin{invisible}
We prove that $\epsilon ^{\natural }$ is an isomorphism with inverse $\gamma
^{\natural }$ defined by $\gamma _{\left( D,d\right) }^{\natural }:= \gamma_D\circ d$ so that $G^{\natural }$ is fully faithful i.e. $F$ is a reflection
up to retracts. Since $FGd\circ \left( \gamma_D\circ d\right) \circ d=FGd\circ \gamma_D\circ d \circ d=\gamma_D\circ d\circ d\circ d=\gamma_D\circ d$, we get the morphism $\gamma_{\left( D,d\right) }^{\natural
}:\left( D,d\right) \rightarrow \left( FGD, FGd\right) $. We compute%
\begin{eqnarray*}
\gamma _{\left( D,d\right) }^{\natural }\circ \epsilon _{\left( D,d\right)
}^{\natural } &=&\gamma _{D}\circ d\circ d\circ \epsilon _{D}=\gamma _{D}\circ d\circ
\epsilon _{D}=\gamma_D\circ \epsilon_{D}\circ FGd=FG\id \circ FGd=FGd=\mathrm{%
Id}_{\left( FGD,FGd\right) }, \\
\epsilon _{\left( D,d\right) }^{\natural }\circ \gamma _{\left( D,d\right)
}^{\natural } &=&d\circ \epsilon _{D}\circ \gamma _{D}\circ d=d\circ \id%
_{D}\circ d=d\circ d=d=\id_{\left( D,d\right) }
\end{eqnarray*}%
so that $\epsilon _{\left( D,d\right) }^{\natural }$ is an isomorphism in $%
\dd^{\natural }$. Conversely, assume that $F$ is a reflection up to retracts. Then $F^\natural$ has a right adjoint $G^{\natural }$ which is fully faithful, so the counit $\epsilon ^{\natural }:F^\natural G^\natural\to\id _{\dd^\natural}$ of the adjunction $(F^\natural ,G^\natural ,\eta^\natural ,\epsilon^\natural )$ is an isomorphism with inverse $\gamma^{\natural }:\id _{\dd^\natural}\to F^\natural G^\natural$. 
Moreover, since there exists a seminatural transformation $\gamma:\id _\dd\to FG$ such that $(\gamma )^\natural =\gamma^\natural$, we have $(\gamma\circ\epsilon )^\natural=\gamma^\natural\circ\epsilon^\natural =\id _{F^\natural G^\natural} =(FG\id  )^{\natural}$ and $(\epsilon\circ\gamma )^\natural=\epsilon^\natural\circ\gamma^\natural = \id _{\id _{\dd^\natural}}=(\id _{\id _\dd})^\natural$, hence by \cite[Lemma 23]{Ho90} it follows that $\gamma\circ\epsilon =FG\id $ and $\epsilon\circ\gamma =\id _{\id _\dd}$, respectively.
\end{invisible}
\end{itemize}
\end{proof}

Proposition \ref{prop:corefupretract} allows us to characterize a (co)reflection up to retracts as part of a semiadjunction as follows.

\begin{cor}[Characterization of (co)reflections up to retracts]\label{cor:charactutr} Let $\cc$ and $\dd$ be categories. 
\begin{enumerate}
  \item A functor $G:\dd\to\cc$ is a coreflection up to retracts if and only if it is part of a semiadjunction $(F,G,\eta,\epsilon)$ and there is $\nu :GF\rightarrow \id_{\cc}$ such that $\eta \circ \nu =GF\id$ and $\nu \circ \eta =\id_{\id_\cc}$.
 \item A functor $F:\cc\to\dd$ is a reflection  up to retracts if and only if it is part of a semiadjunction $(F,G,\eta,\epsilon)$ and there is $\gamma :\id _\dd\rightarrow FG$
such that $\gamma \circ \epsilon =FG\id$ and $\epsilon \circ \gamma =\id_{\id_\dd}$.
\end{enumerate}
\end{cor}
\begin{proof} We prove (1), the proof of (2) being similar. In view of  Proposition \ref{prop:corefupretract}, it suffices to check that a coreflection up to retracts $G:\dd\to\cc$ is always part of a semiadjunction $(F,G,\eta,\epsilon)$. In fact for such a $G$, the completion $G^\natural$ has a fully faithful left adjoint and we conclude by Lemma \ref{lem:corefsemiadj}.
\end{proof}

The following result is a consequence of Corollary \ref{cor:charactutr}.

\begin{cor}\label{cor:surjutr-coreflutr}
Any (co)reflection up to retracts is surjective up to retracts.
\end{cor}

\begin{proof}
Let $G:\dd\rightarrow \cc$ be a coreflection up to retracts. By Corollary \ref{cor:charactutr} (1), $G$ is part of a semiadjunction $\left( F,G,\eta ,\epsilon \right)$ and there is $\nu :GF\rightarrow \id_{\cc}$ such that $\nu \circ \eta =\id_{\id_\cc}$. Given
an object $C$ in $\cc$ we get $\nu_{C} \circ \eta_{C} =\id_{C}$ and hence $C$ is a retract of $GFC$, i.e. $G$ is surjective up to retracts. 
Similarly, any reflection up to retracts is surjective up to retracts by Corollary \ref{cor:charactutr} (2). \begin{invisible} Let $F:\cc\rightarrow \dd$ be a reflection up to retracts. By Corollary \ref{cor:charactutr} (2), it is part of a semiadjunction $(F,G,\eta,\epsilon)$ and there is $\gamma :\id _\dd\rightarrow FG$ such that $\epsilon \circ \gamma =\id_{\id_\dd}$. Given an object $D$ in $\dd$ we get that $\epsilon_{D} \circ \gamma_{D} =\id_{D}$  and hence $D$ is a retract of $FGD$ i.e. $F$ is surjective up to retracts.
\end{invisible}	
%
\end{proof}


Now we give further conditions for a functor to be a (co)reflection up to retracts.  We will apply it in the next section to study the (co)comparison functor attached to an adjunction.

\begin{prop}
\label{prop:trinat}Let $F:\cc\rightarrow \dd$\textbf{\ }and%
\textbf{\ }$G:\dd\rightarrow \cc$ be functors endowed with
natural transformations $\eta :\id_{\cc}\rightarrow GF$ and $%
\epsilon :FG\rightarrow \id_{\dd}$.
\begin{enumerate}
  \item If there is a natural transformation $\nu :GF\rightarrow \mathrm{Id%
}_{\cc}$ such that $\nu \circ \eta =\id$ and $\nu
G=G\epsilon $, then $G$ is a coreflection up to retracts.
  \item If there is a natural transformation $\gamma :\mathrm{Id%
}_{\dd}\to FG$ such that $\epsilon \circ \gamma =\id$ and $\gamma
F=F\eta ,$ then $F$ is a reflection up to retracts.
\end{enumerate}
\end{prop}

\begin{proof} We just prove (1). Given $\nu $ as in the statement, note that $G\epsilon \circ \eta G=\nu
G\circ \eta G=\left( \nu \circ \eta \right) G=\id_{G}$ so that we
are in the setting of Lemma \ref{lem:rightsemi}.  For any $C$ in $\cc$ define $\nu _{C}^{\prime }:=\nu _{C}\circ Ge_{C}$, where $e:=\epsilon F\circ F\eta$. Then $\nu
_{C}^{\prime }\circ GF^{\prime }\left( \id_{C}\right) =\nu _{C}\circ
Ge_{C}\circ Ge_{C}=\nu _{C}\circ Ge_{C}=\nu _{C}^{\prime }$ and for every
morphism $f:X\rightarrow Y$ in $\cc$, we have%
\begin{equation*}
\nu _{Y}^{\prime }\circ GF^{\prime }f=\nu _{Y}\circ Ge_{Y}\circ GFf\circ
Ge_{X}=\nu _{Y}\circ GFf\circ Ge_{X}\circ Ge_{X}=f\circ \nu _{X}\circ
Ge_{X}=f\circ \nu _{X}^{\prime }
\end{equation*}%
so that we can define the seminatural transformation $\nu ^{\prime
}:=\left( \nu _{C}^{\prime }\right) _{C\in \cc}:GF^{\prime
}\rightarrow \id_{\cc}$. We compute%
\begin{align*}
\nu _{C}^{\prime }\circ \eta _{C}^{\prime } &=\nu _{C}\circ Ge_{C}\circ
\eta _{C}\overset{(\ref{form:e1})}{=}\nu _{C}\circ \eta _{C}=\id_{C},
\\
\eta _{C}^{\prime }\circ \nu _{C}^{\prime } &=\eta _{C}\circ \nu _{C}\circ
Ge_{C}\overset{\text{nat.}\nu }{=}\nu _{GFC}\circ GF\eta _{C}\circ Ge_{C} \\
&=G\epsilon _{FC}\circ GF\eta _{C}\circ Ge_{C}=Ge_{C}\circ
Ge_{C}=GF^{\prime }\id_{C}.
\end{align*}%
By Proposition \ref{prop:corefupretract}, we conclude.
\end{proof}

\section{Quotient and (co)comparison functor}\label{se:quotcom}
This section collects the fall-outs of the results we achieved so far. First we prove that the quotient functor $H:\cc\rightarrow \cc_{e}$ onto the coidentifier category is always a coreflection up to retracts. Then also  the (co)comparison functor attached to an adjunction whose associated (co)monad is (co)separable is shown to be a coreflection (reflection) up to retracts. This result allows  to characterize a semiseparable right (left) adjoint in terms of (co)separability of the associated (co)monad and the requirement that the (co)comparison functor is a bireflection up to retracts. To complete the picture, we study the (semi)separability of a pair of functors whose source categories are not idempotent complete, namely the free induction functor and the free restriction of scalars functor.

In Subsection \ref{sub:compfact} we compare the two canonical factorizations we have attached to a semiseparable right adjoint $G:\dd\to\cc$, namely the one through the coidentifier category  and the one through the comparison functor, showing they are connected by an equivalence up to retracts. 

In Subsection \ref{sub:kleisli}, we show that in presence of a separable monad, the associated Kleisli category and Eilenberg-Moore category have  equivalent idempotent completions. Moreover, given a semiseparable right adjoint $G:\dd\to\cc$ these idempotent completions result to be equivalent to the idempotent completion of $\dd_e$, where $e$ is the idempotent natural transformation associated to $G$.

In Subsection \ref{sub:pretriang} we apply the foregoing achievements to  obtain a semi-analogue of a result due to P. Balmer concerning pre-triangulated categories. Finally, we provide conditions for the Eilenberg-Moore category $\cc_{GF}$ to inherit the pre-triangulation from the base category $\cc$.

\medskip
\noindent \textbf{The quotient functor.} We start by proving that the quotient functor $H:\cc\rightarrow \cc_{e}$ of Subsection \ref{sub:coidentifier} is a coreflection up to retracts. Since we know that $H$ is naturally full (as recalled in Subsection \ref{sub:coidentifier}), it reveals to be indeed a bireflection up to retracts.

\begin{thm}\label{thm:H-corefl-utr}
Let $\cc$ be a category, let $e:\id_{\cc}\rightarrow \id_{\cc}$ be an idempotent natural transformation. Then, the quotient functor $H:\cc\rightarrow \cc_{e}$ is a coreflection up to retracts whence a bireflection up to retracts.
\end{thm}
\proof
Define the semifunctor $L:\cc_e\to\cc$  as the identity on objects and by $(\bar{f}:X\to Y)\mapsto (e_Y\circ f:X\to Y)$ on morphisms. Note that it is really a semifunctor as $L\overline{\id }_X =e_X\circ\id_X=e_X\neq \id _{LX}$ in general. Moreover, it is well-defined as $\bar{f}=\bar{g}$ if and only if $e_Y\circ f =e_Y\circ g$. Now we show that $(L,H)$ is a semiadjunction with unit $\eta : \id_{\cc_e}\to HL$, $\eta_X =\overline{\id }_X : X\to HLX=X$, and counit $\epsilon : LH\to \id _{\cc}$, $\epsilon_Y:=e_Y : LHY=Y\to Y$. First, observe that $\eta$ and $\epsilon$ are seminatural transformations. Indeed, for every $\bar{f}:X\to Y$ in $\cc_e$, we have $HL\bar{f}\circ\eta_X = H(e_Y\circ f)\circ \overline{\id }_X = He_Y\circ Hf\circ H\id_X =\id_{HY}\circ Hf\circ\id_{HX}=\overline{\id}_Y\circ \bar{f}=\eta_Y\circ \bar{f}$, hence in particular $HL\overline{\id}_X\circ\eta_X =\eta_X\circ\overline{\id}_X=\eta_X$, thus $\eta$ is a seminatural transformation. The same holds for $\epsilon$, as $\epsilon_Y\circ LHf =e_Y\circ L\bar{f}=e_Y\circ e_Y\circ f=e_Y\circ f=f\circ e_X=f\circ \epsilon_X$ and in particular $\epsilon_Y\circ LH\id_Y=\id_Y\circ\epsilon_Y=\epsilon_Y$. Moreover, for every $X\in\cc$ and $Y\in\cc_e$ we have the identities $\epsilon_{LX}\circ L\eta_X =e_{LX}\circ L\overline{\id}_X=e_X\circ L\overline{\id}_X =e_X\circ e_X\circ\id_X=e_X\circ\id_X=L\overline{\id}_X$ and $H\epsilon_Y\circ\eta_{HY}=He_Y\circ\id_{HY}=He_Y=\id_{HY}=H\id_Y$. So $(L,H,\eta ,\epsilon )$ is a semiadjunction. Since for every object $X\in\cc_e$, $HL(X)=X$, and for every morphism $\bar{f}$ in $\cc_e$, $HL\bar{f}=H(e_Y\circ f)=He_Y\circ Hf=\id_{HY}\circ\bar{f}=\bar{f}$, we have $HL=\id_{\cc_e}$, and thus $\eta =\id_{\id_{\cc_e}}$, hence there exists $\nu=\id_{\id_{\cc_e}}:HL\to\id_{\cc_e}$ such that $\eta \circ \nu =\id_{\id_{\cc_e}}=HL\id$ and $\nu \circ \eta =\id_{\id_{\cc_e}}$. By Proposition \ref{prop:corefupretract} $H:\cc\rightarrow \cc_{e}$ is a coreflection up to retracts. Since $H$ is also naturally full, then, by Lemma \ref{lem:trivfactsutr}, $H$ is bireflection up to retracts.
\endproof

\begin{rmk}\label{rmk:Hnotbiref}
The functor $H:\cc\rightarrow \cc_{e}$ is a bireflection if and only if the idempotent natural transformation $e:\id_{\cc}\rightarrow \id_{\cc}$ splits, see \cite[Proposition 2.27]{AB22}. Thus, in general it is a bireflection up to retracts but not a bireflection.
\end{rmk}

\begin{es}\label{exa:proj} Let $R$ be a ring and let $R$-$\mathrm{Mod}$ be the category of left $R$-modules. Denote by $R\text{-}\mathrm{Mod}_f$ and $R\text{-}\mathrm{Proj}$ the full subcategories of $R$-$\mathrm{Mod}$ whose objects are free left $R$-modules and projective left $R$-modules, respectively. 
Let $\Psi:R\text{-}\mathrm{Mod}_f\to R\text{-}\mathrm{Proj}$ be the inclusion functor. It is an equivalence up to retracts as it is fully faithful and any projective module is a retract of a free module, cf. Lemma \ref{lem:trivfactsutr}\,(6). As a consequence, by \cite[Theorem 6.12, page 30]{Kar78}, the functor $\Psi$ induces an equivalence $\Psi': R\text{-}\mathrm{Mod}_f^\natural\to R\text{-}\mathrm{Proj},\,(F,e)\mapsto \mathrm{Im}(e)$. This fact is well-known and, in the finitely generated case, it is written explicitly in \cite[Theorem 6.16]{Kar78}. 
\begin{invisible}
 However, since we did not find a proof of this fact in the general case in the literature, we include it here for the reader's sake.
Explicitly, let $\Psi: \cc^\natural\to R\text{-}\mathrm{Proj}$ be the functor defined by $(F,e)\mapsto \mathrm{Im}(e)$, where $F$ is a free left $R$-module and $e:F\to F$ is an idempotent (which splits as $e=i\circ p$, where $p:F\to\mathrm{Im}(e)$ is the canonical projection and the inclusion map $i:\mathrm{Im}(e)\to F$ is the induced section, i.e. $p\circ i=\id_{\mathrm{Im}(e)}$); moreover, $\Psi$ assigns to a given morphism $f:(F,e)\to (F',e')$ in $\cc^\natural$ the composite morphism $\Psi(f)=p'\circ f\circ i:\mathrm{Im}(e)\to\mathrm{Im}(e')$ of projective left $R$-modules, where the idempotent arrow $e':F'\to F'$ splits as the canonical projection $p':F'\to\mathrm{Im}(e')$ followed by the inclusion $i':\mathrm{Im}(e')\to F'$. It results that $\Psi$ is an equivalence of categories. Indeed, given a morphism $h:\mathrm{Im}(e)\to\mathrm{Im}(e')$, we set $f:=i'hp$. Then $e'fe=i'p'i'hpip=i'hp=f$ so that we get a morphism $f:(F,e)\to (F',e')$ in $\cc^\natural$. Moreover $h=p'i'hpi=p'fi=\Psi(f)$ so that $\Psi$ is full. If $f,g:(F,e)\to (F',e')$ are such that $\Psi(f)=\Psi(g)$, then $p'fi=p'gi$ and hence $f=e'fe=i'p'fip=i'p'gip=e'ge=g$ so that $\Psi$ is faithful. Moreover given $P$ projective, the canonical projection $\pi:R^{(P)}\to P,\; (r_p)_{p\in P}\mapsto \sum_{p\in P} r_p p$ splits via some morphism $\sigma :P\to R^{(P)}$ as $P$ is projective. Then $\Psi(R^{(P)},\sigma \pi)=\mathrm{Im}(\sigma \pi)=\mathrm{Im}(\sigma )\cong P$ so that $\Psi$ is essentially surjective. This proves that $\Psi$ is effectively an equivalence of categories, cf. \cite[Theorem 1, page 93]{Mac98}.   
\end{invisible}

Now set $\cc:=R\text{-}\mathrm{Mod}_f$. Given a central idempotent element $z\in R$, with $z\neq 0,1$, define the idempotent natural transformation $e:\id_{\cc}\rightarrow \id_{\cc}$ by setting $e_M:M\to M,m\mapsto zm$, for every free left $R$-module $M$. If $e$ splitted, then $e_R:R\to R$ would split in $\cc$ and thus $zR=\im(e_R)$ would be a free $R$-module.
\begin{invisible}
If $e_R$ splits in $\cc$, then there are morphisms of $R$-modules $p:R\to X,i:X\to R$, with $X\in\cc$, such that $e_R=i\circ p$ and $p\circ i=\id_X$. If we denote by $j:\im(e_R)\to R$ the canonical injection, we get that $p\circ j:\im(e_R)\to X$ is an isomorphism with inverse $X\to \im(e_R),x\mapsto i(x)$ (each $x\in X$ can be written as $p(r)$ for some $r$ so that $i(x)=i(p(r))=e_R(r)\in \im(e_R)$). Thus $\im(e_R)\cong X\in\cc$.
\end{invisible}
Since $0\neq z\in zR$, we have $zR\neq 0$ and it is known that a nonzero free module is faithful, i.e. it has trivial annihilator.
\begin{invisible}
If $M\neq 0$ is a free module, we have a non zero basis element $m\in M$. Given $r\in \mathrm{Ann}_R(M)$, we get $rm=0$. By linear independence of $m$ we obtain $r=0$. Thus $\mathrm{Ann}_R(M)=0.$
\end{invisible}
Hence $1-z\in\mathrm{Ann}_R(zR)=0$ and so $z=1$, a contradiction.
Therefore $e$ does not split and hence $H:\cc\rightarrow \cc_{e}$ is a bireflection up to retracts but not a bireflection in view of Remark \ref{rmk:Hnotbiref}. For example, take $R=\RR\times\RR$ and $z=(1,0)$.
\end{es}

\noindent \textbf{The (co)comparison functor.}
Now we move our attention to the (co)comparison functor attached to an adjunction.
\begin{thm}\label{thm:monutr}
Let $F\dashv G:\dd\rightarrow \cc$ be an adjunction with unit $\eta $ and counit $\epsilon $.
\begin{enumerate}
  \item If the monad $\left( GF,G\epsilon F,\eta \right) $ is separable,
then the comparison functor $K_{GF}:\dd\rightarrow \cc_{GF}$ is a coreflection up to retracts.
  \item If the comonad $(FG, F\eta G, \epsilon)$ is coseparable, then the cocomparison functor $K^{FG}:\cc\to \dd^{FG}$ is a reflection up to retracts.
\end{enumerate}
\end{thm}

\begin{proof} We just check (1).
Set $K:=K_{GF}:\dd\rightarrow \cc_{GF},$ $U:=U_{GF}:\mathcal{%
C}_{GF}\rightarrow \cc,$ $V:=V_{GF}:\cc\rightarrow \mathcal{C%
}_{GF}$ and consider $\Lambda :=FU:\cc_{GF}\rightarrow \dd$.
Let us construct three natural transformations $\eta _{1}:\id_{%
\cc_{GF}}\rightarrow K\Lambda ,$ $\epsilon _{1}:\Lambda K\rightarrow
\id_{\dd}$ and $\nu _{1}:K\Lambda \rightarrow \id_{%
\cc_{GF}}$ that fulfill the requirements of Proposition \ref{prop:trinat},
i.e. such that $\nu _{1}\circ \eta _{1}=\id$ and $\nu
_{1}K=K\epsilon _{1}.$
Since $\Lambda K=FUK=FG$ it makes sense to define $\epsilon _{1}:=\epsilon $%
, the counit of the adjunction $\left( F,G\right) $. Since $K\Lambda =KFU=VU$
we can set $\nu _{1}:=\beta ,$ the counit of the adjunction $\left(
V,U\right) $, which is defined by $U\beta _{\left( C,\mu \right) }=\mu $ for
every object $\left( C,\mu \right) $ in $\cc_{GF}$.

Since the monad $\left( GF,G\epsilon F,\eta \right) $ is separable, then the
functor $U$ is separable and hence, by Rafael Theorem, there is a natural
transformation $\eta _{1}:\id_{\cc_{GF}}\rightarrow VU$ such
that $\beta \circ \eta _{1}=\id$, i.e. $\nu _{1}\circ \eta _{1}=%
\id$.

Moreover $U\beta _{KD}=U\beta _{\left( GD,G\epsilon D\right) }=G\epsilon
D=UK\epsilon _{1}D$ so that $\beta K=K\epsilon _{1}$, i.e. $\nu
_{1}K=K\epsilon _{1}$.
\end{proof}

Theorem \ref{thm:monutr} allows to obtain the following characterization improving Theorem \ref{thm:ssepMonad}.

\begin{thm}\label{thm:computr}Let $F\dashv G:\dd\rightarrow \cc$ be an adjunction with unit $\eta $ and counit $\epsilon $.
\begin{enumerate}
\item $G$ is semiseparable if and only if the monad $\left( GF,G\epsilon F,\eta \right) $ is separable and the comparison functor $K_{GF}:\dd\rightarrow \cc_{GF}$ is a bireflection up to retracts.
\item $F$ is semiseparable  if and only if the comonad $(FG, F\eta G, \epsilon)$ is coseparable and the cocomparison functor $K^{FG}:\cc\to \dd^{FG}$ is a bireflection up to retracts.
\end{enumerate}
\end{thm}
\begin{proof}
We just prove (1). By Theorem \ref{thm:ssepMonad}, $G$ is semiseparable if and only if the monad $\left( GF,G\epsilon F,\eta \right) $ is separable and $K_{GF}$ is a naturally full. When $\left( GF,G\epsilon F,\eta \right) $ is separable, $K_{GF}$ is  a coreflection up to retracts by Theorem \ref{thm:monutr}, and hence it is naturally full if and only it it is a naturally full coreflection up to retracts if and only if it is a bireflection up to retracts by Lemma \ref{lem:trivfactsutr}.
\end{proof}

Theorem \ref{thm:computr} allows to retrieve the following characterization improving Corollary \ref{cor:sepmonad}.

\begin{cor}\label{cor:compcorefretr} Let $F\dashv G:\dd\rightarrow \cc$ be an adjunction with unit $\eta $ and counit $\epsilon $.

\begin{enumerate}
\item \cite[Proposition 3.5]{Chen15} $G$ is separable if and only if the monad $\left( GF,G\epsilon F,\eta \right) $ is separable and the comparison functor $K_{GF}:\dd\rightarrow \cc_{GF}$ is an equivalence up to retracts.
\item  \cite[Proposition 2.3]{Sun19} $F$ is separable if and only if the comonad $(FG, F\eta G, \epsilon)$ is coseparable and the cocomparison functor $K^{FG}:\cc\to \dd^{FG}$ is an equivalence up to retracts.
\end{enumerate}
\end{cor}

\begin{proof}
We just prove (1). By Proposition \ref{prop:sep}, $G$ is separable if and only if it is semiseparable and faithful. By Theorem \ref{thm:computr}, $G$ is semiseparable if and only if the monad $\left( GF,G\epsilon F,\eta \right) $ is separable and $K_{GF}$ is a bireflection up to retracts. On the other hand, since $G=U_{GF}\circ K_{GF}$ and $U_{GF}$ is faithful, we have that $G$ is faithful if and only if so is $K_{GF}$. Summing up $G$ is separable if and only if $\left( GF,G\epsilon F,\eta \right) $ is separable and $K_{GF}$ is a fully faithful bireflection up to retracts. By Lemma \ref{lem:trivfactsutr}, the latter requirements on  $K_{GF}$ means it is an equivalence up to retracts.
\end{proof}

The special features we proved for coreflections up to retracts yield the following result. 

\begin{cor}\label{cor:compsemi}
	Let $F\dashv G:\dd\to\cc$ be an adjunction with comparison functor $K_{GF}:\dd\to\cc_{GF}$ and cocomparison functor $K^{FG}:\cc\to \dd^{FG}$.
	\begin{enumerate}
		\item Assume $G$ is semiseparable. If $K_{GF}$ has a left adjoint, then $K_{GF}$ is a bireflection.
		\item Assume $F$ is semiseparable. If $K^{FG}$ has a right adjoint, then $K^{FG}$ is a bireflection.
		\item (cf. \cite[Proposition, page 93]{Par71} and \cite[Proposition 2.16(3)]{AGM15}) Assume $G$ is separable. If $K_{GF}$ has a left adjoint, then $K_{GF}$ is an equivalence (i.e. $G$ is monadic)
		\item (cf. \cite[Proposition 3.16]{Mes06}) Assume $F$ is separable. If $K^{FG}$ has a right adjoint, then $K^{FG}$ is an equivalence (i.e. $F$ is comonadic).
	\end{enumerate}
	In case $\dd$ (resp. $\cc$) is idempotent complete, if $G$ (resp. $F$) is (semi)separable, then $K_{GF}$ (resp. $K^{FG}$) has a left (resp. right) adjoint so the previous assertions apply.
\end{cor}

\begin{proof}
We just prove (1) and (3). If $G$ is semiseparable (resp. separable), by Theorem \ref{thm:computr} (resp. Corollary \ref{cor:compcorefretr}) we know that $K_{GF}$ is a bireflection (resp. equivalence) up to retracts. Then, if $K_{GF}$ has a left adjoint, by Proposition \ref{prop:corefl} $K_{GF}$ is a bireflection (resp.  equivalence).
 By Proposition \ref{prop:idpcom-cutr}, if $\dd$ is idempotent complete, then $K_{GF}$ has a left adjoint as it is a bireflection (resp. equivalence) up to retracts.
\begin{invisible}
We prove (2) and (4). If $F$ is semiseparable (resp. separable), by Theorem \ref{thm:computr} (resp. Corollary \ref{cor:compcorefretr}) we know that $K^{FG}$ is a bireflection (resp. equivalence) up to retracts. Then, if $K^{FG}$ has a left adjoint, by Proposition \ref{prop:corefl} $K^{FG}$ is a bireflection (resp. an equivalence).
 By Proposition \ref{prop:idpcom-cutr}, if $\dd$ is idempotent complete, then $K^{FG}$ has a left adjoint as it is a bireflection (resp. equivalence) up to retracts.
\end{invisible}
\end{proof}

\begin{invisible} Corollari non usati.

\begin{cor}  Assume that $F\dashv G:\dd\to\cc$ is an adjunction. Then the following are equivalent.
\begin{enumerate}
  \item $G$ is a bireflection and $K_{GF}:\dd\to\cc_{GF}$ has a left adjoint.
  \item $K_{GF}:\dd\to\cc_{GF}$ is naturally full and $U_{GF}:\cc_{GF}\to \cc$ is an equivalence.
\end{enumerate}
\end{cor}
\begin{proof}
$(1)\Rightarrow (2)$. Since $G$ is a bireflection, by Theorem \ref{thm:frobenius}, it is in particular semiseparable and a coreflection. By Corollary \ref{cor:compsemi}, we have that $K_{GF}$ is a bireflection whence naturally full and a coreflection by Theorem \ref{thm:frobenius}. Since $U_{GF}$ is conservative (\cite[Proposition 4.1.4]{BorII94}) and both $G$ and $K_{GF}$ are in particular coreflections, in view of the equality $U_{GF}\circ K_{GF}=G$ and by Lemma \ref{lem:Berger}, we get that $U_{GF}$ is an equivalence.

$(2)\Rightarrow (1)$. Since $K_{GF}$ is naturally full and $U_{GF}$ is an equivalence, then, by Theorem \ref{thm:ssepMonad}, $G$ is semiseparable. Since $U_{GF}:\cc_{GF}\to \cc$ is an equivalence and $G$ has a left adjoint, from the equality $U_{GF}\circ K_{GF}=G$ we get that also $K_{GF}$ has a left adjoint. Hence, by Corollary \ref{cor:compsemi}, we have that $K_{GF}$ is a bireflection so that, from the equality $U_{GF}\circ K_{GF}=G$ and the fact that $U_{GF}$ is an equivalence, we deduce that $G$ is a bireflection.
\end{proof}

\begin{cor}  Assume that $F\dashv G:\dd\to\cc$ is an adjunction. Then the following are equivalent.
\begin{enumerate}
  \item $F$ is a bireflection and $K^{FG}:\cc\to \dd^{FG}$ has a right adjoint.
  \item $K^{FG}:\cc\to \dd^{FG}$ is naturally full and $U^{FG}:\dd^{FG}\to\dd$ is an equivalence.
\end{enumerate}
\end{cor}
\end{invisible}

What follows is an example of a coreflection (up to retracts) which is not an equivalence (up to retracts) and not even a bireflection (up to retracts).
\begin{es}
Consider the forgetful functor $G:\Top\to\Set$ and its left adjoint $F:\Set\to\Top$ which assigns to each set $X$ the topological space $X$ equipped with the discrete topology (all subsets of $X$ are open), see \cite[page 144]{Mac98}. This adjunction defines on $\Set$ the identity monad $\mathbb{I}=(\id_\Set,\id,\id)$. The Eilenberg-Moore category of modules over $\mathbb{I}$ is then $\Set$, thus the comparison functor $K_{GF}:\Top\to\Set_\mathbb{I}=\Set$ is the given forgetful functor $G$. Note that the identity monad $\mathbb{I}$ is separable, thus by Theorem \ref{thm:monutr} $K_{GF}$ is a coreflection up to retracts and then a coreflection either by Proposition \ref{prop:idpcom-cutr}, as $\Top$ is an idempotent complete category (it has in fact equalizers, see \cite[Theorem 2.15]{Ho93}), or by Proposition \ref{prop:corefl}, as $K_{GF}=G$ has a left adjoint. Since $K_{GF}$ is not an equivalence, again by Proposition \ref{prop:idpcom-cutr} it follows that $K_{GF}$ is not even an equivalence up to retracts. By Corollary \ref{cor:compcorefretr} we have that $G$ is not separable and, since $G$ is faithful, $G$ is not semiseparable by Proposition \ref{prop:sep}. Then, by Theorem \ref{thm:computr} $K_{GF}$ is not even a bireflection up to retracts, and hence not a bireflection by Proposition \ref{prop:idpcom-cutr}. 
\end{es}

Next aim is to exhibit examples of (semi)separable adjoints to whom Theorem \ref{thm:computr} and Corollary \ref{cor:compcorefretr} apply even if the relevant categories are not idempotent complete, namely the free induction functor and the free restriction of scalars functor.\medskip

\noindent\textbf{The free induction and restriction functors.} 
In order to study the (semi)separability of the free induction functor and of the free restriction of scalars functor, we will use the following lemma, inspired by \cite[Lemma 2.9]{AM16}.
\begin{lem}\label{lem:subadjunction}
	Let $F\dashv G:\cc\to\dd$ be an adjunction of functors and let $S:\cc'\to\cc$ and $T:\dd'\to\dd$ be fully faithful functors. Assume that there exist functors $F':\dd'\to\cc'$ and $G':\cc'\to\dd'$ such that both squares
	\begin{gather*}\label{eq:STff}
		\vcenter{
			\xymatrixcolsep{1cm}\xymatrixrowsep{.6cm}
			\xymatrix{\cc'\ar[r]^-{S}
				\ar@<1ex>[d]^*-<0.1cm>{^{G'}}& \cc \ar@<1ex>[d]^*-<0.1cm>{^{G}}
				\\
				\dd'\ar[r]^{T} \ar@<1ex>[u]^*-<0.1cm>{^{F'}}&\dd \ar@<1ex>[u]^*-<0.1cm>{^{F}}\ar@{}[u]|{\dashv}
		} }
	\end{gather*}
	are commutative, i.e. $F\circ T=S\circ F'$ and $T\circ G'=G\circ S$. Then, $(F',G')$ is an adjunction in a unique way such that the pair of functors $(S,T)$ is a map of adjunctions in the sense of \cite[IV.7]{Mac98}. 
\\ Moreover, if $G$ (respectively, $F$) is (semi)separable, then also $G'$ (respectively, $F'$) is (semi)separable. 
\end{lem}

\begin{proof}
Consider $D'\in\dd'$, $C'\in\cc'$.  The composition of natural isomorphisms yields the natural isomorphism $\varphi_{D',C'}:=(\mathcal{F}^{T}_{D',G'C'})^{-1}\circ \varphi_{TD',SC'} \circ \mathcal{F}^{S}_{F'D',C'}$. By construction the diagram
\[\xymatrix{\Hom_{\cc'}(F'D',C')\ar[rr]^{\mathcal{F}^{S}_{F'D',C'}}\ar@{.>}[d]^{\varphi_{D',C'}}&& \Hom_{\cc}(SF'D',SC')\ar@{=}[r]&\Hom_{\cc}(FTD', SC')\ar[d]^{\varphi_{TD',SC'}}\\
 \Hom_{\dd'}(D',G'C')\ar[rr]^{\mathcal{F}^{T}_{D',G'C'}}&&\Hom_{\dd}(TD',TG'C')\ar@{=}[r]&  \Hom_\dd(TD', GSC')}\] 
 commutes and this means that the pair of functors $(S,T)$ is a map of adjunctions.
\begin{invisible}	
 Consider $D'\in\dd'$, $C'\in\cc'$. Then, since in the commutative square \eqref{eq:STff} $S$, $T$ are fully faithful and $F\dashv G$ are adjoint functors, we get $\Hom_{\cc'}(F'D',C')\cong\Hom_{\cc}(SF'D',SC')=\Hom_\cc(F(TD'), SC')\cong\Hom_\dd(TD', GSC')=\Hom_\dd(TD',TG'C')\cong\Hom_{\dd'}(D',G'C')$, thus $$\Hom_{\cc'}(F'D',C')\cong\Hom_{\dd'}(D',G'C').$$ Explicitly, we have the isomorphism $\varphi_{D',C'}:\Hom_{\cc'}(F'D',C')\to  \Hom_{\dd'}(D', G'C')$ that assigns to any morphism $f:F'D'\to C'$ the unique morphism $h:D'\to G'C'$ in $\dd'$ such that $T(h)=G(S(f))\circ\eta_{TD'}$. 
 Its inverse is given for any $g:D'\to G'C'$ in $\dd'$ by $\phi_{D', C'}(g)=k$, where $k:F'D'\to C'$ is the unique morphism in $\cc'$ such that $S(k)=\epsilon_{SC'}\circ F(T(g))$. In fact, for any morphism $f:F'D'\to C'$ in $\cc'$, we have $\phi_{D',C'}(\varphi_{D',C'}(f))=\phi_{D',C'}(h)$, where $h:D'\to G'C'$ is the unique morphism in $\dd'$ such that $T(h)=G(S(f))\circ\eta_{TD'}$, and $\phi_{D',C'}(h)=k$, with $k:F'D'\to C'$ the unique morphism in $\cc'$ such that $S(k)=\epsilon_{SC'}\circ F(T(h))$. Note that $S(k)=\epsilon_{SC'}\circ F(T(h))=\epsilon_{SC'}\circ FG(S(f))\circ F(\eta_{TD'})=S(f)\circ \epsilon_{SF'D'}\circ F(\eta_{TD'})=S(f)\circ \epsilon_{FTD'}\circ F(\eta_{TD'})=S(f)\circ\mathrm{Id}_{FTD'}=S(f)$, hence $k=f$ and so $\phi_{D',C'}(\varphi_{D',C'}(f))=f$. Similarly, for any $g:D'\to G'C'$ in $\dd'$, we have $\varphi_{D',C'}(\phi_{D',C'}(g))=\varphi_{D',C'}(k)$, where $k:F'D'\to C'$ is the unique morphism in $\cc'$ such that $S(k)=\epsilon_{SC'}\circ F(T(g))$ and $\varphi_{D',C'}(k)=h$, with $h:D'\to G'C'$ the unique morphism in $\dd'$ such that $T(h)=G(S(k))\circ\eta_{TD'}$. Note that $T(h)=G(S(k))\circ\eta_{TD'}=G(\epsilon_{SC'}\circ FT(g))\circ\eta_{TD'}=G\epsilon_{SC'}\circ GFT(g)\circ\eta_{TD'}=G\epsilon_{SC'}\circ \eta_{TG'C'}\circ T(g)=G\epsilon_{SC'}\circ\eta_{GSC'}\circ T(g)=\mathrm{Id}_{GSC'}\circ T(g)=T(g)$, hence $h=g$ and so $\varphi_{D',C'}(\phi_{D',C'}(g))=g$. Moreover, $\varphi_{D',C'}$ is natural in $D'$ as for any morphism $u:\bar{D}\to D'$ in $\dd'$ we have that $\varphi_{D',C'}(f)\circ u=\varphi_{\bar{D},C'}(f\circ F'u)$, for every $f:F'D'\to C'$ in $\cc'$. Indeed, $\varphi_{D',C'}(f)\circ u=h\circ u$, where $h:D'\to G'C'$ is the unique morphism in $\dd'$ such that $T(h)=G(S(f))\circ\eta_{TD'}$, while $\varphi_{\bar{D},C'}(f\circ F'u)=k$, where $k:\bar{D}\to G'C'$ is the unique morphism in $\dd'$ such that $T(k)=G(S(f\circ F'u))\circ \eta_{T\bar{D}}$. Note that $G(S(f\circ F'u))\circ \eta_{T\bar{D}}=GS(f)\circ GSF'(u)\circ \eta_{T\bar{D}}=GS(f)\circ GFT(u)\circ \eta_{T\bar{D}}=GS(f)\circ\eta_{TD'}\circ Tu=T(h)\circ T(u)=T(h\circ u)$, hence $T(k)=T(h\circ u)$, thus $k=h\circ u$, and then $\varphi_{D',C'}(f)\circ u=\varphi_{\bar{D},C'}(f\circ F'u)$. In an analogous way one can prove that $\varphi_{D',C'}$ is natural in $C'$. Indeed, for any morphism $u:C'\to \bar{C}$ in $\cc'$ we have that, for every $f:F'D'\to C'$ in $\cc'$, $G'u\circ \varphi_{D',C'}(f) = G'u\circ h$, where $h:D'\to G'C'$ is the unique morphism in $\dd'$ such that $T(h)=G(S(f))\circ\eta_{TD'}$. On the other hand, $\varphi_{D',\bar{C}}(u\circ f)=k$, where $k:D'\to G'\bar{C}$ is the unique morphism in $\dd'$ such that $T(k)=G(S(u\circ f))\circ \eta_{TD'}$, and then, since $G(S(u\circ f))\circ \eta_{TD'}=GS(u)\circ GS(f)\circ\eta_{TD'}=TG'(u)\circ T(h)=T(G'u\circ h)$, we get $k=G'u\circ h$, hence $\varphi_{D',C'}$ is natural in $C'$ as $\varphi_{D',\bar{C}}(u\circ f)=G'u\circ \varphi_{D',C'}(f)$. Thus, $F'\dashv G'$ is an adjunction. Moreover, there exist natural transformations $\epsilon':F'G'\to\mathrm{Id}_{\dd'}$ and $\eta':\mathrm{Id}_{\cc'}\to G'F'$ such that $\epsilon' F'\circ F'\eta'=\mathrm{Id}_{F'}$ and $G'\epsilon'\circ\eta' G'=\mathrm{Id}_{G'}$.
 \end{invisible}
	Finally, assume that $G$ is semiseparable. Since $S$ is fully faithful, by Lemma \ref{lem:comp} (ii) $G\circ S$ is semiseparable, and then $T\circ G'$ is semiseparable, hence, since $T$ is faithful, by Lemma \ref{lem:comp} (iii) it follows that also $G'$ is semiseparable. If $G$ is separable, the proof follows analogously. The case with $F$ and $F'$ is similar. \begin{invisible}
		Assume that $F$ is semiseparable. Since $T$ is fully faithful, by Lemma \ref{lem:comp} (ii) $F\circ T$ is semiseparable, and then $S\circ F'$ is semiseparable, hence, since $S$ is faithful, by Lemma \ref{lem:comp} (iii) it follows that also $F'$ is semiseparable.
	\end{invisible}
\end{proof}

As in Example \ref{exa:proj}, denote by $R\text{-}\mathrm{Mod}_f$ the full subcategory of $R\text{-}\mathrm{Mod}$ consisting of free left $R$-modules. 
Given a ring morphism $\varphi:R\to S$, the induction functor $\varphi^*=S\otimes_R(-):{}R\text{-}\mathrm{Mod}\to{}S\text{-}\mathrm{Mod}$ has a right adjoint, namely the restriction of scalars functor $\varphi_{*}: {}S\text{-}\mathrm{Mod}\to {}R\text{-}\mathrm{Mod}$. Moreover $\varphi^*$
preserves free modules as $S\otimes_RR^{(B)}\cong (S\otimes_RR)^{(B)}\cong S^{(B)}$, 
\begin{invisible}
    (see \cite[Theorem 19.10]{AF92}),
\end{invisible} 
giving rise to the functor 
\[\varphi_f^*=S\otimes_R(-):{}R\text{-}\mathrm{Mod}_f\to{}S\text{-}\mathrm{Mod}_f,\] that we call the \textbf{free induction functor}.\medskip

We have the following result.

\begin{prop}\label{prop:varfree} Let $\varphi:R\to S$ be a ring morphism. The following assertions are equivalent.
\begin{enumerate}
    \item The free induction functor $\varphi^*_f: {}R\text{-}\mathrm{Mod}_f\to{}S\text{-}\mathrm{Mod}_f$ has a right adjoint $\varphi_{*f}$.
    \item $S$ is free as a left $R$-module.
    \item The restriction of scalars functor $\varphi_{*}: {}S\text{-}\mathrm{Mod}\to {}R\text{-}\mathrm{Mod}$ preserves free modules.
\end{enumerate}
In case the above equivalent conditions hold, then $\varphi_{*f}$ is induced by $\varphi_{*}$ and the unit and counit of $(\varphi^*_f,\varphi_{*f})$ are the restrictions of the ones of $(\varphi^*,\varphi_{*})$. Moreover, if $S\neq 0$, then $\varphi$ is injective and $\varphi^*_{f}$ is faithful.
\end{prop}
We call the functor $\varphi_{*f}$  the \textbf{free restriction of scalars functor}.
\begin{proof}
$(1)\Rightarrow (2)$. Assume that $\varphi_f^*$ has a right adjoint $G:{}S\text{-}\mathrm{Mod}_f\to{}R\text{-}\mathrm{Mod}_f$. Then, we have the following isomorphisms of left $R$-modules: $S\cong {}_S\Hom({}_S S,{}_S S)\cong{}_S\Hom(S\otimes_R R,{}_S S)={}_S\Hom(\varphi_f^*(R),{}_S S)\cong {}_R\Hom({}_R R,{}_R G(S))\cong {}_RG(S)$. 
\begin{invisible} [Abbiamo inserito le seguenti righe in invisible perché gli isomorfismi in gioco sono quelli standard.] Explicitly, we have the isomorphism $\beta:S \to {}_S\Hom({}_S S,{}_S S)$, $x\mapsto [s\mapsto sx]$, with inverse $f\mapsto f(1_S)$; the isomorphism $\gamma:{}_S\Hom({}_S S,{}_S S)\to {}_S\Hom(S\otimes_R R,{}_S S)$, $f\mapsto f\circ\mu$, with inverse $\delta: {}_S\Hom(S\otimes_R R,{}_S S)\to {}_S\Hom({}_S S,{}_S S)$, $h\mapsto h\circ \mu^{-1}$, where $\mu: S\otimes _R R\to S$ is the isomorphism $s\otimes_R r\mapsto s\cdot r=s\varphi(r)$ with inverse $\mu^{-1}:S\to S\otimes_R R$, $s\mapsto s\otimes_R 1_R$; the isomorphism $\alpha: {}_RG(S)\to {}_R\Hom({}_R R,{}_R G(S))$, $m\mapsto [r\mapsto r\cdot m]$, with inverse $f\mapsto f(1_R)$. Indeed, $\beta(x)$ is a left $S$-module homomorphism as, for every $s,s',t,t'\in S$, we have $\beta(x)(s\cdot s'+t\cdot t')=\beta(x)(ss'+tt')=(ss'+tt')x=ss'x+tt'x=s\cdot\beta(x)(s')+t\cdot\beta(x)(t')$, and $\beta$ is left (and right) $R$-linear as for every $r,r'\in R$ and $x,y\in S$, we have $\beta(r\cdot x+r'\cdot y)(s)=\beta(\varphi(r)x+\varphi(r')y)(s)=s\varphi(r)x+s\varphi(r')y=\beta(x)(s\varphi(r))+\beta(y)(s\varphi(r'))=\beta(x)(s\cdot r)+\beta(y)(s\cdot r')=(r\cdot\beta(x))(s)+(r'\cdot\beta(y))(s)$ (and $\beta(x\cdot r+ y\cdot r')(s)=\beta(x\varphi(r)+y\varphi(r'))(s)=sx\varphi(r)+sy\varphi(r')=(\beta(x)(s))\varphi(r)+(\beta(y)(s))\varphi(r')=(\beta(x)\cdot r)(s)+(\beta(y)\cdot r')(s)$), for every $s\in S$. 
	Moreover, $\beta$ is injective as, if $\beta(x)=0$, then $x=\beta(x)(1)=0$, for any $x\in S$, and it is also surjective as, if $f\in {}_S\Hom({}_S S,{}_S S)$, then $f(s)=sf(1)=\beta(f(1))(s)$, for every $s\in S$. Next consider $\gamma:{}_S\Hom({}_S S,{}_S S)\to {}_S\Hom(S\otimes_R R,{}_S S)$. We have that $\gamma(f)=f\circ\mu$ is a composition of left $S$-module homomorphisms, so it is a morphism of left $S$-modules; $\gamma$ is left $R$-linear as for every $f,f'\in {}_S\Hom({}_S S,{}_S S)$ we have $\gamma(\bar{r}\cdot f+r'\cdot f')(s\otimes_R r)=(\bar{r}\cdot f+r'\cdot f')(s\varphi(r))=(\bar{r}\cdot f)(s\varphi(r))+(r'\cdot f')(s\varphi(r))=f(s\varphi(r)\varphi(\bar{r}))+f'(s\varphi(r)\varphi(r'))=f(s\varphi(r\bar{r}))+f'(s\varphi(rr'))=f\mu(s\otimes_R r\bar{r})+f'\mu(s\otimes_R rr')=\gamma(f)((s\otimes_R r)\bar{r})+\gamma(f')((s\otimes_R r)r'))=(\bar{r}\cdot\gamma(f))(s\otimes_R r)+(r'\cdot\gamma(f'))(s\otimes_R r)=(\bar{r}\cdot\gamma(f)+ r'\cdot\gamma(f'))(s\otimes_R r)$. Moreover, $\gamma$ is invertible, as from \cite[Proposition 19.6]{AF92} $\mu$ is an isomorphism, and then we have $\delta\gamma(f)=\delta(f\mu)=(f\mu)\mu^{-1}=f\mu\mu^{-1}=f$ and $\gamma\delta(f)=\gamma(f\mu^{-1})=(f\mu^{-1})\mu=f\mu^{-1}\mu=f$. Finally, $\alpha$ results to be an isomorphism of left $R$-modules. In fact, $\alpha(m)$ is a left $R$-module homomorphism as $\alpha(m)(ra+r'b)=(ra+r'b)\cdot m=(ra)\cdot m+(r'b)\cdot m=r\cdot(a\cdot m)+r'\cdot(b\cdot m)=r\cdot(\alpha(m)(a))+r'\cdot(\alpha(m)(b))$, for every $r,r',a,b\in R$, and $\alpha$ is left $R$-linear as for every $r,r'\in R$ and $m,n\in G(S)$, we have $\alpha(r\cdot m+r'\cdot n)(a)=a\cdot(r\cdot m+r'\cdot n)=(ar)\cdot m+(ar')\cdot n =\alpha(m)(ar)+\alpha(n)(ar')=r\cdot\alpha(m)(a)+r'\cdot \alpha(n)(a)=(r\cdot\alpha(m)+r'\cdot \alpha(n))(a)$, for every $a\in R$. Moreover, $\alpha$ is injective as, if $\alpha(m)=0$, then $m=\alpha(m)(1)=0$, for any $m\in G(S)$, and it is also surjective as, if $f\in {}_R\Hom({}_R R,{}_R G(S))$, then $f(r)=r\cdot f(1)=\alpha(f(1))(r)$, for every $r\in R$. 
\end{invisible}
Since ${}_RG(S)$ is a free left $R$-module, then so is  $S$. 

$(2)\Rightarrow (3)$. Assume that $S$ is a free left $R$-module. Then, $S\cong R^{(J)}$. 
If $X$ is a free left $S$-module (i.e. $X\cong S^{(A)}$), then it can be regarded as a left $R$-module where the action of $R$ is given by $R\times X\to X$, $(r, x)\mapsto \varphi(r)x$. Then  $\varphi_*(X)={}_RX\cong({}_R S)^{(A)}\cong (R^{(J)})^{(A)}\cong R^{(A\times J)}$
 is a free left $R$-module. 
 
$(3)\Rightarrow (1)$.  If $\varphi_{*}$ preserves free modules, it induces $\varphi_{*f}:S\text{-}\mathrm{Mod}_f\to R\text{-}\mathrm{Mod}_f$. Since the inclusion functors   $i_S:S\text{-}\mathrm{Mod}_f\hookrightarrow S\text{-}\mathrm{Mod}$ and $i_R:R\text{-}\mathrm{Mod}_f\hookrightarrow R\text{-}\mathrm{Mod}$ are
fully faithful, then the assumptions of Lemma \ref{lem:subadjunction} are satisfied and $(\varphi^*_{f},\varphi_{*f})$ results to be an adjunction. Indeed, the square 
\begin{gather*}
		\xymatrixcolsep{1cm}\xymatrixrowsep{.8cm}
		\xymatrix{S\text{-}\mathrm{Mod}_f\ar[r]^-{i_S}
			\ar@<1ex>[d]^*-<0.1cm>{^{\varphi_{*f}}}& S\text{-}\mathrm{Mod} \ar@<1ex>[d]^*-<0.1cm>{^{\varphi_{*}}}
			\\
			R\text{-}\mathrm{Mod}_f\ar[r]^{i_R} \ar@<1ex>[u]^*-<0.1cm>{^{\varphi^*_{f}}}& R\text{-}\mathrm{Mod} \ar@<1ex>[u]^*-<0.1cm>{^{\varphi^*}}\ar@{}[u]|{\dashv}
	 }
\end{gather*} 
is commutative, i.e. $i_R \circ \varphi_{*f}=\varphi_*\circ i_S$ and $i_S\circ\varphi^*_{f}=\varphi^*\circ i_R$, since $\varphi_{*f}$ and $\varphi^*_{f}$ have been defined as the restrictions of $\varphi_{*}$ and $\varphi^*$ respectively. 
Since the pair $(i_S,i_R)$ constitute ad morphism of adjuctions, by \cite[Proposition 1, page 99]{Mac98} we know that the unit $\eta_f$ and counit $\epsilon_f$ of $(\varphi^*_{f},\varphi_{*f})$ are related to the unit $\eta$ and counit $\epsilon$ of $(\varphi^*,\varphi_{*})$ by the equalities $\eta i_R=i_R\eta_f$ and $\epsilon i_S=i_S\epsilon_f$.
This means that $\eta_f$ and $\epsilon_f$ are just the restrictions of $\eta$ and $\epsilon$ respectively. 
 Explicitly, the unit is defined as $(\eta_f)_M  : M\to S\otimes_R M,\, m\mapsto 1_{S}\otimes_R m$,
for any $M\in R\text{-}\mathrm{Mod}_f$. Note that $(\eta_f)_M=(\varphi\otimes_R M)\circ l_M^{-1}$ where $l_M:R\otimes_R M\to M$ is the canonical isomorphism. 
Assume $S\neq 0$. Since $M$ if a free left $R$-module, then it is flat, so that $(\eta_f)_M$ is injective as so is $\varphi$ since $\mathrm{Ker}(\varphi)\subseteq \mathrm{Ann}_R(S)$ and the annihilator is zero as every non-trivial free left $R$-module is faithful. Then, $\varphi^*_f$ is faithful.
 \begin{invisible}
    Let $r\in\mathrm{Ker}(\varphi)$. Then $rs=\varphi(r)s=0$ so that $r\in \mathrm{Ann}_R(S)$. This proves the inclusion $\mathrm{Ker}(\varphi)\subseteq \mathrm{Ann}_R(S)$. Now let $r\in \mathrm{Ann}_R(S)$. Since $S$ is free it has a non-zero basis element, say $s$. From $rs=0$ and linear independence, we deduce that $r=0$ so that $\mathrm{Ann}_R(S)=0$, i.e. $S$ is faithful.
 \end{invisible}
\end{proof}


We recall the following known facts:
\begin{itemize}
	\item $\varphi_*$ is separable if and only if $S/R$ is separable, i.e. the multiplication $m_S:S\otimes_R S\to S$, $s\otimes_R s'\mapsto ss'$ splits as an $S$-bimodule map, see \cite[Proposition 1.3]{NVV89};

	\item $\varphi^*$ is separable if and only if $\varphi$ is split-mono as an $R$-bimodule map, i.e. if there is $E\in {}_{R}\Hom_{R}(S,R)$ such that $E\circ \varphi=\id$, see \cite[Proposition 1.3]{NVV89}; 
	\item $\varphi^*$ is semiseparable if and only if $\varphi$ is a regular morphism of $R$-bimodules, i.e. there is $E\in {}_{R}\Hom_{R}(S,R)$ such that $\varphi\circ E\circ\varphi =\varphi$, see \cite[Proposition 3.1]{AB22}.
\end{itemize}
Note that the free restriction of scalars functor $\varphi_{*f}$ is a faithful functor, so by Proposition \ref{prop:sep} it is semiseparable if and only if it is separable. Assuming that $S\neq 0$ is free as a left $R$-module,  then by Proposition \ref{prop:varfree} the functor $\varphi^*_{f}$ is faithful, hence again by Proposition \ref{prop:sep} it is semiseparable if and only if it is separable. It remains to check when $\varphi_{*f}$ and $\varphi^*_{f}$ are separable functors.

\begin{prop}\label{prop:inducfunc-free}
	Let $\varphi :R\to S$ be a morphism of rings, with $S$ a free left $R$-module. 
 \begin{enumerate}
     \item[1)] The free induction functor $\varphi_f^*= S\otimes_{R}(-):R\text{-}\mathrm{Mod}_f\rightarrow S\text{-}\mathrm{Mod}_f$ is separable if and only if  $\varphi$ is a split-mono as an $R$-bimodule map.
     \item[2)] The free restriction of scalars functor $\varphi_{*f}:S\text{-}\mathrm{Mod}_f\to R\text{-}\mathrm{Mod}_f$ is separable if and only if $S/R$ is separable.
 \end{enumerate}
\end{prop}

\begin{proof}
1)
Assume that $\varphi_f^*$ is separable. Then, by Rafael Theorem, there exists a natural transformation $\nu \in \mathrm{Nat}(\varphi_{*f}\varphi_f^*,\id_{R\text{-}\mathrm{Mod}_f})$ such that $\nu\circ\eta = \id$, where $\eta$ is the unit of $(\varphi^*_{f},\varphi_{*f})$ i.e. $\eta_M  : M\to S\otimes_R M,\, m\mapsto 1_{S}\otimes_R m$,
for any $M\in R\text{-}\mathrm{Mod}_f$. Now, since $R$ is a free $R$-module, we consider $E\in {}_{R}\Hom_{R}(S,R)$ defined by setting $E(s):=\nu_R(s\otimes_R 1_R)$, for every $s\in S$ (note that the right $R$-linearity of $E$ descends from the naturality of $\nu$.\begin{invisible}
    Indeed, for any $s\in S$, $r\in R$ we have that $E(s)r=\nu_R(s\otimes_R 1_R)r=(f_r\circ\nu_R)(s\otimes_R 1_R)=(\nu_R\circ(S\otimes_Rf_r))(s\otimes_R1_R)=\nu_R(s\otimes_R r)=\nu_R(sr\otimes_R1_R)=E(sr)$, where $f_r:R\to R$ is the left $R$-module map $r'\mapsto r'r$.
\end{invisible} Then, for every $r\in R$,
	we get $
	(E\circ\varphi )(r) = E(\varphi(r))=\nu_R(\varphi(r)\otimes_R 1_R)=\nu_R(\eta_R(r)) = r
	$. Thus $E\circ\varphi =\id$. 
 Conversely, if $\varphi$ is a split-mono as an $R$-bimodule map, we mentioned that $\varphi^*$ is separable. By Lemma \ref{lem:subadjunction}, so is $\varphi_f^*$.
 
2) Assume now that $\varphi_{f*}$ is separable. Then, by Rafael Theorem,  there exists a natural transformation $\gamma \in \mathrm{Nat}(\id_{S\text{-}\mathrm{Mod}_f},\varphi_f^*\varphi_{*f})$ such that $\epsilon\circ\gamma = \id$, where $\epsilon$ is the counit of $(\varphi^*_{f},\varphi_{*f})$ i.e. $\epsilon_N  : S\otimes_R N\to N,\, s\otimes n\mapsto sn$,
for any $N\in S\text{-}\mathrm{Mod}_f$. Now, since $S$ is a free $S$-module, we consider $\gamma_S\in {}_{S}\Hom_{S}(S,S\otimes_R S)$ (note that the right $S$-linearity of $\gamma_S$ descends from the naturality of $\gamma$). 
\begin{invisible}
    Indeed, for any $s,x\in S$ consider the left $S$-module map $f_s:S\to S$, $x\mapsto xs$. Then, by naturality of $\gamma$ we have $\gamma_S(xs)=\gamma_Sf_s(x)=(S\otimes_Rf_s)\gamma_S(x)=\gamma_S(x)s.$
\end{invisible}
Since $\epsilon_S\circ \gamma_S=\id$, we conclude that the multiplication $m_S=\epsilon_S:S\otimes_R S\to S$ splits as an $S$-bimodule map so that $S/R$ is separable.
Conversely, if $S/R$ is separable, we mentioned that $\varphi_*$ is separable. By Lemma \ref{lem:subadjunction}, so is $\varphi_{f*}$.
\end{proof}

\begin{es}
1) Consider the morphism of rings $\varphi:\mathbb{R}\to \mathbb{R}\times \mathbb{R},r\mapsto (r,r)$. The $\mathbb{R}$-bimodule structure induced on $\mathbb{R}\times \mathbb{R}$ via $\varphi$ is the canonical one so that it is free.
The canonical projection $E:\mathbb{R}\times \mathbb{R}\to \mathbb{R},(a,b)\mapsto a$, is a morphism of $\mathbb{R}$-bimodules such that $E\circ \varphi=\id$. By Proposition \ref{prop:inducfunc-free}, the free induction functor $\varphi_f^*= \mathbb{R}^2\otimes_{\mathbb{R}}(-):\mathbb{R}\text{-}\mathrm{Mod}_f\rightarrow \mathbb{R}^2\text{-}\mathrm{Mod}_f$  is separable.

2) Let $R$ be a ring and let $\varphi:R\to \mathrm{M}_n(R)$ be the canonical inclusion into the ring of $n\times n$ matrices over $R$. It is well-known that $\mathrm{M}_n(R)/R$ is separable (see e.g. \cite[Example II]{DI71}) and clearly $\mathrm{M}_n(R)\cong R^{n^2}$ is free as a left $R$-module.
By Proposition \ref{prop:inducfunc-free}, the free restriction of scalars functor $\varphi_{*f}:\mathrm{M}_n(R)\text{-}\mathrm{Mod}_f\to R\text{-}\mathrm{Mod}_f$ is separable.
\end{es}

\bl{\begin{invisible}
\begin{prop}\label{prop:fact-free}
	Let $\varphi :R\to S$ be a morphism of rings with $S$ a free left $R$-module. Write $\varphi=\iota\circ \overline{\varphi}$ where $\iota:\varphi(R)\to S$ is the canonical inclusion and $\overline{\varphi}:R\to \varphi(R)$ is the corestriction of $\varphi$ to its image $\varphi(R)$. Consider the free induction functors $\iota^*_f:R\text{-}\mathrm{Mod}_f\to \varphi(R)\text{-}\mathrm{Mod}_f$ and $\overline{\varphi}^*_f:\varphi(R)\text{-}\mathrm{Mod}_f\to S\text{-}\mathrm{Mod}_f$.
	Then, $\varphi_f^*:R\text{-}\mathrm{Mod}_f\rightarrow S\text{-}\mathrm{Mod}_f$ is separable if and only if $\iota^*_f$ is separable and $\overline{\varphi}^*_f$ is an equivalence. 
	\end{prop}
	\begin{proof} \rd{[Nota che se $\overline{\varphi}^*_f$ è un'equivalenza, allora ha un'aggiunto e quindi per Proposizione \ref{prop:varfree} si ha che $\overline{\varphi}$ è iniettiva, dunque biiettiva. Ciò significa $\varphi$ iniettiva. ]}
		If $\varphi^*_f$ is separable, by Proposition \ref{prop:inducfunc-free} there is $E\in {}_{R}\Hom_{R}(S,R)$ such that $\id =E\circ \varphi=E\circ\iota\circ \overline{\varphi}$, hence $\varphi^*$ is separable and also $\overline{\varphi}^*$ results to be separable. By \cite[Proposition 3.5]{AB22} since $\varphi^*$ is in particular semiseparable we have that $\iota^*$ is separable and $\overline{\varphi}^*$ is a bireflection. Thus, $\overline{\varphi}^*$ is an equivalence by Lemma \ref{lem:trivfactsutr} (5) since a separable bireflection is a fully faithful bireflection. Consider the free induction functors $\iota^*_f$ and $\overline{\varphi}^*_f$ which fulfill $i_{\varphi(R)}\circ \iota^*_{f}=\iota^*\circ i_{R}$. and $i_S\circ\overline{\varphi}^*_{f}=\overline{\varphi}^*\circ i_{\varphi(R)}$. Then, $\iota^*_f$ and $\overline{\varphi}^*_f$ are separable as so are $\iota^*$ and $\overline{\varphi}^*$, respectively. Note that $\overline{\varphi}^*_f$ has a right adjoint as $S$ is free as a left $\varphi(R)$-module with basis $\{1_S=\varphi(1_R)\}$, thus Proposition \ref{prop:varfree} applies. Since the restriction of scalars functor $\overline{\varphi}_*$ is full (as $\overline{\varphi}$ is surjective), the counit $\bar{\epsilon}_D$ of the adjunction $(\overline{\varphi}^*,\overline{\varphi}_*)$ splits for every $D\in S\text{-}\mathrm{Mod}$, and hence also the counit $\bar{\epsilon}_{fD}$ of the adjunction $(\overline{\varphi}^*_f,\overline{\varphi}_{*f})$, which is given by the restriction of $\bar{\epsilon}$, splits for every $D\in S\text{-}\mathrm{Mod}_f$, thus $\overline{\varphi}_{*f}$ is full. Since it is also faithful, we get that $\overline{\varphi}^*_f$ is a reflection. By Theorem \ref{thm:frobenius} $\overline{\varphi}^*_f$ results to be a bireflection, and then an equivalence as it is separable. Conversely, if  $\iota^*_f$ is separable and $\overline{\varphi}^*_f$ is an equivalence, whence separable, then the composition $\iota^*_f\circ \overline{\varphi}^*_f\cong (\iota\circ \overline{\varphi})^*_f=\varphi^*_f$ is separable. 
\end{proof}
\end{invisible}}

\subsection{Comparing the factorizations of a semiseparable adjoint}\label{sub:compfact}
Let $F\dashv G:\dd\to\cc$ be an adjunction. So far we have seen that if $G$ is semiseparable, then it admits two canonical factorizations as a bireflection up to retracts followed by a separable functor, namely $G=G_e\circ H$  (cf. Theorems  \ref{thm:coidentifier} and \ref{thm:H-corefl-utr}) and $G=U_{GF}\circ K_{GF}$ (cf. Theorems  \ref{thm:ssepMonad} and \ref{thm:computr}).
$$\xymatrixcolsep{1.6cm}\xymatrixrowsep{.6cm}\xymatrix{\dd\ar[r]^{{H}}_{\text{biref. u.t.r.}}& {\dd_e}\ar[r]^{{G_e}}_{\text{sep.}}& \cc}\qquad \text{and}\qquad\xymatrix{\dd\ar[r]^{ K_{GF}}_{\text{biref. u.t.r.}}& {\cc_{GF}}\ar[r]^{ U_{GF}}_{\text{sep.}}& \cc}$$
Similar factorizations have been obtained also for $F$ in case it is semiseparable. 
Next aim is to compare these factorizations. First we need Lemma \ref{lem:coididp}, an easy result concerning the idempotent completeness of the coidentifier, and the useful Lemma \ref{lem:Berger}, regarding the composition of (co)reflections (up to retracts). The subsequent Proposition \ref{prop:F_e:equiv-utr} provides a factorization of bireflections up to retracts.\medskip

In order to state next result, we adopt the following terminology : A functor $F:\cc\to\dd$ \textbf{lifts idempotents} whenever each idempotent morphism in $\dd$ is of the form $F(q)$ for some idempotent morphism $q$ in $\cc$. It is clear that, given such a functor, if $\cc$ is idempotent complete so is $\dd$.
\begin{invisible}
Let $F:\cc\to\dd$ be a functor that lifts idempotents morphisms and assume that $\cc$ is idempotent complete.
If $d$ is an idempotent morphism in $\dd$, then there is an idempotent morphism $q$ in $\cc$ such that $d=Fq$. By hypothesis we can write $q=i\circ p$
for some morphisms $p$ and $i$ such that $p\circ i=\id$. Then $%
d=F(q)=F(i\circ p)=Fi\circ Fp$ and $%
Fp\circ Fi=F(p\circ i)=F(\id)=\id$
so that the idempotent $d$ splits. We have so proved that $\dd$ is idempotent complete.
\end{invisible}

\begin{lem}
\label{lem:coididp}Let $\cc$ be a category and let $e:\id_{%
\cc}\rightarrow \id_{\cc}$ be an idempotent natural
transformation. Then the quotient functor $H:\cc\to\cc_e$ lifts idempotents. As a consequence, if $\cc$ is idempotent complete so is the coidentifier $\cc_{e}$.
\end{lem}

\begin{proof}
Let $\overline{h}:C\rightarrow C$ be an idempotent morphism in $\cc%
_{e}$. Then $\overline{h}\circ \overline{h}=\overline{h}$ i.e. $\overline{%
h\circ h}=\overline{h}$ and hence $e_{C}\circ h\circ h=e_{C}\circ h$. Set $%
q:=e_{C}\circ h:C\rightarrow C$. Then $q\circ q=e_{C}\circ h\circ e_{C}\circ
h=e_{C}\circ e_{C}\circ h\circ h=e_{C}\circ h\circ h=e_{C}\circ h=q$ and
hence $q$ is an idempotent morphism in $\cc$. Moreover $Hq=\overline{q}=\overline{e_{C}\circ h}=\overline{h}$.
\end{proof}

\begin{lem}
	\label{lem:Berger}Let $G:\dd\rightarrow \cc$ and $U:\mathcal{C\rightarrow C}^{\prime }$
	be functors.
	\begin{enumerate}
		\item[1)] If $G$ is a (co)reflection and $U$ is conservative, then $U$ is an equivalence if and
		only if $U\circ G$ is a (co)reflection.
		\item[2)]  If $G$ is a (co)reflection up to retracts and $U$ is separable, then $U$ is an equivalence up to retracts if and
		only if $U\circ G$ is a (co)reflection up to retracts.
	\end{enumerate}
	
\end{lem}

\begin{proof}Set $G^{\prime }:=U\circ G.$
	
	1) Since $U$ is conservative, if $G'$ is a coreflection, by \cite[%
	Corollary 4.9]{AGM18}, which is a consequence of \cite[Lemma 1.2]{Ber07}, we
	get that $U$ is an equivalence. Conversely, if $U$
	is an equivalence then it is in particular a coreflection and hence, by Remark \ref{rmk:coref-biref}, $G'$ is a coreflection as a composition of coreflections. The statement for $G$ a reflection is proved dually.
	
	2) By Corollary \ref{cor:sep-nat-ff}, since $U$ is separable so is $U^{\natural }$. In particular $U^{\natural }$ is conservative, by Remark \ref{rmk:Maschke}.
	Therefore we have that $\left( G^{\prime
	}\right) ^{\natural }=U^{\natural }\circ G^{\natural }$ where $G^{\natural }$
	is a (co)reflection and $U^{\natural }$ is conservative. By 1), we get that $U^{\natural }$ is an equivalence (i.e. $U$ is an equivalence up to retracts) if and only if $\left( G^{\prime
	}\right) ^{\natural }$ is a (co)reflection (i.e. $G'$ is a (co)reflection up to retracts). 
\end{proof}
\begin{prop}\label{prop:F_e:equiv-utr}
Let $F:\cc\rightarrow \dd$ be a bireflection up to retracts.
If we consider the associated idempotent natural transformation
$e:\id_{\cc}\rightarrow \id_{\cc}$ and the corresponding
factorization $F=F_{e}\circ H$, then the unique functor $F_{e}:\cc%
_{e}\rightarrow \dd$ is an equivalence up to retracts. If $\cc$ is idempotent complete, then $F_{e}$ is an equivalence.
\end{prop}

\begin{proof}
If $F$ is a bireflection up to retracts, it is a semiseparable coreflection up to retracts by Lemma \ref{lem:trivfactsutr}. In particular, $F$ admits the associated idempotent natural transformation $e:\id_{\cc}\rightarrow
\id_{\cc}$, see Proposition \ref{prop:idempotent}. By Theorem \ref{thm:coidentifier}, there is a factorization $F=F_{e}\circ H$ for a unique
functor $F_{e}:\cc_{e}\rightarrow \dd$ which is separable. Since both $F$ and $H$ are coreflections up to retracts (see Theorem \ref{thm:H-corefl-utr}) and $F_{e}$ is separable, by Lemma \ref{lem:Berger}, we get that $F_{e}$ is an equivalence
up to retracts.

If $\cc$ is idempotent complete so is $\cc_e$ by Lemma \ref{lem:coididp}. Then $F_{e}$ is an equivalence in view of Proposition \ref{prop:idpcom-cutr}.
\end{proof}

\begin{es}
 Let $F:\cc\to\dd$ be a bireflection up to retracts. Thus $F^\natural:\cc^\natural\to\dd^\natural$ is a bireflection. In particular, by Lemma \ref{lem:trivfactsutr}, it is a bireflection up to retracts whose source category $\cc^\natural$ is idempotent complete. By Proposition \ref{prop:F_e:equiv-utr}, $(F^\natural)_{\alpha}:(\cc^\natural)_{\alpha}\to \dd^\natural$ is an equivalence where $\alpha:\id_{\cc^\natural}\rightarrow
\id_{\cc^\natural}$ is the idempotent natural transformation associated to $F^\natural$. By definition and running through again the proof of Corollary \ref{cor:sep-nat-ff}, we get that
$$\alpha_{(C,c)}=\p^{F^\natural}_{(C,c),(C,c)}(\id_{F^\natural(C,c)})
=\p^{F^\natural}_{(C,c),(C,c)}(\id_{(FC,Fc)})
=\p^{F}_{C,C}(Fc)=\p^{F}_{C,C}(\id_{FC})\circ c=e_{C}\circ c$$ so that $\alpha=e^\natural$ where $e:\id_{\cc}\rightarrow
\id_{\cc}$ is the idempotent natural transformation associated to $F$. This shows that $(F^\natural)_{e^\natural}:(\cc^\natural)_{e^\natural}\to\dd^\natural$ is an equivalence and hence $\dd^\natural\cong (\cc^\natural)_{e^\natural}$.

  In particular, in Theorem \ref{thm:H-corefl-utr} we proved that $H:\cc\to\cc_e$ is a bireflection up to retracts. By the foregoing, $(H^\natural)_{e^\natural}:(\cc^\natural)_{e^\natural}\to(\cc_e)^\natural$ is an equivalence and hence $(\cc_e)^\natural\cong (\cc^\natural)_{e^\natural}$.
\end{es}

We are now able to compare the two factorizations we are interested in.

\begin{prop}\label{prop:CoidEil}Consider an adjunction $F\dashv G:\dd\rightarrow
\cc$.
\begin{enumerate}
  \item If $G$ is semiseparable and $e:\id_\dd\to\id_\dd$ is the associated idempotent natural transformation, then there is a unique functor
$\left( K_{GF}\right) _{e}:\dd_{e}\rightarrow \cc_{GF}$ such
that $\left( K_{GF}\right) _{e}\circ H=K_{GF}$ and $U_{GF}\circ \left(
K_{GF}\right) _{e}=G_{e}$. Moreover, the functor $\left( K_{GF}\right) _{e}$ is an equivalence up to
retracts. If $\dd$ is idempotent complete, then $\left( K_{GF}\right) _{e}$ is an equivalence of categories.
  \item If $F$ is semiseparable and $e:\id_\cc\to\id_\cc$ is the associated idempotent natural transformation, then there is a unique functor $\left( K^{FG}\right) _{e}:\cc_{e}\rightarrow\dd^{FG}$ such that $\left( K^{FG}\right) _{e}\circ H=K^{FG}$ and $U^{FG}\circ \left( K^{FG}\right) _{e}=F_{e}$.   Moreover, the functor $\left( K^{FG}\right) _{e}$ is an equivalence up to retracts. If $\cc$ is idempotent complete, then $\left( K^{FG}\right) _{e}$ is an equivalence of categories.
\end{enumerate}
\begin{equation*}
\xymatrix{\dd\ar[r]^H\ar[d]_{K_{GF}}&\dd_e
\ar@{.>}[dl]|{(K_{GF})_e}\ar[d]^{G_e}\\\cc_{GF}\ar[r]_{U_{GF}}&\cc }
\qquad
  \xymatrix{\cc\ar[r]^H\ar[d]_{K^{FG}}&\cc_e \ar@{.>}[dl]|{(K^{FG})_e}\ar[d]^{F_e}\\
  \dd^{FG}\ar[r]_{U^{FG}}&\dd  }
  \end{equation*}
\end{prop}

\begin{proof}We just prove (1). The existence of a unique functor $\left( K_{GF}\right) _{e}$ that makes commutative the diagram in the statement has already been observed in \cite[Remark 2.10]{AB22}. Moreover the functors $G_e$ and $U_{GF}$ are separable while the functors $H$ and $K_{GF}$ are naturally full. Furthermore, by Theorem \ref{thm:coidentifier} $G$ and  $K_{GF}$ have the same associated idempotent natural transformation $e$.

Since $e$ is the associated idempotent natural transformation for $K_{GF}$, the factorization $K_{GF}=\left( K_{GF}\right) _{e}\circ H$ is necessarily the one of Proposition \ref{prop:F_e:equiv-utr}, once observed that  $K_{GF}$ is a bireflection up to retracts by Theorem \ref{thm:computr}. As a consequence $\left( K_{GF}\right) _{e}$ is an equivalence up to retracts (an equivalence in case $\dd$ is idempotent complete).
\begin{invisible}
By Theorem \ref{thm:ssepMonad}(ii) $F$ factorizes as $F=U^{FG}\circ K^{FG}$ with $U^{FG}$ separable and $K^{FG}$ naturally full. By Theorem \ref{thm:coidentifier}, $F$ also factorizes as $F=F_{e}\circ H$, for a unique functor $F_{e}:\cc_{e}\rightarrow \dd$, and there is a unique fully
faithful functor $\left( K^{FG}\right) _{e}:\cc_{e}\rightarrow
\dd^{FG}$ such that $\left( K^{FG}\right) _{e}\circ H=K^{FG}$ and $%
U^{FG}\circ \left( K^{FG}\right) _{e}=F_{e}$. Moreover, $\left( K^{FG}\right) _{e}$ is an equivalence up to retracts by Proposition \ref{prop:F_e:equiv-utr}, as $K^{FG}$ is a bireflection up to retracts by the dual of Theorem \ref{thm:computr}, whence a semiseparable reflection up to retracts. If $\cc$ is idempotent complete, so is $\cc_e$ by Lemma \ref{lem:coididp}. As a consequence $\left( K^{FG}\right) _{e}$ is an equivalence in view of Proposition \ref{prop:idpcom-cutr}.
\end{invisible}\end{proof}

Although in the present paper we usually deduced the general results from
weaker ones (e.g. we deduced results on separable functors from those on
semiseparable functors), we could, in some cases, also have done the
opposite. For instance, given an adjunction $(F,G)$ with $G$ semiseparable, since the equality $\left( K_{GF}\right) _{e}\circ H=K_{GF}$ holds and the functor $H$ is a coreflection up to retracts, by Lemma \ref{lem:Berger} we can conclude that $K_{GF}$ is a coreflection up to retracts (whence a bireflection up to retracts) if we know that $\left( K_{GF}\right) _{e}$  is an equivalence up to
retracts. In other words, we can give a different proof of Theorem \ref{thm:computr}, by first showing that  $\left( K_{GF}\right) _{e}$  is an equivalence up to
retracts. To this aim we first need the following lemma.

\begin{lem}
\label{lem:coidadj}Let $G_{e}:\dd_{e}\rightarrow \cc$ be a
functor. If $G:=G_{e}\circ H:\dd\rightarrow \cc$ has a left
adjoint $F$ with unit $\eta $ and counit $\epsilon $, then $F_{e}:=H\circ F$
is a left adjoint of $G_{e}$ with unit $\eta _{e}$ and counit $\epsilon _{e}$
uniquely defined by the identities $\eta _{e}=\eta $ and $\epsilon
_{e}H=H\epsilon $. Moreover the adjunctions $\left(
F_{e},G_{e}\right) $ and $\left( F,G\right) $ have the same associated
monad (whence $\cc_{G_{e}F_{e}}=\cc_{GF}$) and the respective comparison functors are related by the equality $K_{G_{e}F_{e}}\circ H=K_{GF}$.
\begin{gather*}
\vcenter{
\xymatrixcolsep{1.6cm}\xymatrixrowsep{.6cm}
\xymatrix{\dd\ar[rd]_{K_{GF}}\ar[rr]^{H}
\ar@<1ex>[dd]^*-<0.1cm>{^G}&& \dd_e \ar[ld]^{K_{G_{e}F_{e}}}\ar@<1ex>[dd]^*-<0.1cm>{^{G_e}}\\
&\cc _{GF}=\cc_{G_{e}F_{e}}\ar[ld]_-{U_{GF}}\ar[rd]^-{U_{G_eF_e}}
\\
\cc\ar[rr]^{\id}  \ar@<1ex>[uu]^*-<0.1cm>{^F}\ar@{}[uu]|{\dashv}&&\cc \ar@<1ex>[uu]^*-<0.1cm>{^{F_e}}\ar@{}[uu]|{\dashv}
} }
\end{gather*}
\end{lem}

\begin{proof}
Given $\epsilon :FG\rightarrow \id_{\dd}$ we have $H\epsilon
:HFG\rightarrow H$. By the universal property of the coidentifier, since $%
\left( F_{e}G_{e}\right) \circ H=HFG$ and $\id_{\dd%
_{e}}\circ H=H,$ we have $\left( HFG\right) _{e}=F_{e}G_{e}$ and $H_{e}=%
\id_{\dd_{e}}$ and hence there is a unique natural
transformation $\epsilon _{e}:F_{e}G_{e}\rightarrow \id_{\dd%
_{e}}$ such that $\epsilon _{e}H=H\epsilon $ (see Lemma \ref{lem:coidentifier}). Since $G_{e}\circ
F_{e}=G_{e}\circ H\circ F=G\circ F,$ it makes sense to define $\eta
_{e}:=\eta $. Then
\begin{equation*}
G_{e}\epsilon _{e}H\circ \eta _{e}G_{e}H=G_{e}H\epsilon \circ \eta
_{e}G_{e}H=G\epsilon \circ \eta G=\id_{G}=\id_{G_{e}H}.
\end{equation*}%
Since $H$ is the identity on objects, we deduce that $G_{e}\epsilon
_{e}\circ \eta _{e}G_{e}=\id_{G_{e}}$. Moreover
\begin{equation*}
\epsilon _{e}F_{e}\circ F_{e}\eta _{e}=\epsilon _{e}HF\circ HF\eta
_{e}=H\epsilon F\circ HF\eta =H\id_{F}=\id_{HF}=\id%
_{F_{e}}.
\end{equation*}
Since $G_{e}\circ F_{e}=G\circ F,$ $G_{e}\epsilon _{e}F_{e}=G_{e}\epsilon _{e}HF=G_{e}H\epsilon
F=G\epsilon F$ and $\eta _{e}=\eta $ we have that the adjunctions $\left(
F_{e},G_{e}\right) $ and $\left( F,G\right) $ have the same associated
monad. Thus $\cc_{G_{e}F_{e}}=\cc_{GF}$. Note
that
\begin{eqnarray*}
K_{G_{e}F_{e}}HX &=&\left( G_{e}HX,G_{e}\epsilon _{e}HX\right) =\left(
G_{e}HX,G_{e}H\epsilon X\right) =\left( GX,G\epsilon X\right) =K_{GF}X, \\
K_{G_{e}F_{e}}Hf &=&G_{e}Hf=Gf=K_{GF}f
\end{eqnarray*}%
so that $K_{G_{e}F_{e}}\circ H=K_{GF}$.
\end{proof}

 By Lemma \ref{lem:coidadj}, the adjunctions $\left(
F_{e}:=H\circ F,G_{e}\right) $ and $\left( F,G\right) $ have the same associated
monad (whence $\cc_{G_{e}F_{e}}=\cc_{GF}$) and the respective comparison functors are related by the equality $K_{G_{e}F_{e}}\circ H=K_{GF}$. Since the functor $\left(
K_{GF}\right) _{e}:\dd_{e}\rightarrow \cc_{GF}$ of Proposition \ref{prop:CoidEil} is uniquely
determined by the equality $\left( K_{GF}\right) _{e}\circ H=K_{GF},$ we get  $\left(
K_{GF}\right) _{e}=K_{G_{e}F_{e}}$. Since $G_{e}$ is separable, by \cite[Proposition 3.5]{Chen15}, we get that $%
K_{G_{e}F_{e}}$ is an equivalence up to retracts. Thus $\left(
K_{GF}\right) _{e}$ is an equivalence up to retracts as desired.

In a similar way, given an adjunction $(F,G)$ with $F$ semiseparable, we can conclude that $K^{FG}$ is a reflection up to retracts if we know that $\left( K^{FG}\right) _{e}$ is an equivalence up to retracts. This is a consequence of the following dual of Lemma \ref{lem:coidadj}. 
\begin{lem}
\label{lem:coidadj-dual}Let $F_{e}:\cc_{e}\rightarrow \dd$ be a
functor. If $F:=F_{e}\circ H:\cc\rightarrow \dd$ has a right
adjoint $G$ with unit $\eta $ and counit $\epsilon $, then $G_{e}:=H\circ G$
is a right adjoint of $F_{e}$ with unit $\eta _{e}$ and counit $\epsilon _{e}$
uniquely defined by the identities $\eta_{e}H =H\eta $ and $\epsilon
_{e}=\epsilon $. Moreover the adjunctions $\left(
F_{e},G_{e}\right) $ and $\left( F,G\right) $ have the same associated
comonad (whence $\dd^{F_{e}G_{e}}=\dd^{FG}$) and the respective cocomparison functors are related by the equality $K^{F_{e}G_{e}}\circ H=K^{FG}$.
\begin{gather*}
\vcenter{
\xymatrixcolsep{1.6cm}\xymatrixrowsep{.6cm}
\xymatrix{\dd\ar[rr]^{\id}
\ar@<1ex>[dd]^*-<0.1cm>{^G}&& \dd_e \ar@<1ex>[dd]^*-<0.1cm>{^{G_e}}\\
&\dd^{FG}=\dd^{F_{e}G_{e}}\ar[lu]^-{U^{FG}}\ar[ru]_-{U^{F_eG_e}}
\\
\cc\ar[rr]^{H} \ar[ru]^{K^{FG}} \ar@<1ex>[uu]^*-<0.1cm>{^F}\ar@{}[uu]|{\dashv}&&\cc_e\ar[lu]_-{K^{F_eG_e}} \ar@<1ex>[uu]^*-<0.1cm>{^{F_e}}\ar@{}[uu]|{\dashv}
} }
\end{gather*}
\end{lem}

\subsection{Idempotent completion of Kleisli category}\label{sub:kleisli}
As another application of the results about conditions up to retracts, we now focus on the Kleisli construction for a monad $(\top,m:\top\top\to \top,\eta:\id_\cc\to \top)$ on a category $\cc$.
Recall that a $\top$-module is \emph{free} when it is isomorphic to one of the form $V_{\top}C=(\top C,m _C)$, for some object $C\in\cc$, and the full subcategory of $\cc_{\top}$ generated by the free $\top$-modules is equivalent to the so-called \emph{Kleisli category $\top$-$\mathrm{Free}_\cc$ of free $\top$-modules} (see \cite{Kl65}). Explicitly the objects of $\top$-$\mathrm{Free}_\cc$  are those of $\cc$ and a morphism $f:C\nrightarrow D$ in $\top$-$\mathrm{Free}_\cc$ is a morphism $f:C\to \top (D)$ in $\cc$; the composite of two morphisms $f:C\nrightarrow D$, $g:D\nrightarrow E$ in $\top$-$\mathrm{Free}_\cc$ is given in $\cc$ by the composite $$\xymatrixcolsep{1cm}\xymatrixrowsep{0.8cm}\xymatrix{C\ar[r]^-{f}& \top(D)\ar[r]^-{\top(g)}&\top\top(E)\ar[r]^-{m_E}&\top(E),}$$ and the identity $C\nrightarrow C$ on an object $C$ of $\top$-$\mathrm{Free}_\cc$ is the unit $\eta_C: C\to \top (C)$ in $\cc$. There is (see \cite[Proposition 4.1.6]{BorII94}) a fully faithful functor 
$$J_\top:\top\text{-}\mathrm{Free}_\cc\to \cc_\top,\quad C\mapsto(\top C,m_C),\quad \left[f:C\nrightarrow D\right]\mapsto m_D\circ\top (f),$$ that fits into the following diagram
\begin{equation}\label{eq:J:monad}
\begin{split}
\xymatrixcolsep{1cm}\xymatrixrowsep{0.8cm}\xymatrix{
&\cc\ar@<0.5ex>[dr]^-{V_\top}\ar@<0.5ex>[dl]^-{V'_\top}&\\
\top\text{-}\mathrm{Free}_\cc  \ar@<0.5ex>[ur]^-{U'_\top}\ar@{^(->}[rr]_-{J_\top} &&\cc_\top ,\ar@<0.5ex>[ul]^-{U_\top} }    
\end{split}
\end{equation}
where the adjunction $(V_\top ,U_\top )$ restricts to an adjunction $(V'_\top ,U'_\top )$ between $\cc$ and $\top$-$\mathrm{Free}_\cc$, that is, $U'_\top = U_\top\circ J_\top$ and $J_\top\circ V'_\top=V_\top$ (see \cite[Corollary 4.1.7]{BorII94}).
Explicitly $U'_\top$ and  $V'_\top$ are given by
\begin{align}
U'_\top&:\top\text{-}\mathrm{Free}_\cc \to \cc,\quad C\mapsto \top (C),\quad f\mapsto m_D\circ \top (f),\label{def:U'top}\\
V'_\top&:\cc\to \top\text{-}\mathrm{Free}_\cc,\quad  C\mapsto C,\quad f\mapsto \eta_D\circ f\label{def:V'top}
\end{align}
\begin{invisible}
hence we have $J_\top V'_\top C=J_\top C=(\top C,m _C)=V_{\top}C$ and $J_\top V'_\top f=J_\top (\eta_D\circ f)= m_D\circ\top (\eta_D\circ f)=m_D\circ\top (\eta_D)\circ \top (f)=\id_{\top D}\circ\top (f)=\top (f)=V_\top (f)$.
\end{invisible}

In the next result we investigate the functor $J_\top$ in case the monad $\top$ is separable.
\begin{prop}\label{prop:J:equiv-utr}
Let $(\top,m,\eta)$ be a separable monad on a category $\cc$. Then, the canonical functor $J_\top:\top$-$\mathrm{Free}_\cc\to \cc_\top$ is an equivalence up to retracts. In particular $\top$-$\mathrm{Free}_\cc^\natural\cong \cc_\top^\natural$.
\end{prop}

\proof
By \cite[2.9 (1)]{BBW09} the separability of the monad $(\top,m,\eta)$ is equivalent to the separability of the forgetful functor $U_\top:\cc_\top\to \cc$, hence, by Rafael Theorem this is also equivalent to the fact that the counit $\beta:V_\top U_\top\to \id_{\cc_\top}$ of the adjunction $(V_\top,U_\top )$
is a split natural epimorphism. Thus, we get that $V_\top$ is surjective up to retracts and hence so is $J_\top$ in view of the equality $V_\top=J_\top\circ V'_\top$. But $J_\top$ is also fully faithful, hence it is an equivalence up to retracts by Lemma \ref{lem:trivfactsutr}.
\endproof

Now, given an adjunction $F\dashv G:\dd\to\cc$, with unit $\eta$, counit $\epsilon$, consider the diagram \eqref{eq:J:monad} for the associated monad $(GF, G\epsilon F, \eta )$. Then, (see \cite[Proposition 4.2.1]{BorII94}) there is the so-called \emph{Kleisi comparison functor} 
$$L_{GF}:GF\text{-}\mathrm{Free}_\cc\to \dd,\quad C\mapsto F(C),\quad f\mapsto \epsilon_{FD}\circ F(f),$$
such that $K_{GF}\circ L_{GF}$ is the functor $J_{GF}:GF\text{-}\mathrm{Free}_\cc\to \cc_{GF}$, $C\mapsto(GFC, G\epsilon_{FC})$, $f\mapsto G\epsilon_{FD}\circ GF(f)$.
\begin{equation}\label{eq:kleisli-eilenb}
\begin{split}
\xymatrixcolsep{0.8cm}\xymatrixrowsep{0.5cm}\xymatrix{
&&\cc\ar@<-0.5ex>[drr]_-{V_{GF}}\ar@<-0.5ex>[dll]_-{V'_{GF}}\ar@<-0.5ex>[d]_-{F}&&\\
GF\text{-}\mathrm{Free}_\cc  \ar@<-0.5ex>[urr]_-{U'_{GF}}\ar[rr]_-{L_{GF}}&&\dd\ar[dr]_-{H}\ar@<-0.5ex>[u]_-{G}\ar[rr]_-{K_{GF}}&&\cc_{GF} \ar@<-0.5ex>[ull]_-{U_{GF}}\\
&&&\dd_e\ar[ur]_-{(K_{GF})_e}}
\end{split}
\end{equation}
\begin{invisible}
Indeed, for every object $C$ in $GF\text{-}\mathrm{Free}_\cc$, we have $K_{GF} L_{GF}(C)=K_{GF}F(C)=(GFC, G\epsilon_{FC})=J_{GF}(C)$, and for every morphism $f:C\nrightarrow D$ in $GF$-$\mathrm{Free}_\cc$, $K_{GF} L_{GF}(f)=GL_{GF}(f)=G(\epsilon_{FD}\circ F(f))=G\epsilon_{FD}\circ GF(f)=J_{GF}(f)$. Moreover, $L_{GF}$ is a functor: if $f:C\nrightarrow D$ and $g:D\nrightarrow E$ are morphisms in $GF$-$\mathrm{Free}_\cc$, then $L_{GF}(g)\circ L_{GF}(f)=\epsilon_{FE}\circ F(g)\circ\epsilon_{FD}\circ F(f)=\epsilon_{FE}\circ \epsilon_{FGFE}\circ FGF(g)\circ F(f)=\epsilon_{FE}\circ FG(\epsilon_{FE})\circ FGF(g)\circ F(f)=\epsilon_{FE}\circ F(G\epsilon_{FE}\circ GF(g)\circ f)=\epsilon_{FE}\circ F(g\bullet f)=L_{GF}(g\bullet f)$, where we denoted by $\bullet$ the composition in $GF\text{-}\mathrm{Free}_\cc$, and $L_{GF}(\id_C)=\epsilon_{FC}\circ F(\id_C)=\epsilon_{FC}\circ F(\eta_C)=\id_{F(C)}=\id_{L_{GF}(C)}$.
\end{invisible}
Moreover, $G\circ L_{GF}=U'_{GF}$ and $L_{GF}\circ V'_{GF}=F$, where $U'_{GF}:GF\text{-}\mathrm{Free}_\cc \to \cc$ is defined as in \eqref{def:U'top}, i.e. by setting $U'_{GF} (C)=GF(C)$, $U'_{GF} (f)=G\epsilon_{FD}\circ GF(f)$, for every object $C$ and every morphism $f:C\nrightarrow D$ in $GF\text{-}\mathrm{Free}_\cc$, and $V'_{GF}:\cc\to GF\text{-}\mathrm{Free}_\cc $ as in \eqref{def:V'top}, i.e. it is the identity map on objects and, for every morphism $f:C\to D$ in $\cc$, it is given by $V'_{GF}(f)=\eta_D\circ f$.
\begin{invisible}
$GL_{GF}(C)=GF(C)=U'_{GF}(C)$ and $G L_{GF}(f)=G(\epsilon_{FD}\circ F(f))=G\epsilon_{FD}\circ GF(f)=U'_{GF} (f)$; $L_{GF} V'_{GF}(C)=L_{GF}(C)=F(C)$ and $L_{GF} V'_{GF}(f)=L_{GF}(\eta_D\circ f)=\epsilon_{FD}\circ F(\eta_D\circ f)=F(f)$.
\end{invisible}
In particular, since $K_{GF}\circ L_{GF}=J_{GF}$ and $J_{GF}$ is faithful, then the functor $L_{GF}:GF\text{-}\mathrm{Free}_\cc\to \dd$ is faithful too. Moreover, a morphism $h:F(C)\to F(D)$ in $\dd$ corresponds by adjunction with the morphism $f:=Gh\circ\eta_C:C\to GF(D)$ in $\cc$, i.e. a morphism $f:C\nrightarrow D$ in $GF\text{-}\mathrm{Free}_\cc$ such that $L_{GF}(f)=h$, hence $L_{GF}$ is full as well.\\

The next step is to show that, given an adjunction, the semiseparability of the right adjoint provides an equivalence between the associated Kleisli and Eilenberg-Moore categories, after idempotent completion. As a consequence, these categories are also equivalent up to retracts to the coidentifier category associated to the semiseparable right adjoint.

\begin{prop}\label{prop:KL-HL-equiv-utr}
Let $F\dashv G:\dd\to\cc$ be an adjunction, and consider the diagram \eqref{eq:kleisli-eilenb}. Assume $G$ is a semiseparable functor. Then, the composite functor $K_{GF}\circ L_{GF}:GF\text{-}\mathrm{Free}_\cc\to\cc_{GF}$ is an equivalence up to retracts. Moreover, also the composite $H\circ L_{GF}:GF\text{-}\mathrm{Free}_\cc\to\dd_e$ is an equivalence up to retracts and hence $GF\text{-}\mathrm{Free}_\cc^\natural\cong \dd_e^\natural\cong \cc_{GF}^\natural$.
\end{prop}
\proof
By Theorem \ref{thm:ssepMonad} (i), since $G$ is semiseparable, then the associated monad $(GF,G\epsilon F,\eta)$ is separable. Since the composite functor $K_{GF}\circ L_{GF}:GF\text{-}\mathrm{Free}_\cc\to\cc_{GF}$ equals the canonical functor $J_{GF}:GF\text{-}\mathrm{Free}_\cc\to \cc_{GF}$, by applying Proposition \ref{prop:J:equiv-utr}, we get that it is an equivalence up to retracts.

Moreover, by Proposition \ref{prop:CoidEil} there is a unique functor $\left( K_{GF}\right) _{e}:\dd_{e}\rightarrow \cc_{GF}$ such that $\left( K_{GF}\right) _{e}\circ H=K_{GF}$ and $U_{GF}\circ \left(K_{GF}\right) _{e}=G_{e}$, and in particular $\left( K_{GF}\right) _{e}$ is an equivalence up to retracts, so the fact that $H\circ L_{GF}$ is an equivalence up to retracts follows from the equality $(K_{GF})_e^\natural\circ (H\circ L_{GF})^\natural =(K_{GF}\circ L_{GF})^\natural$.
\endproof

As a consequence of Proposition \ref{prop:KL-HL-equiv-utr}, we recover \cite[Lemma 2.10]{Balm15} (see also \cite[Theorem 5.17 (d)]{Balm11} in the setting of idempotent complete suspended categories), in which the Kleisli and the Eilenberg-Moore comparison functors $L_{GF} : GF\text{-}\mathrm{Free}_\cc\to\dd$ and $K_{GF}:\dd\to\cc_{GF}$ result to be equivalences up to retracts, whenever the counit $\epsilon: FG\to\id_\dd$ of the adjunction $(F,G)$ admits a natural section, i.e. if there is a natural transformation $\xi:\id_\dd\to FG$ such that $\epsilon\circ\xi =\id_{\id_\dd}$. 

Explicitly we have the following:
\begin{cor}[cf. {\cite[Lemma 2.10]{Balm15}}]
Let $F\dashv G:\dd\to\cc$ be an adjunction with $G$ separable. Then, the functors $L_{GF}:GF\text{-}\mathrm{Free}_\cc\to\dd$ and $K_{GF}:\dd\to\cc_{GF}$ are both equivalences up to retracts. Moreover, if $\dd$ is idempotent complete, then $G$ is monadic, i.e. $K_{GF}:\dd\to\cc_{GF}$ is an equivalence.
\end{cor}
\proof
Since $G$ is a separable functor, then, by Corollary \ref{cor:esep}, the associated idempotent natural transformation $e:\id_\dd\to\id_\dd$ is the identity $\id_{\id_\dd}$, and hence the quotient functor $H:\dd\to\dd_\id$ is an equivalence. Thus, by Proposition \ref{prop:KL-HL-equiv-utr}, $L_{GF}:GF\text{-}\mathrm{Free}_\cc\to\dd$ results to be an equivalence up to retracts. 
Concerning $K_{GF}$, it is an equivalence up to retracts, in view of Corollary \ref{cor:compcorefretr}. Furthermore, it is an equivalence if $\dd$ is idempotent complete, by Corollary \ref{cor:compsemi}.
\endproof

\begin{rmk}
A similar result has been obtained in the setting of idempotent complete triangulated categories in \cite[Theorem 1.6]{DS18} where $G$ is only required to be conservative, which is always satisfied by a separable functor (Remark \ref{rmk:Maschke}).
    \end{rmk}

\subsection{Pre-triangulated categories}\label{sub:pretriang}

Our aim here is to extend to semiseparable functors a result obtained by P. Balmer
for separable functors in the context of pre-triangulated categories. First we need to recall the required
definitions. Following \cite[Definition 1.1]{Balm11}, by a \textbf{suspended
category} $(\cc,\Sigma)$ we mean an additive category $\cc$ endowed with an
autoequivalence $\Sigma :\cc\overset{\sim }{\rightarrow }\cc$%
, called the \textbf{suspension}. As in loc. cit., for simplicity we consider $\Sigma $ as an
isomorphism i.e. $\Sigma ^{-1}\circ \Sigma =\id_{\cc}=\Sigma
\circ \Sigma ^{-1}.$

If $\cc$ and $\dd$ are suspended categories, as in \cite[%
Remark 2.7]{Balm11}, when we say that $F\dashv G:\dd\rightarrow
\cc$ is an \textbf{adjunction of functors commuting with suspension}
we mean that both $F$ and $G$ commute with suspension and we tacitly assume
that the unit $\eta $ and counit $\epsilon $ commute with suspension as
well. In this case the monad $\left( GF,G\epsilon F,\eta \right) $ is
\textbf{stable}, meaning that the functor $GF:\cc\rightarrow
\cc,$ the multiplication $G\epsilon F$ and the unit $\eta $ commute
with suspension, see \cite[Definition 2.1]{Balm11}.

Let $(\cc,\Sigma)$ and $(\cc',\Sigma')$ be suspended categories. By
adapting \cite[Definition 3.7]{Balm11}, if a functor $G:\cc^{\prime
}\rightarrow \cc$ commutes with the suspension, i.e. $G\circ \Sigma
^{\prime }=\Sigma \circ G$, we say that $G$ is \textbf{stably semiseparable}
if it is semiseparable through some $\mathcal{P}^G_{X,Y}:\Hom_{%
\cc}\left( GX,GY\right) \rightarrow \Hom_{\cc%
^{\prime }}\left( X,Y\right) $ that commutes with suspension, i.e. such that the diagram
\begin{equation*}
\xymatrixcolsep{1cm}\xymatrixrowsep{0.8cm}\xymatrix{\Hom_\cc(GX,GY)\ar[d]_{\f^{\Sigma}_{GX,GY}}\ar[rr]^{\mathcal{P}^G_{X,Y}}&&\Hom_{\cc'}(X,Y)\ar[d]^{\f^{\Sigma'}_{X,Y}}\\
\Hom_\cc(\Sigma GX,\Sigma GY)\ar@{=}[r]&\Hom_\cc( G\Sigma' X,G\Sigma'Y)\ar[r]^-{\mathcal{P}^G_{\Sigma ^{\prime }X,\Sigma ^{\prime }Y}}&\Hom_{\cc'}(\Sigma'X,\Sigma'Y)}
\end{equation*}
is commutative. In order to simplify the notation all suspensions will be denoted by the
same letter $\Sigma $ from now on.

Given a suspended category $\left( \cc,\Sigma \right) $, by a
(candidate) triangle in $\cc$ (with respect to $\Sigma $) we mean a
diagram of the form
\begin{equation*}
\xymatrix{X\ar[r]^-u&Y\ar[r]^-v&Z\ar[r]^-w&\Sigma X}
\end{equation*} %
A \textbf{pre-triangulated} category $\cc$ is a suspended category  $%
\left( \cc,\Sigma \right) $ together with a class of triangles (with
respect to $\Sigma $) called \textbf{distinguished triangles} subject to the
axioms listed in \cite[Definition 1.3]{Balm11}. This definition is equivalent
to the one given in \cite[Definition 1.1.2]{Ne01} (see the comment after
\cite[Definition 1.3]{Balm11}): We just point out that the requirement that $\Sigma:\cc\to\cc$ is additive, included in \cite[Definition 1.1.2]{Ne01},  is superfluous as $\Sigma$ is part of an adjunction and, if $F\dashv G:\cc\to\dd$ is an adjunction with $\cc$ and $\dd$ additive, then both $F$ and $G$ are additive, see e.g. \cite[Corollary 1.3]{Pop73}.

A functor between pre-triangulated categories is called \textbf{%
exact} if it commutes with the suspension and preserves distinguished
triangles. 
It is well-known that an exact functor of pre-triangulated categories is
additive\footnote{See e.g. \url{https://stacks.math.columbia.edu/tag/05QY}.}.\\


In order to prove the main result of this section, namely Theorem \ref{thm:4.1}, we need the following further results concerning the coidentifier, see Subsection \ref{sub:coidentifier}.

\begin{lem}
\label{lem:coidadd}Let $\cc$ be a category and let $e:\id_{%
\cc}\rightarrow \id_{\cc}$ be an idempotent natural
transformation.

1) If $\cc$ is pointed (i.e. it has a zero object) so is the
coidentifier $\cc_{e}$.

2) If $\cc$ is (pre)additive so is the coidentifier $\cc_{e}$
and the functor $H:\cc\rightarrow \cc_{e}$ is an additive
functor.

\end{lem}

\begin{proof}
Recall that $\cc_e$ is the quotient category $\cc/\!\sim$ where the congruence relation $\sim$ is defined, for all $f,g:A\to B$ by setting $f\sim g$ if and only if $e_B\circ f=e_B\circ g$.

1) Clearly a zero object in $\cc$ is zero also in $\cc_e$.
\begin{invisible}
If $1$ is a terminal object in $\cc$, then the set $\Hom_{\cc}(C,1)$ is a singleton, for
every object $C\in\cc$. Hence $\Hom_{\cc_{e}}(C,1)$ is a singleton
too as a quotient of $\Hom_{\cc}(C,1)$. Thus $1$ is terminal also in $\cc_{e}$. Similarly an initial object in $\cc$ is initial also in $\cc_e$. A zero object is both a terminal and an initial object.
\end{invisible}

2) If $\cc$ is (pre)additive, for any $A,B\in\cc$ the set $\Hom_{\cc%
}(A,B)$ is an abelian group via a binary operation $+$. Note that $\sim$ is an additive congruence relation. In fact, for all $f,g,f',g':A\to B$, if $f\sim f'$ and $g\sim g'$, then  $e_{B}\circ f=e_{B}\circ f^{\prime }$ and $e_{B}\circ g=e_{B}\circ
g^{\prime }$ so that $e_{B}\circ \left( f+g\right) =e_{B}\circ f+e_{B}\circ
g=e_{B}\circ f^{\prime }+e_{B}\circ g^{\prime }=e_{B}\circ \left( f^{\prime
}+g^{\prime }\right) $ and hence $f+g\sim f'+g'$. As a consequence it is well-known that the quotient is also (pre)additive and the quotient functor $H$ is an additive functor.\qedhere

\begin{invisible}
Assume that $\cc$ is preadditive. If $f,g,f^{\prime
}g^{\prime }\in \Hom_{\cc}(A,B)$ are such that $\overline{f}=%
\overline{f^{\prime }}$ and $\overline{g}=\overline{g^{\prime }},$ then we
get $f\sim f$ and $g\sim g^{\prime }$ so that $f+g\sim f^{\prime
}+g^{\prime } $ i.e. $\overline{f+g}=\overline{f^{\prime
}+g^{\prime }}.$ Hence we can define a binary operation on$\ \Hom_{%
\cc_{e}}(A,B)$ by setting $\overline{f}+\overline{g}:=\overline{f+g}$
which makes $\Hom_{\cc_{e}}(A,B)$ an abelian group as well.
Moreover $\left( \overline{f}+\overline{g}\right) \circ \overline{h}=%
\overline{(f+g)}\circ \overline{h}=\overline{\left( f+g\right) \circ h}=%
\overline{f\circ h+g\circ h}=\overline{f\circ h}+\overline{g\circ h}=%
\overline{f}\circ \overline{h}+\overline{g}\circ \overline{h}$ and similarly
on the other side. Thus $\cc_{e}$ is preadditive too. Since, by
construction, $H\left( f+g\right) =\overline{f+g}=\overline{f}+\overline{g}%
=Hf+Hg,$ we have that the functor $H:\cc\rightarrow \cc_{e}$
is an additive functor.

Assume that $\cc$ is additive. By $1)$ $\cc_e$ has a zero object. We have to check that it also has binary biproducts.
Let $A$ and $B$ be objects in $\cc$. Then we have the binary
biproduct $\left( A\oplus B,p_{A},p_{B},s_{A},s_{B}\right) $ in $\cc$%
. Then $\left( A\oplus B,\overline{p_{A}},\overline{p_{B}},\overline{s_{A}},%
\overline{s_{B}}\right) $ is the binary biproduct of $A$ and $B$ in $%
\cc_{e}$. In fact $\overline{p_{A}}\circ \overline{s_{A}}=\overline{%
p_{A}\circ s_{A}}=\overline{\id_{A}},\overline{p_{B}}\circ \overline{%
s_{B}}=\overline{p_{B}\circ s_{B}}=\overline{\id_{B}},\overline{p_{A}%
}\circ \overline{s_{B}}=\overline{p_{A}\circ s_{B}}=\overline{0},\overline{%
p_{B}}\circ \overline{s_{A}}=\overline{p_{B}\circ s_{A}}=\overline{0}$ and $%
\overline{s_{A}}\circ \overline{p_{A}}+\overline{s_{B}}\circ \overline{p_{B}}%
=\overline{s_{A}\circ p_{A}}+\overline{s_{B}\circ p_{B}}=\overline{%
s_{A}\circ p_{A}+s_{B}\circ p_{B}}=\overline{\id_{A\oplus B}}$. This
proves that $\cc_{e}$ has a zero object and binary biproducts so that $%
\cc_{e}$ is additive.
\end{invisible}
\end{proof}

\begin{lem}
\label{lem:coidSusp}Let $\cc$ be a category and let $e:\id_{%
\cc}\rightarrow \id_{\cc}$ be an idempotent natural
transformation. If $\cc$ has an endofunctor $\Sigma $ such that $%
\Sigma e=e\Sigma $, then the coidentifier $\cc_{e}$ has an
endofunctor $\Sigma _{e}$ such that $H\circ \Sigma =\Sigma _{e}\circ H$, where $H:\cc\to\cc_e$ is the quotient functor.
Moreover, $\Sigma _{e}$ is an additive functor whenever $\Sigma $ is.
\end{lem}

\begin{proof}
We have $H\Sigma e=He\Sigma =\id_{H}\circ \Sigma =\id%
_{H\Sigma }$ so that, by Lemma \ref{lem:coidentifier}, there is a unique
functor $\Sigma _{e}:\cc_{e}\rightarrow \cc_{e}$ 
such that $%
H\circ \Sigma =\Sigma _{e}\circ H$. Since $H$ acts as the identity on
objects, we get that $\Sigma _{e}$ acts as $\Sigma $ on objects. Moreover $%
\Sigma _{e}\overline{f}=\Sigma _{e}Hf=H\Sigma f=\overline{\Sigma f}.$
Since $\Sigma _{e}\left( \overline{f}+\overline{g}\right) =\Sigma _{e}\left(
\overline{f+g}\right) =\overline{\Sigma \left( f+g\right) }=\overline{\Sigma
f+\Sigma g}=\overline{\Sigma f}+\overline{\Sigma g}=\Sigma _{e}\overline{f}%
+\Sigma _{e}\overline{g},$ we get that $\Sigma _{e}$ is an additive functor
if so is $\Sigma $.
\end{proof}

\begin{lem}\label{lem:asidpsusp}
Let  $F:\cc\to\dd$ be a stably semiseparable functor. Then, the associated idempotent natural transformation commutes with the suspension.
\end{lem}

\begin{proof}
By definition, $F$ is semiseparable through some $\mathcal{P}^F$ such
that $\mathcal{P}^F_{\Sigma X,\Sigma Y}\circ \f^\Sigma _{FX,FY}=\f^\Sigma
_{X,Y}\circ \mathcal{P}^F_{X,Y}$. Consider the associated idempotent natural transformation $e:\id_\cc\to\id_\cc$ which is defined by setting $e_{X}:=\mathcal{P}^F_{X,X}\left( \id_{FX}\right)$ for every $X$ in $\cc$. Then
$\Sigma e_{X}=\f^\Sigma _{X,X}\mathcal{P}^F_{X,X}\left( \id_{FX}\right) =%
\mathcal{P}^F_{\Sigma X,\Sigma X}\f^\Sigma _{FX,FY}\left( \id_{FX}\right)
=\mathcal{P}^F_{\Sigma X,\Sigma X}\Sigma \left( \id_{FX}\right) =%
\mathcal{P}^F_{\Sigma X,\Sigma X}\left( \id_{\Sigma FX}\right) =%
\mathcal{P}^F_{\Sigma X,\Sigma X}\left( \id_{F\Sigma X}\right)
=e_{\Sigma X}$
and hence $\Sigma e=e\Sigma $, i.e. $e$ commutes with the suspension.
\end{proof}

We are now ready to prove our announced semi-analogue of Balmer's \cite[Theorem 4.1]%
{Balm11}.

\begin{thm}
\label{thm:4.1}Let $\cc$ be a pre-triangulated category and let $%
\dd$ be an idempotent complete suspended category. Let $F\dashv G:%
\dd\rightarrow \cc$ be an adjunction of functors commuting
with the suspension. Suppose that the stable monad $GF:\cc\rightarrow
\cc$ is an exact functor and that $G$ is a stably semiseparable functor. Then,
the coidentifier $\dd_{e}$ is idempotent complete and
pre-triangulated with distinguished triangles being exactly the ones whose
image through the functor $G_{e}:\dd_{e}\rightarrow \cc$ (determined by the factorization $G=G_e\circ H$) is
distinguished in $\cc$. Moreover, with respect to this
pre-triangulation, both functors $G_{e}:\dd_{e}\rightarrow \mathcal{C%
}$ and its left adjoint $F_{e}:\cc\rightarrow \dd_{e}$ become exact.
\begin{invisible}
and $G$
becomes additive.
\end{invisible}
\end{thm}

\begin{proof}
Since $G$ is stably semiseparable, by Lemma \ref{lem:asidpsusp}, the associated idempotent natural transformation $e:\id_\cc\to\id_\cc$ commutes with the suspension, i.e. $e\Sigma=\Sigma e$.
By Lemma \ref%
{lem:coidadd}, the coidentifier $\dd_{e}$ is additive and, by Lemma %
\ref{lem:coididp}, it is idempotent complete. By Lemma \ref{lem:coidSusp}, the coidentifier
$\dd_{e}$ has an endofunctor $\Sigma _{e}$ such that $H\circ \Sigma
=\Sigma _{e}\circ H$. 
From $\Sigma e=e\Sigma $ we deduce $e\Sigma ^{-1}=\Sigma
^{-1}e $ so that we also have an endofunctor $\Sigma _{e}^{-1}$ such that $%
H\circ \Sigma ^{-1}=\Sigma _{e}^{-1}\circ H.$ We compute $\Sigma _{e}\circ
\Sigma _{e}^{-1}\circ H=\Sigma _{e}\circ H\circ \Sigma ^{-1}=H\circ \Sigma
\circ \Sigma ^{-1}=H=\id_{\dd_{e}}\circ H$ and hence $\Sigma
_{e}\circ \Sigma _{e}^{-1}=\id_{\dd_{e}}$ in view of Lemma %
\ref{lem:coidentifier}. Similarly $\Sigma _{e}^{-1}\circ \Sigma _{e}=\mathrm{%
Id}_{\dd_{e}}$, so that $\Sigma _{e}$ is an isomorphism.

Since $G$ is semiseparable, by Theorem \ref{thm:coidentifier} it
factorizes as $G=G_{e}\circ H$ for a unique separable functor $G_{e}:%
\dd_{e}\rightarrow \cc$. Moreover, $G_e$ is separable via $\p^{G_e}$ defined by $\p^{G_e}_{HX,HY}:=\f^H_{X,Y}\circ\p^G_{X,Y}$ for all $X,Y$ in $\dd$.
Since $G$ commutes with the suspension,
we have $G_{e}\circ \Sigma _{e}\circ H=G_{e}\circ H\circ \Sigma =G\circ
\Sigma =\Sigma \circ G=\Sigma \circ G_{e}\circ H$ and hence $G_{e}\circ
\Sigma _{e}=\Sigma \circ G_{e}$, i.e. $G_{e}$ commutes with the suspension as
well. Now consider the composite functor $F_e=H\circ F:\cc\to\dd_e$, which is the left adjoint of $G_e$ with unit $\eta_e$ and counit $\epsilon_e$ given as in Lemma \ref{lem:coidadj}. Then,  $\Sigma _{e}\circ F_{e}=\Sigma _{e}\circ H\circ F=H\circ
\Sigma \circ F=H\circ F\circ \Sigma =F_{e}\circ \Sigma $ so that $F_{e}$
commutes with the suspension too. Note that $\epsilon _{e}\Sigma _{e}H=\epsilon
_{e}H\Sigma =H\epsilon \Sigma =H\Sigma \epsilon =\Sigma _{e}H\epsilon
=\Sigma _{e}\epsilon _{e}H$ so that $\epsilon _{e}\Sigma _{e}=\Sigma
_{e}\epsilon _{e}.$ Moreover $\eta _{e}\Sigma =\eta \Sigma =\Sigma \eta
=\Sigma \eta _{e}$. Thus also the unit and counit of the adjunction $\left(
F_{e},G_{e}\right) $ commute with the suspensions. Hence $F_{e}\dashv G_{e}$
is what we called an adjunction of functors commuting with suspension. By
Lemma \ref{lem:coidadj}, the adjunctions $\left( F_{e},G_{e}\right) $ and $%
\left( F,G\right) $ have the same associated monad. As a consequence, we get
that $G_{e}\circ F_{e}$ is a stable monad and an exact functor by assumption.
We have $\f^{\Sigma_e}_{HX,HY}\p^{G_e}_{HX,HY}=\f^{\Sigma_e}_{HX,HY}\f^H_{X,Y}\p^G_{X,Y}
=\f^{\Sigma_eH}_{X,Y}\p^G_{X,Y}=\f^{H\Sigma}_{X,Y}\p^G_{X,Y}=\f^{H}_{\Sigma X,\Sigma Y}\f^{\Sigma}_{X,Y}\p^G_{X,Y}
=\f^{H}_{\Sigma X,\Sigma Y}\p^G_{\Sigma X,\Sigma Y}\f^{\Sigma}_{GX,GY}
=\p^{G_e}_{H\Sigma X,H\Sigma Y}\f^{\Sigma}_{G_eHX,G_eHY}
=\p^{G_e}_{\Sigma_e HX,\Sigma_e HY}\f^{\Sigma}_{G_eHX,G_eHY}$ for all $X,Y$ in $\dd$. Since $H$ is surjective on objects, this means $\f^{\Sigma_e}_{X,Y}\p^{G_e}_{X,Y}=\p^{G_e}_{\Sigma_e X,\Sigma_e Y}\f^{\Sigma}_{G_eX,G_eY}$ for all $X,Y$ in $\dd_e$ i.e. that $G_e$ is a stably separable functor.

Then we can apply \cite[Theorem 4.1]{Balm11} to the adjunction $F_{e}\dashv
G_{e}:\dd_{e}\rightarrow \cc.$ As a consequence, the
coidentifier $\dd_{e}$ is pre-triangulated with distinguished
triangles $\Delta $ being exactly the ones whose image $G_e(\Delta)$ through the functor $G_{e}:%
\dd_{e}\rightarrow \cc$ is distinguished in $\cc$.
Moreover, with respect to this pre-triangulation, both functors $G_{e}:%
\dd_{e}\rightarrow \cc$ and $F_{e}:\cc\rightarrow
\dd_{e}$ become exact.
\begin{invisible} [Ho nascosto questa parte perché ho già scritto prima che le aggiunzioni tra categorie additive sono fatte di funtori additivi ed è quindi ovvio che $G_{e}:\dd_{e}\rightarrow \cc$ sia additivo.]
Since an exact functor of pre-triangulated categories is
additive, we get that $G_{e}:%
\dd_{e}\rightarrow \cc$ is additive. Since, by Lemma \ref{lem:coidadd}, the functor $H$ is also additive, we deduce that $%
G=G_{e}\circ H$ is additive as well.
\end{invisible}
\end{proof}

\begin{rmk}
\label{rem:EMsuspended}In \cite[Definition 2.4]{Balm11}, it is claimed that,
when $\cc$ is a suspended category and $\top $ an additive stable
monad on it, then the Eilenberg-Moore category $\cc_{\top }$
inherits a structure of suspended category such that $V_{\top }\dashv
U_{\top }:\cc_{\top }\rightarrow \cc$ is an adjunction of
additive functors commuting with suspension. Explicitly the suspension $%
\Sigma _{\top }:\cc_{\top }\rightarrow \cc_{\top }$ is
defined on objects by setting $\Sigma _{\top }\left( C,\mu \right) :=\left(
\Sigma C,\Sigma \mu \right) $ and on morphisms by $\Sigma _{\top }f:=\Sigma
f $.

\begin{invisible}
Balmer does not require $\top $ to be additive in \cite[Definition 2.4]%
{Balm11}, but when he uses the property above $\top $ is automatically
additive.

Let $\left( \top ,m:\top \top \rightarrow \top ,\eta :\id_{\mathcal{C%
}}\rightarrow \top \right) $ be a monad on $\cc$. First note that $%
\cc_{\top }$ is additive. In fact, given $f,g:\left( C,\mu \right)
\rightarrow \left( C^{\prime },\mu ^{\prime }\right) $ we have that
\begin{eqnarray*}
\left( U_{\top }f+U_{\top }g\right) \circ \mu &=&U_{\top }f\circ \mu
+U_{\top }g\circ \mu =\mu ^{\prime }\circ \top U_{\top }f+\mu ^{\prime
}\circ \top U_{\top }g=\mu ^{\prime }\circ \left( \top U_{\top }f+\top
U_{\top }g\right) \\
&=&\mu ^{\prime }\circ \left( \top U_{\top }f+\top U_{\top }g\right) \overset%
{\top \text{ additive}}{=}\mu ^{\prime }\circ \top \left( U_{\top }f+U_{\top
}g\right)
\end{eqnarray*}%
so that there is a unique morphism $f+g:\left( C,\mu \right) \rightarrow
\left( C^{\prime },\mu ^{\prime }\right) $ such that $U_{\top }\left(
f+g\right) =U_{\top }f+U_{\top }g$. Then $U_{\top }\left( g+f\right)
=U_{\top }g+U_{\top }f=U_{\top }f+U_{\top }g=U_{\top }\left( f+g\right) $
and hence, since $U_{\top }$ is faithful, we obtain $g+f=f+g.$

Moreover $U_{\top }\left( \left( f+g\right) +h\right) =U_{\top }\left(
f+g\right) +U_{\top }h=\left( U_{\top }f+U_{\top }g\right) +U_{\top
}h=U_{\top }f+\left( U_{\top }g+U_{\top }h\right) =U_{\top }f+U_{\top
}\left( g+h\right) =U_{\top }\left( f+\left( g+h\right) \right) $ and hence,
since $U_{\top }$ is faithful, we obtain $\left( f+g\right) +h=f+\left(
g+h\right) .$ Note that, given objects $\left( C,\mu \right) $ and $\left(
C^{\prime },\mu ^{\prime }\right) $ in $\cc_{\top }$ and the zero
morphism $0_{C,C^{\prime }}:C\rightarrow C^{\prime },$ then
\begin{equation*}
0_{C,C^{\prime }}\circ \mu =0_{\top C,C^{\prime }}=\mu ^{\prime }\circ
0_{\top C,\top C^{\prime }}\overset{\top \text{ additive}}{=}\mu ^{\prime
}\circ \top \left( 0_{C,C^{\prime }}\right)
\end{equation*}%
and hence there is a unique morphism $0_{\left( C,\mu \right) ,\left(
C^{\prime },\mu ^{\prime }\right) }:\left( C,\mu \right) \rightarrow \left(
C^{\prime },\mu ^{\prime }\right) $ such that $U_{\top }0_{\left( C,\mu
\right) ,\left( C^{\prime },\mu ^{\prime }\right) }=0_{C,C^{\prime }}$. Hence%
\begin{equation*}
U_{\top }\left( f+0_{\left( C,\mu \right) ,\left( C^{\prime },\mu ^{\prime
}\right) }\right) =U_{\top }f+U_{\top }0_{\left( C,\mu \right) ,\left(
C^{\prime },\mu ^{\prime }\right) }=U_{\top }f+0_{C,C^{\prime }}=U_{\top }f
\end{equation*}%
and hence $f+0_{\left( C,\mu \right) ,\left( C^{\prime },\mu ^{\prime
}\right) }=f.$ Therefore $\Hom_{\cc_{\top }}\left( \left(
C,\mu \right) ,\left( C^{\prime },\mu ^{\prime }\right) \right) $ is an
abelian group. Moreover
\begin{eqnarray*}
U_{\top }\left( \left( f+g\right) \circ h\right) &=&U_{\top }\left(
f+g\right) \circ U_{\top }h=\left( U_{\top }f+U_{\top }g\right) \circ
U_{\top }h=U_{\top }f\circ U_{\top }h+U_{\top }g\circ U_{\top }h \\
&=&U_{\top }\left( f\circ h\right) +U_{\top }\left( g\circ h\right) =U_{\top
}\left( f\circ h+g\circ h\right)
\end{eqnarray*}%
so that $\left( f+g\right) \circ h=f\circ h+g\circ h$ and similarly on the
other side. Thus $\cc_{\top }$ is preadditive category and $U_{\top
}:\cc_{\top }\rightarrow \cc$ is an additive functor by
construction. In order to check that $\cc_{\top }$ is additive it
remains to prove it has zero object and binary biproducts. Concerning
biproducts, by \cite[Proposition 1.2.4]{BorII94} it suffices to notice that
$U_{\top }:\cc_{\top }\rightarrow \cc$ since we have the
product of two objects in $\cc$ we also have the same type of product
in $\cc_{\top }$ by \cite[Proposition 4.3.1]{BorII94}. If $0_{%
\cc}$ is a zero object in $\cc$, then, since $V_{\top
}\dashv U_{\top }$, we get that $V_{\top }0_{\cc}$ is an initial
object $\cc_{\top }.$ Since $\cc_{\top }$ is preadditive, by
the proof of \cite[Proposition 1.2.3]{BorII94}, we get that $V_{\top }0_{%
\cc}$ is a zero object in $\cc_{\top }$. Hence $\cc%
_{\top }$ is additive as desired. [FORSE\ LA\ PARTE\ PRECEDENTE\ E'
DIMOSTRATA\ ESPLICITAMENTE\ DA\ QUALCHE\ PARTE]

Let us check that $V_{\top }:\cc\rightarrow \cc_{\top }$ is
an additive functor. Given $f,g:C\rightarrow C^{\prime },$ we compute%
\begin{eqnarray*}
U_{\top }\left( V_{\top }f+V_{\top }g\right) &=&U_{\top }V_{\top }f+U_{\top
}V_{\top }g=U_{\top }V_{\top }f+U_{\top }V_{\top }g \\
&=&\top f+\top g\overset{\top \text{ additive}}{=}\top \left( f+g\right)
=U_{\top }V_{\top }\left( f+g\right)
\end{eqnarray*}%
and hence, since $U_{\top }$ is faithful, we obtain $V_{\top }f+V_{\top
}g=V_{\top }\left( f+g\right) $ as desired.

Let us construct a suspension on $\cc_{\top }$. Given an object $%
\left( C,\mu :\top C\rightarrow C\right) $ in $\cc_{\top },$ we have
that
\begin{eqnarray*}
\mu \circ m_{C} &=&\mu \circ \top \mu \Rightarrow \Sigma \mu \circ \Sigma
m_{C}=\Sigma \mu \circ \Sigma \top \mu \Rightarrow \Sigma \mu \circ
m_{\Sigma C}=\Sigma \mu \circ \top \Sigma \mu \\
\mu \circ \eta _{C} &=&\id_{C}\Rightarrow \Sigma \mu \circ \Sigma
\eta _{C}=\Sigma \id_{C}\Rightarrow \Sigma \mu \circ \eta _{\Sigma
C}=\id_{\Sigma C}
\end{eqnarray*}%
so that $\left( \Sigma C,\Sigma \mu :\Sigma \top C=\top \Sigma C\rightarrow
\top C\right) $ is an object in $\cc_{\top }$ and we can set $\Sigma
_{\top }\left( C,\mu \right) :=\left( \Sigma C,\Sigma \mu \right) .$ Given a
morphism $f:\left( C,\mu \right) \rightarrow \left( C^{\prime },\mu ^{\prime
}\right) $ in $\cc_{\top },$ then%
\begin{equation*}
f\circ \mu =\mu ^{\prime }\circ \top f\Rightarrow \Sigma f\circ \Sigma \mu
=\Sigma \mu ^{\prime }\circ \Sigma \top f\Rightarrow \Sigma f\circ \Sigma
\mu =\Sigma \mu ^{\prime }\circ \top \Sigma f
\end{equation*}%
so that we get a morphism $\Sigma f:\left( \Sigma C,\Sigma \mu \right)
\rightarrow \left( \Sigma C^{\prime },\Sigma \mu ^{\prime }\right) $ in $%
\cc_{\top }$. This defines a functor $\Sigma _{\top }:\cc%
_{\top }\rightarrow \cc_{\top }$. Similarly one defines $\Sigma
_{\top }^{-1}:=\left( \Sigma ^{-1}\right) _{\top }:\cc_{\top
}\rightarrow \cc_{\top }$ so that%
\begin{eqnarray*}
\Sigma _{\top }\Sigma _{\top }^{-1}\left( C,\mu \right) &=&\Sigma _{\top
}\left( \Sigma ^{-1}C,\Sigma ^{-1}\mu \right) =\left( \Sigma \Sigma
^{-1}C,\Sigma \Sigma ^{-1}\mu \right) =\left( C,\mu \right) , \\
\Sigma _{\top }\Sigma _{\top }^{-1}f &=&\Sigma _{\top }\left( \Sigma
^{-1}f\right) =\Sigma \Sigma ^{-1}f=f, \\
\Sigma _{\top }^{-1}\Sigma _{\top }\left( C,\mu \right) &=&\Sigma _{\top
}^{-1}\left( \Sigma C,\Sigma \mu \right) =\left( \Sigma ^{-1}\Sigma C,\Sigma
^{-1}\Sigma \mu \right) =\left( C,\mu \right) , \\
\Sigma _{\top }^{-1}\Sigma _{\top }f &=&\Sigma _{\top }^{-1}\left( \Sigma
f\right) =\Sigma ^{-1}\Sigma f=f
\end{eqnarray*}%
and hence $\Sigma _{\top }\Sigma _{\top }^{-1}=\id_{\cc%
_{\top }}=\Sigma _{\top }^{-1}\Sigma _{\top }$.

Let us check that $U_{\top }:\cc_{\top }\rightarrow \cc$
commutes with suspension i.e. that $U_{\top }\Sigma _{\top }=\Sigma U_{\top
}:$
\begin{eqnarray*}
U_{\top }\Sigma _{\top }\left( C,\mu \right) &=&U_{\top }\left( \Sigma
C,\Sigma \mu \right) =\Sigma C=\Sigma U_{\top }\left( C,\mu \right) , \\
U_{\top }\Sigma _{\top }\left( \left( C,\mu \right) \overset{f}{\rightarrow }%
\left( C^{\prime },\mu ^{\prime }\right) \right) &=&U_{\top }\left( \left(
\Sigma C,\Sigma \mu \right) \overset{\Sigma f}{\rightarrow }\left( \Sigma
C^{\prime },\Sigma \mu ^{\prime }\right) \right) \\
&=&\left( \Sigma C\overset{\Sigma f}{\rightarrow }\Sigma C^{\prime }\right)
=\Sigma \left( C\overset{f}{\rightarrow }C^{\prime }\right) =\Sigma U_{\top
}\left( \left( C,\mu \right) \overset{f}{\rightarrow }\left( C^{\prime },\mu
^{\prime }\right) \right) .
\end{eqnarray*}%
Let us check that $\Sigma _{\top }:\cc_{\top }\rightarrow \cc%
_{\top }$ is additive. Given morphisms $f,g:\left( C,\mu \right) \rightarrow
\left( C^{\prime },\mu ^{\prime }\right) $ in $\cc_{\top },$%
\begin{equation*}
U_{\top }\left( \Sigma _{\top }f+\Sigma _{\top }g\right) =U_{\top }\Sigma
_{\top }f+U_{\top }\Sigma _{\top }g=\Sigma U_{\top }f+\Sigma U_{\top
}g=\Sigma \left( U_{\top }f+U_{\top }g\right) =\Sigma U_{\top }\left(
f+g\right) =U_{\top }\Sigma _{\top }\left( f+g\right)
\end{equation*}%
so that $\Sigma _{\top }f+\Sigma _{\top }g=\Sigma _{\top }\left( f+g\right)
. $

Let us check that $V_{\top }:\cc\rightarrow \cc_{\top }$
commutes with suspension i.e. that $\Sigma _{\top }V_{\top }=V_{\top }\Sigma
:$%
\begin{eqnarray*}
\Sigma _{\top }V_{\top }C &=&\Sigma _{\top }\left( \top C,m_{C}\right)
=\left( \Sigma \top C,\Sigma m_{C}\right) =\left( \top \Sigma C,m_{\Sigma
C}\right) =V_{\top }\Sigma C, \\
\Sigma _{\top }V_{\top }\left( C\overset{f}{\rightarrow }C^{\prime }\right)
&=&\Sigma _{\top }\left( \left( \top C,m_{C}\right) \overset{\top f}{%
\rightarrow }\left( \top C^{\prime },m_{C^{\prime }}\right) \right) =\left(
\left( \Sigma \top C,\Sigma m_{C}\right) \overset{\Sigma \top f}{\rightarrow
}\left( \Sigma \top C^{\prime },\Sigma m_{C^{\prime }}\right) \right) \\
&=&\left( \left( \top \Sigma C,m_{\Sigma C}\right) \overset{\top \Sigma f}{%
\rightarrow }\left( \top \Sigma C^{\prime },m_{\Sigma C^{\prime }}\right)
\right) =V_{\top }\left( \Sigma C\overset{\Sigma f}{\rightarrow }\Sigma
C^{\prime }\right) =V_{\top }\Sigma \left( C\overset{f}{\rightarrow }%
C^{\prime }\right) .
\end{eqnarray*}%
The unit of $\left( V_{\top },U_{\top }\right) $ is $\eta $ which commutes
with suspension. The counit is $\beta :U_{\top }V_{\top }\rightarrow \mathrm{%
Id}$ which is uniquely determined by the equality $U_{\top }\beta _{\left(
C,\mu \right) }=\mu $. Thus%
\begin{equation*}
U_{\top }\Sigma _{\top }\beta _{\left( C,\mu \right) }=\Sigma U_{\top }\beta
_{\left( C,\mu \right) }=\Sigma \mu =U_{\top }\beta _{\left( \Sigma C,\Sigma
\mu \right) }=U_{\top }\beta _{\Sigma _{\top }\left( C,\mu \right) }
\end{equation*}%
and hence, since $U_{\top }$ is faithful, we get $\Sigma _{\top }\beta
_{\left( C,\mu \right) }=\beta _{\Sigma _{\top }\left( C,\mu \right) }$ i.e.
$\Sigma _{\top }\beta =\beta \Sigma _{\top }.$ Thus the counit commutes with
suspension. We have so proved that $V_{\top }\dashv U_{\top }:\cc%
_{\top }\rightarrow \cc$ is an adjunction of functors commuting with
suspension.
\end{invisible}
\end{rmk}

Given a monad $\top $ on a triangulated category $\cc$, in
\cite{DS18} the authors investigate whether the Eilenberg-Moore category $%
\cc_{\top }$ inherits the structure of triangulated category from $%
\cc$. They also claim this seems to rarely occur in Nature, quoting
\cite{Balm11} as a particular occurrence. In the following result $\cc%
_{GF}$ inherits the structure of pre-triangulated category from $\cc$.

\begin{cor}\label{coro:EMpretriang}
Let $\cc$ be a pre-triangulated category and let $\dd$ be an
idempotent complete suspended category. Let $F\dashv G:\dd%
\rightarrow \cc$ be an adjunction of functors commuting with
suspension. Suppose that the stable monad $GF:\cc\rightarrow
\cc$ is an exact functor and that $G$ is a stably semiseparable functor. Then,
the Eilenberg-Moore category $\cc_{GF}$ is idempotent complete and
pre-triangulated with distinguished triangles being exactly the ones whose
image through the forgetful functor $U_{GF}:\cc_{GF}\rightarrow
\cc$ is distinguished in $\cc$. Moreover, with respect to
this pre-triangulation, both the functor $U_{GF}:\cc_{GF}\rightarrow
\cc$ and its left adjoint $V_{GF}:\cc\rightarrow \cc%
_{GF}$ become exact. Furthermore, there is a unique exact equivalence of
categories $\left( K_{GF}\right) _{e}:\dd_{e}\rightarrow \cc%
_{GF}$ such that $\left( K_{GF}\right) _{e}\circ H=K_{GF}$ and $U_{GF}\circ
\left( K_{GF}\right) _{e}=G_{e}$.
\end{cor}

\begin{proof}
By Proposition \ref{prop:CoidEil}, there is a unique functor $\left(
K_{GF}\right) _{e}:\dd_{e}\rightarrow \cc_{GF}$ such that $%
\left( K_{GF}\right) _{e}\circ H=K_{GF}$ and $U_{GF}\circ \left(
K_{GF}\right) _{e}=G_{e}$. Moreover, since $\dd$ is idempotent
complete, then the functor $\left( K_{GF}\right) _{e}$ is an equivalence of
categories. By Lemma \ref{lem:coididp}, $\dd_{e}$ is idempotent
complete so that also $\cc_{GF}$ becomes idempotent complete.
Note that, since $\cc$ is pre-triangulated, it is suspended. Since $%
F\dashv G:\dd\rightarrow \cc$ is an adjunction of functors
commuting with suspension, the monad $\left( GF,G\epsilon F,\eta \right) $
is stable. Moreover the functor $GF:\cc\rightarrow \cc$ is
additive being an exact functor between pre-triangulated categories.
Thus, by Remark \ref{rem:EMsuspended}, the Eilenberg-Moore category $%
\cc_{GF}$ inherits a structure of suspended category through the suspension $\Sigma_{GF}$ such that $%
V_{GF}\dashv U_{GF}:\cc_{GF}\rightarrow \cc$ is an
adjunction of additive functors commuting with suspension. Also the comparison functor $K_{GF}:\dd\to\cc_{GF}$ commutes with suspension.
Note that the monad $\left( GF,G\epsilon F,\eta \right) $ is separable
in view of Theorem \ref{thm:ssepMonad}. By construction this separability is given by the section $\sigma:=G\gamma F:GF\to GFGF$ where $\gamma:\id\to FG$ is defined by $\gamma_X:=\p_{X,FGX}(\eta_{GX})$.
We noticed it is stable. Thus $\sigma_{\Sigma X}=G\gamma_{F\Sigma X}=G\p_{F\Sigma X,FGF\Sigma X}(\eta_{GF\Sigma X})=G\p_{\Sigma F X,\Sigma FGF X}(\eta_{\Sigma GF X})=G\p_{\Sigma F X,\Sigma FGF X}(\Sigma\eta_{ GF X})=G\Sigma\p_{ F X, FGF X}(\eta_{ GF X})=G\Sigma\gamma_{FX}=\Sigma G\gamma_{FX}=\Sigma\sigma_X$ and hence $\sigma$  commutes with suspension, obtaining that it is a stably separable monad in the sense of \cite[Definition 3.5]{Balm11}. By \cite[%
Proposition 3.11]{Balm11}, this means that $U_{GF}:\cc_{GF}\rightarrow
\cc$ is a stably separable functor.

Then \cite[Theorem 4.1]{Balm11}, applied to the adjunction $\left(
V_{GF},U_{GF}\right) $, yields a pre-triangulation on $\cc_{GF}$ with
distinguished triangles $\Delta $ being exactly the ones such that $%
U_{GF}\left( \Delta \right) $ is distinguished in $\cc$. Moreover,
with respect to this pre-triangulation, both functors $U_{GF}$ and $V_{GF}$
become exact.

Coming back to the equivalence of categories $\left( K_{GF}\right) _{e}:%
\dd_{e}\rightarrow \cc_{GF}$, note that $\Sigma _{GF}\circ
\left( K_{GF}\right) _{e}\circ H=\Sigma _{GF}\circ K_{GF}=K_{GF}\circ \Sigma
=\left( K_{GF}\right) _{e}\circ H\circ \Sigma =\left( K_{GF}\right)
_{e}\circ \Sigma _{e}\circ H$ and hence $\Sigma _{GF}\circ \left(
K_{GF}\right) _{e}=\left( K_{GF}\right) _{e}\circ \Sigma _{e}$, i.e. $\left(
K_{GF}\right) _{e}$ commutes with suspension.

\begin{invisible}
Above we used the equality $\Sigma _{GF}\circ K_{GF}=K_{GF}\circ \Sigma $
that we check here:%
\begin{equation*}
\left( \Sigma _{GF}\circ K_{GF}\right) D=\Sigma _{GF}\left( GD,G\epsilon
_{D}\right) =\left( \Sigma GD,\Sigma G\epsilon _{D}\right) =\left( G\Sigma
D,G\Sigma \epsilon _{D}\right) =\left( G\Sigma D,G\epsilon _{\Sigma
D}\right) =\left( K_{GF}\circ \Sigma \right) D
\end{equation*}%
and $\left( \Sigma _{GF}\circ K_{GF}\right) f=\Sigma _{GF}\left( Gf\right)
=\Sigma Gf=G\Sigma f=\left( K_{GF}\circ \Sigma \right) f.$
\end{invisible}

Since an exact functor of pre-triangulated categories is
additive, the functor $G_{e}$ is additive as it is exact in view of
Theorem \ref{thm:4.1}. Thus, given morphisms $f,g:D\rightarrow D^{\prime }$
in $\dd$, we have
\begin{eqnarray*}
U_{GF}\left( \left( K_{GF}\right) _{e}\left( \overline{f}\right) +\left(
K_{GF}\right) _{e}\left( \overline{g}\right) \right) &=&U_{GF}\left(
K_{GF}\right) _{e}\left( \overline{f}\right) +U_{GF}\left( K_{GF}\right)
_{e}\left( \overline{g}\right) \\
&=&G_{e}\left( \overline{f}\right) +G_{e}\left( \overline{g}\right)
=G_{e}\left( \overline{f}+\overline{g}\right) =U_{GF}\left( \left(
K_{GF}\right) _{e}\left( \overline{f}+\overline{g}\right) \right)
\end{eqnarray*}%
so that $\left( K_{GF}\right) _{e}\left( \overline{f}\right) +\left(
K_{GF}\right) _{e}\left( \overline{g}\right) =\left( K_{GF}\right)
_{e}\left( \overline{f}+\overline{g}\right) $ and hence $\left( K_{GF}\right) _{e}$
is additive.
To check that $\left( K_{GF}\right) _{e}$ is exact, it remains to
prove that it preserves distinguished triangles. Let $\Delta $ be a
distinguished triangle in $\dd_{e}$. Then, by Theorem \ref{thm:4.1},
$G_{e}\left( \Delta \right) $ is distinguished in $\cc$. Since $%
U_{GF}\circ \left( K_{GF}\right) _{e}=G_{e},$ we get that $U_{GF}\left(
\left( K_{GF}\right) _{e}\left( \Delta \right) \right) $ is distinguished in
$\cc$. By definition of pre-triangulation on $\cc_{GF}$ we
obtain that $\left( K_{GF}\right) _{e}\left( \Delta \right) $ is
distinguished in $\cc_{GF}$. Thus $\left( K_{GF}\right) _{e}$ is
exact.
\end{proof}

\end{document}